\documentclass[english,10pt]{amsart}
\usepackage{lmodern}
\usepackage[T1]{fontenc}
\usepackage[latin9]{inputenc}
\usepackage[centering]{geometry}

\usepackage{adjustbox}

\usepackage{extpfeil}
\usepackage[usenames,dvipsnames,svgnames,table]{xcolor}
\usepackage{mathtools}
\usepackage{amsmath}
\usepackage{amsthm}
\usepackage{amssymb}
\usepackage{mathrsfs}
\usepackage{nicefrac}
\usepackage{graphicx}
\usepackage{caption}
\usepackage{float}
\usepackage{comment}
\usepackage{tikz-cd}

\usepackage{thmtools}
\usepackage{thm-restate}

\usepackage[dvipsnames,svgnames,table]{xcolor}
\usepackage[all]{xy}
\usepackage{hyperref}
\hypersetup{bookmarksnumbered,colorlinks=true}
\hypersetup{
     colorlinks   = true,
     linkcolor	  = DarkGreen,
     citecolor    = Purple
}

\newcommand{\RomanNumeralCaps}[1]
    {\MakeUppercase{\romannumeral #1}}
\newcommand*{\rom}[1]

\newcommand{\sph}{\mathbb{S}}
\newcommand{\sphwf}{\mathbb{S}_{W(\mathbb{F)}}}

\DeclareMathOperator{\kernel}{ker}
\DeclareMathOperator{\coker}{coker}
\DeclareMathOperator{\im}{im}

\newcommand{\bS}{\mathbb{S}}
\newcommand{\rel}{\textup{rel}}

\DeclareMathOperator{\fib}{fib}

\newcommand{\highlightgabe}[1]{\ifmmode{\text{\sethlcolor{SeaGreen}\hl{$#1$}}}\else{\sethlcolor{SeaGreen}\hl{#1}}\fi}
 
\newcommand{\field}{\mathbb{F}}

\usepackage{dsfont}

\makeatletter

\def\makeautorefname#1#2{\expandafter\def\csname#1autorefname\endcsname{#2}}
\makeautorefname{equation}{Equation}%
\def\equationautorefname~#1\null{(#1)\null}

\makeautorefname{footnote}{footnote}%
\makeautorefname{figure}{Figure}%
\makeautorefname{table}{Table}%
\makeautorefname{part}{Part}%
\makeautorefname{appendix}{Appendix}%
\makeautorefname{chapter}{Chapter}%
\makeautorefname{section}{Section}%
\makeautorefname{subsection}{Section}%
\makeautorefname{subsubsection}{Section}%
\makeautorefname{theorem}{Theorem}%
\makeautorefname{thm}{Theorem}%
\makeautorefname{theo}{Theorem}%
\makeautorefname{thmx}{Theorem}%
\makeautorefname{cor}{Corollary}%
\makeautorefname{coro}{Corollary}%
\makeautorefname{lem}{Lemma}%
\makeautorefname{lemm}{Lemma}%
\makeautorefname{prop}{Proposition}%
\makeautorefname{conj}{Conjecture}%
\makeautorefname{conv}{Convention}%
\makeautorefname{defn}{Definition}%
\makeautorefname{defi}{Definition}%
\makeautorefname{notn}{Notation}
\makeautorefname{nota}{Notation}
\makeautorefname{rem}{Remark}%
\makeautorefname{rema}{Remark}%
\makeautorefname{term}{Terminology}%
\makeautorefname{exm}{Example}%
\makeautorefname{exam}{Example}%
\makeautorefname{ax}{Axiom}%
\makeautorefname{claim}{Claim}%
\makeautorefname{ass}{Assumption}%
\makeautorefname{asses}{Assumptions}%
\makeautorefname{con}{Construction}%
\makeautorefname{cons}{Construction}%
\makeautorefname{prob}{Problem}%
\makeautorefname{warn}{Warning}%
\makeautorefname{obs}{Observation}%
\makeautorefname{quest}{Question}%
\newtheorem{thm}{Theorem}[section]

\newtheorem{cor}{Corollary}[section]
\newtheorem{coro}{Corollary}[section]
\newtheorem{prop}{Proposition}[section]
\newtheorem{lem}{Lemma}[section]
\newtheorem{lemm}{Lemma}[section]

\theoremstyle{definition}
\newtheorem{defn}{Definition}[section]
\newtheorem{defi}{Definition}[section]

\newtheorem{con}{Construction}[section]
\newtheorem{cons}{Construction}[section]

\newtheorem{exam}{Example}[section]

\newtheorem{nota}{Notation}[section]
\newtheorem{rem}{Remark}[section]
\newtheorem{rema}{Remark}[section]
\newtheorem{term}{Terminology}[section]

\makeatletter
\let\c@obs=\c@thm
\let\c@cor=\c@thm
\let\c@coro=\c@thm
\let\c@prop=\c@thm
\let\c@lem=\c@thm
\let\c@lemm=\c@thm
\let\c@prob=\c@thm
\let\c@con=\c@thm
\let\c@cons=\c@thm
\let\c@conv=\c@thm
\let\c@conj=\c@thm
\let\c@defn=\c@thm
\let\c@defi=\c@thm
\let\c@notn=\c@thm
\let\c@nota=\c@thm
\let\c@exm=\c@thm
\let\c@exam=\c@thm
\let\c@ax=\c@thm
\let\c@ass=\c@thm
\let\c@warn=\c@thm
\let\c@rem=\c@thm
\let\c@rema=\c@thm
\let\c@sch=\c@thm
\let\c@term=\c@thm
\let\c@equation\c@thm
\numberwithin{equation}{section}
\makeatother

\DeclareMathOperator{\equalizer}{eq}    

\newcommand{\mot}{\textup{mot}}
\DeclareMathOperator{\TC}{TC}
\newcommand{\ntc}{\mathrm{TC}^{-}}
\DeclareMathOperator{\tp}{\ensuremath{\textup{TP}}}
\DeclareMathOperator{\tc}{\ensuremath{\textup{TC}}}
\DeclareMathOperator{\thh}{\ensuremath{\textup{THH}}}
\newcommand{\THH}{\thh}
\newcommand{\rntc}{\mathrm{TC}_{\rel}^{-}}
\newcommand{\rtp}{\mathrm{TP}_{\rel}}
\newcommand{\rtc}{\mathrm{TC}_{\rel}}
\newcommand{\rthh}{\mathrm{THH}_{\rel}}
\newcommand{\rtcpl}{\mathrm{TC}_{\rel}^+}
\newcommand{\nc}{\newcommand}
\nc{\ot}{\otimes}
\nc{\otuz}{\otimes^{\mathbb{U}}_{\z}}
\nc{\otsph}{\otimes_{\mathbb{S}}}
\nc{\I}{\mathbb{I}}
\nc{\f}{\mathbb{F}}
\nc{\fp}{\mathbb{F}_p}
\nc{\fpl}{\mathbb{F}_{p^l}}
\nc{\fplpo}{\mathbb{F}_{p^{l+1}}}
\nc{\fpc}{\overline{\mathbb{F}}_p}

\nc{\zpl}{\mathbb{Z}_{(p)}}
\nc{\fq}{\mathbb{F}_q}
\nc{\CP}{\mathbb{CP}}
\nc{\wfq}{W(\mathbb{F}_q)}
\nc{\wfpn}{W(\mathbb{F}_{p^n})}
\nc{\sphtriv}{\mathbb{S}^{\on{triv}}}
\nc{\sphp}{\mathbb{S}_{p}}
\nc{\sphpl}{\mathbb{S}_{(p)}}
\nc{\sphwq}{\mathbb{S}_{W(\mathbb{F}_q)}}
\nc{\sphwpn}{\mathbb{S}_{W(\mathbb{F}_{p^n})}}
\nc{\sphwk}{\mathbb{S}_{W(\field)}}
\nc{\id}{{\on{id}}}
\nc\spec{{\on{Spec}}}
\nc\End{{\on{End}}}
\nc{\ev}{\on{ev}}
\DeclareMathOperator*{\colim}{colim}
\nc{\alg}{\on{Alg}}
\nc{\lmod}{\on{LMod}}
\nc{\red}{\color{red}}
\nc{\green}{\color{orange}}
\nc{\blue}{\color{blue}}
\nc{\Top}{\on{Top}}
\nc{\Map}{\on{Map}}

\newcommand{\mmod}{/\!\!/}
\newcommand{\sslash}{\mmod}

\DeclareMathOperator{\cycsp}{\ensuremath{\textup{CycSp}}}
\newcommand{\filorb}{\mathrm{fil}^{*}_{\textup{orb}}}
\newcommand{\filorbo}{\ensuremath{\textup{fil}}^0_{\textup{orb}}}
\newcommand{\filorbi}{\ensuremath{\mathrm{fil}}^i_{\textup{orb}}}

\newcommand{\grorbast}{\mathrm{gr}_{\textup{orb}}^*}
\newcommand{\fil}{\mathrm{fil}}
\newcommand{\filorbast}{\ensuremath{\textup{fil}}^*_{\textup{orb}}}
\newcommand{\filorbastplusone}{\ensuremath{\textup{fil}}^{*+1}_{\textup{orb}}}
\newcommand{\filorbj}{\ensuremath{\textup{fil}}^j_{\textup{orb}}}
\newcommand{\filorbjplusone}{\ensuremath{\textup{fil}}^{j+1}_{\textup{orb}}}
\newcommand{\filorbn}{\ensuremath{\textup{fil}}^n_{\textup{orb}}}

\DeclareMathOperator{\can}{\ensuremath{\textup{can}}}
\nc{\C}{\cat C}
\nc{\D}{\cat D}
\nc{\V}{\cat V}

\nc{\bE}{\mathbb{E}}
\nc{\An}{\mathrm{An}}
\nc{\Grp}{\mathrm{Grp}}
\nc{\Mon}{\mathrm{Mon}}
\nc{\Sp}{\mathrm{Sp}}
\nc{\Ztr}{\mathfrak{Z}}
\nc{\Fun}{\mathrm{Fun}}

\nc{\Cobar}{\mathrm{Cobar}}
\nc{\Barc}{\mathrm{Bar}}
\nc{\Ab}{\mathrm{Ab}}
\nc{\Q}{\mathbb{Q}}
\nc{\MU}{\mathrm{MU}}

\newcommand{\cC}{\mathcal{C}}
\newcommand{\Fil}{\mathrm{Fil}}
\newcommand{\Tow}{\mathrm{Tow}}
\newcommand{\Alg}{\mathrm{Alg}}
\newcommand{\TP}{\mathrm{TP}}
\newcommand{\grmot}{\mathrm{gr}_{\textup{mot}}^*}

\DeclareMathOperator{\tcs}
    
    \DeclareMathOperator{\hh}{\ensuremath{\textup{HH}}}

        \newcommand{\z}{\mathbb{Z}}
                \newcommand{\N}{\mathbb{N}}
       \newcommand{\zp}{\mathbb{Z}_p}
    \newcommand{\wtw}{\tau_{\geq \bullet}}
    \newcommand{\pis}{\pi_*}

    \def\co{\colon\thinspace}

\newcommand{\lv}{\lvert}
\newcommand{\rv}{\rvert}

\DeclarePairedDelimiter\floor{\lfloor}{\rfloor}

\newcommand{\gr}{\textup{gr}^*}
\newcommand{\grzero}{\textup{gr}^0}
\newcommand{\Gr}{\textup{Gr}}

\makeatother

\date{}

\author{Gabriel Angelini-Knoll}
\address{Department of Mathematics, Applied Mathematics, and Statistics, Case
Western Reserve University, Cleveland, OH, USA}
\email{gja39@case.edu}

\author{Haldun {\"O}zg{\"u}r  Bay{\i}nd{\i}r}
\address{Departament de Matem\`atiques i Inform\`atica, Universitat de Barcelona, Barcelona, Spain}
\email{ozgurbayindir@gmail.com}

\begin{document}

\title{On the integral algebraic K-theory of Morava K-theory}

\begin{abstract}
We compute the cardinalities of the integral algebraic K-theory groups of connective Morava K-theory $k(n)$ in all degrees that are not congruent to $0$ or $1$ modulo $2p-2$. After base-changing the coefficients of $k(n)$ to the algebraic closure $\fpc$, we determine the corresponding cardinalities in all degrees and show that the groups vanish in even degrees. Our approach uses  what we call the orbit filtration on topological cyclic homology arising from the May filtration on topological Hochschild homology. Combining this with an analysis of the topological cyclic homology of formal DGAs of the form  $\mathbb{F}[x_{2m}]$, where $\mathbb{F}$ is a perfect field of characteristic $p$, we prove strong cardinality results for the topological cyclic homology of  $\mathbb{E}_1$-rings with homotopy $\mathbb{F}[x_{2m}]$. 
For  Morava K-theories over finite fields, we further use the motivic filtration of Hahn--Raksit--Wilson.
As an application of our methods, we compute the cardinalities of the algebraic K-theory groups of the truncation $W(\fpc)/p^n$ of the $p$-typical Witt vectors of $\fpc$; in particular, they vanish in even degrees. 
\end{abstract}

\maketitle
\setcounter{tocdepth}{1}
\tableofcontents{}

\section{Introduction}
In algebra, the prime division rings are the finite fields $\mathbb{F}_p$ and the rational numbers. In homotopy theory, there are more prime division rings known as Morava K-theory $K(n)$ depending on a height $n$ and a prime $p$. These are multiplicative cohomology theories with coefficients $K(n)_*=\mathbb{F}_p[v_n^{\pm 1}]$ when $0<n<\infty$, where $v_n$ has degree $2p^n-2$ exhibiting a higher height analogue of Bott periodicity. These periodic classes have been related to periodic phenomena in the homotopy groups of spheres~\cite{MRW77}, which has been investigated in the field of chromatic homotopy theory and connected to arithmetic information such as special values of the Riemann zeta function~\cite{Ada66} and integral modular forms~\cite{Beh09}.\footnote{See \cite[pp.~2658--2661]{AK23} for a more detailed discussion of these facts and the relation to algebraic K-theory.}

Algebraic K-theory contains deep arithmetic and geometric information; for example, it can be used to recover special values of Dedekind zeta functions~\cite{Lic73,Voe11} and study diffeomorphism groups of manifolds~\cite{WJR13}. We consider the algebraic K-theory of division rings and their chromatic analogues such as Morava K-theories to be of fundamental importance  because algebraic K-theory can often be reduced to understanding the algebraic K-theory of fraction fields and residue fields using localization sequences.

The computation of the algebraic K-theory groups of $\fp$ is due to Quillen~\cite{Qui72} and historically, it was the first higher algebraic K-theory computation. The rational algebraic K-theory of $\mathbb{Q}$ is computed by Borel  \cite{borel1972cohomologiereele,borel1974stablerealcohomologyofarithmeticgroups} and the torsion in the algebraic K-theory of $\mathbb{Q}$ is closely related to special values of the Riemann zeta function by the resolution of the Beilinson--Lichtenbaum conjecture~\cite{Voe11}. For the higher chromatic analogue of the prime fields $K(n)$, one considers the algebraic K-theory after smashing with a finite complex, for example Ausoni and Rognes computed $\pis \big(\textup{K}(K(1))/(p,v_1)\big)$ for $p \geq 5$ in \cite{ausoni2012kthrymoravaktheory} and the first author, Hahn and Wilson extended this computation to  $p\ge 3$~\cite{angelini2024syntomic}. These computations have played an important role in providing evidence for the redshift conjectures of Ausoni--Rognes~\cite{AR08}. 

Beyond height $n=1$, no complete computations of the algebraic K-theory of Morava K-theory have been available even with finite coefficients. However, the first author, Hahn and Wilson \cite{angelini2024syntomic} computed the mod $(p,\cdots,v_{n+1})$-syntomic cohomology of connective Morava K-theory $k(n)$ at all primes and heights, which approximates algebraic K-theory. More precisely, syntomic cohomology is the associated graded of the motivic filtration on the topological cyclic homology of $k(n)$ constructed in~\cite{hahn2022motivic} and topological cyclic homology is a close approximation to algebraic K-theory by the Dundas--Goodwillie--McCarthy theorem~\cite{DGM13}. 
This strategy allowed the first author, Hahn and Wilson to prove redshift for Morava K-theory~\cite{angelini2024syntomic}. 

In this work, we compute the cardinalities of the integral algebraic K-theory groups of $k(n)$ in infinitely many degrees for $p>2$. Our result is the first of its kind because computations of orders of infinite families of integral algebraic K-theory groups of ring spectra that are not Eilenberg--MacLane have not appeared in the literature previously. The only results of this kind that the authors are aware of are the impressive computations of the integral algebraic K-theory groups of the sphere spectrum in low degrees by Rognes~\cite{Rog03} and Blumberg--Mandell~\cite{BM19}.

In fact, we also compute the quotients $\lvert \mathrm{K}_{2k+1}(k(n))\rvert/\lvert \mathrm{K}_{2k}(k(n))\rvert$ for all integers $k$. 
Quotients of orders of algebraic K-theory groups may be desirable from an arithmetic perspective, for example Lichtenbaum's conjecture~\cite[Conjecture~2.4]{Lic73} states that when $F$ is a totally real number field 
\[ \zeta_{F}(1-2k) = (-1)^{kr_1}2^{r_1}\frac{\lvert \mathrm{K}_{4k-2}(\mathcal{O}_F)\rvert }{\lvert \mathrm{K}_{4k-1}(\mathcal{O}_F)\rvert }
\]
for $k\geq 1$ where $F$ has $r_1$ real places (this has been prove when  $\mathrm{Gal}(F/\mathbb{Q})$ is abelian by~\cite{Voe11}, see~\cite[Theorem~8.8]{Wei13} for details).

Our approach to this computation begins with an extensive study of the algebraic K-theory of general connective ring spectra with polynomial homotopy groups $\field[x_{2m}]$; where for the rest of this work $\field$ denotes a perfect field of characteristic $p$ and $\lvert x_{2m} \rvert = 2m$. For $\field = \fpc$ the algebraic closure of $\fp$, we compute the cardinalities  of the relative  K-theory groups of such ring spectra for $p \nmid m$ and establish that they vanish in even degrees. To this end, we construct what we call the orbit filtration for the  topological cyclic homology of connective ring spectra that may be of independent interest. This filtration originates from the May filtration on topological Hochschild homology~\cite{AKS18} and its twisted cyclotomic structure from~\cite{antieau2020beilinson}. Working relatively, its associated graded is governed by the $S^1$-homotopy orbits of topological Hochschild homology. We combine this approach with a different filtration, the motivic filtration of Bhatt--Morrow--Scholze \cite{bhatt2019thhandintegralpadichodge} and Hahn--Raksit--Wilson~\cite{hahn2022motivic}, in our cardinality computations of algebraic K-theory of $k(n)$. 

Computations of algebraic K-theory of truncations $W(\field)/p^n$ for $\field$ a finite field over $\fp$ have been an important recent breakthrough in the field~\cite{KS24,AKN24}. In this work, we extend our methods to the $p$-adic filtration on the truncation $W(\field)/p^n$ of the $p$-typical Witt vectors when $\field$ is an algebraically closed field of characteristic $p$. This allows us to completely compute the cardinality of the algebraic K-theory groups of $W(\fpc)/p^n$. In particular, we prove that these groups vanish in even degrees. 

The motivic pullback squares of Land--Tamme \cite{land2023kthrypushouts} relate various relative algebraic K-theory groups for discrete rings to those for certain $\z$-algebras with homotopy $\field[x_{2}]$. In applications, we combine our results with these motivic pullback squares to compute the cardinalities of relative algebraic K-theory groups for some discrete rings including the Burnside ring of the cyclic group of order $p$ over $W(\fpc)$. This could be of use for studying the $C_p$-equivariant algebraic K-theory, in the sense of~\cite{BGS20}, of the Burnside Mackey functor.\footnote{Writing $\underline{\mathrm{K}}$ for the spectral Mackey functor version of algebraic K-theory constructed in~\cite{BGS20}, $\underline{A}^{C_p}$ for the Burnside Mackey functor, and $A(C_p)$ for the Burnside ring, then $\underline{\mathrm{K}}(\underline{A}^{C_p})(C_p/C_p)=\mathrm{K}(A(C_p))$ in light of~\cite[Example~A.7]{Yan25}.}

\subsection{Results and methods}

In this work, we study algebraic K-theory through trace methods meaning we consider related invariants such as topological cyclic homology $\tc$, topological negative cyclic homology $\ntc$ and  topological periodic cyclic homology $\tp$ of $\bE_1$-rings (i.e.\ ring spectra). Each of these related invariants is constructed from the cyclotomic structure on topological Hochschild homology $\thh$. 

Given a functor $F$ from connective $\mathbb{E}_1$-rings to spectra, in this introduction we write \[F_{\rel}(A):=\fib(F(A)\to F(\pi_0A))\]
for a given connective $\bE_1$-ring $A$ with $\pis A\cong \field[x_{2m}]$, see~\autoref{nota relative versions for filtered rings} for a more general notation used in the body of the paper.

Let $\mathbb{S}_{W(\field)}$ denote the spherical Witt vectors of~\cite[Example 5.2.7]{lurie2018ellipticII} (cf.~\cite[\S~2.1]{BSY22}). Our first main theorem is as follows.
\begin{restatable}[Even vanishing of relative K-theory]{thmx}{mainthmevenvanishing}\label{mainthmevenvanish}
Let $\field$ be an algebraically closed field of characteristic $p$ and let $A$ be an $\bE_1$-ring with $\pis A = \fp[x_{2m}]$ where $p \nmid m>0$. Then 
\begin{enumerate}
\item\label{item 1 thmevenvanishing} $\pi_{2k}\mathrm{K}_{\rel}(\sphwf \otimes A)= 0$ for all $k$.
\item\label{item 2 thmevenvanishing}  $\pi_{2k+1}\mathrm{K}_{\rel}(\sphwf \otimes  A)$ is trivial for $k< m$ and infinite for $k\geq m$ 
 where it is a $p$-group with bounded $p$-torsion.
\end{enumerate}
\end{restatable}

An important example of an $\bE_1$-ring satisfying the hypothesis of \autoref{mainthmevenvanish} is  what we call the connective Morava K-theory with $\field$-coefficients defined by
\[k(n)_{\field}:= \sphwf \otimes k(n)\]
where $k(n)$ denotes the connective Morava K-theory spectrum equipped with an arbitrary $\bE_1$-ring structure; there is a ring isomorphism  $\pis (k(n)_\field) \cong \field[v_n]$ where $\lvert v_n\rvert=2p^n-2$. For the $2$-periodic analogue, we consider the quotient $E(\fp,\Gamma_n,)/\mathfrak{m}$ of the Lubin-Tate spectrum on a height $n$ formal group law $\Gamma_n$ over $\fp$ with a regular sequence generating the maximal ideal $\mathfrak{m}$ of $\pi_0E(\fp,\Gamma_n)$ and then consider the base-change $K(\field,\Gamma_n)=\mathbb{S}_{W(\field)}\otimes E(\fp,\Gamma_n)/\mathfrak{m}$ and its connective cover $k(\field,\Gamma_n)=\tau_{\ge 0}K(\field,\Gamma_n)$.

\begin{restatable}{corox}{coro evenvanishing}
 Let $\field$ be  an algebraically closed field of characteristic $p$, then the conclusion of \autoref{mainthmevenvanish} holds for the following $\mathbb{E}_1$-rings:
\begin{enumerate}
\item connective Morava K-theory $k(n)_{\field}$. 
\item $2$-periodic connective Morava K-theory $k(\field,\Gamma_n)$. 
\item $W(\field)\sslash p$, the $\bE_1$-quotient of $ W(\field)$ by $p \in W(\field)$ as in \eqref{eq defining mod mod p guy}.
\end{enumerate}
\end{restatable}

In the special case of \autoref{mainthmevenvanish} where $\field=\fpc$, we completely compute the cardinalities of the relative algebraic K-theory of $\mathbb{S}_{W(\fpc)}\otimes A$ where $A$ is an $\mathbb{E}_1$-ring with homotopy ring $\pi_*A=\field_p[x_{2m}]$ and $p\nmid m>0$, see \autoref{coro absolute k theory is odd}. For example, we completely compute the cardinality of the algebraic K-theory groups of Morava K-theory over $\bar{\mathbb{F}}_p$. 

\begin{restatable}{corox}{cardoverFpbar}\label{overFpbar}
The groups $\mathrm{K}_{2k}(k(n)_{\bar{\field}_p})$ vanish and the groups $\mathrm{K}_{0}(k(n)_{\bar{\field}_p})$ and $\mathrm{K}_{2k-1}(k(n)_{\bar{\field}_p})$ are countably infinite with bounded $p$-torsion for integers $k\geq 1$. 
\end{restatable}

For the proof of \autoref{mainthmevenvanish}, we begin by proving the following oddness result for $\rntc$ and $\rtp$ that holds without the assumption that $\field$ is algebraically closed.

\begin{restatable}[Even vanishing of relative $\ntc$ and $\tp$]{thmx}{oddnessofntcandtp}\label{oddness}
Let $\field$ be a perfect field of characteristic $p$ and let $A$ be an $\mathbb{E}_1$-ring with  $\pis A \cong \mathbb{F}[x_{2m}]$ where $p\nmid m>0$. Then  \[\pis \rntc(A) \textup{\ \ and\  \ } \pis \rtp(A)\] 
are concentrated in odd degrees.
\end{restatable}

\autoref{oddness} follows from 
the computations of $\TC^{-}$ and $\TP$ of the free $\mathbb{E}_1$ $\field$-algebra $\mathbb{F}[x_{2m}]$ on $\sph^{2m}$  from~\autoref{TC of formal DGAs} and 
the existence of a May spectral sequence for $\TC^{-}$ and $\TP$ which we construct in \autoref{Whitehead}, extending~\cite{AKS18,Kee25}. 
The computation of $\tc(\field[x_{2m}])$, $\TC^{-}(\field[x_{2m}])$ and $\TP(\field[x_{2m}])$ for $m=1$ is due to \cite{bayindir2020kthryofthh} and independently due to Land--Tamme \cite{land2023kthrypushouts}. 
Building on this, we provide the computation for $p\nmid m>0$ using the root adjunction  methods of \cite{ausoni2022adjroot}  in \autoref{TC of formal DGAs} and this suffices for our applications. This computation was already known to Land, Tamme and Speirs  and it will appear in their upcoming work joint with the second author.
We expect that this upcoming work will extend our results (with the exception of \autoref{prop root adjunction in thh}) to arbitrary integers $m>0$; removing the assumption $p \nmid m$ (see \autoref{rema on p nmid m generalization}).

One of the main theoretical inputs for our work is what we call the orbit filtration on $\rtc(A)$, which we produce as a consequence of the May filtrations on $\ntc$ and $\tp$ and the twisted cyclotomic structure from~\cite{antieau2020beilinson} in \autoref{sec:orbit-filt}. Similar filtrations have been considered previously by Brun~\cite{Bru01} and Angeltveit~\cite{Angeltveit2015kTheoryfinitewitt} using the more classical approach to topological cyclic homology. 
For instance, in light of \autoref{oddness}, the proof of \autoref{mainthmevenvanish} \eqref{item 1 thmevenvanishing} reduces to showing that $\varphi -\mathrm{can}$ is surjective when $\field$ is algebraically closed. 
A first obstacle is that both $\pi_{2i+1} \rntc(A)$ and $\pi_{2i+1} \rtp(A)$ are uncountably infinite in this situation.  A variant of the orbit filtration discussed in \autoref{sec partial orbit filtration for tcrel} allows us to define quotients of $\rntc(A)$ and $\rtp(A)$ that recover $\rtc(A)$ in sufficiently low degrees while having finite homotopy groups when $\pis A \cong \fp[x_{2m}]$. Moreover, the last statement of \autoref{mainthmevenvanish} \eqref{item 2 thmevenvanishing} is a consequence of the spectral sequence corresponding to the orbit filtration, whose $\mathrm{E}_1$-page is given by $\pis \big(\Sigma \rthh(\pis A)_{hS^1}\big)$. 

Let $\mathbb{W}_{s}(\field)$ denote the big Witt vectors of length $s$, which is trivial when $s\le 0$ by convention. For the general case of an $\bE_1$-ring with homotopy $\mathbb{F}_{p^k}[x_{2m}]$, the aforementioned quotients of $\ntc$ and $\tp$ allow for counting arguments. Through this, we prove the following analogous result.

\begin{restatable}[Quotients of cardinalities of $\mathrm{TC}$]{thmx}{theoeventrivialgivesodd}
\label{theo even trivial gives odd}
Let $\field_q$ be a finite field of characteristic $p$ and let $A$ be a connective $\bE_1$-ring such that $\pis A \cong \field_q[x_{2m}]$ where $p\nmid m>0$. Then $\lvert \pi_s \rtc(A)\rvert$ is finite for each integer $s$ and 
\[\frac{\lvert \pi_{2r+1}\rtc(A)\rvert}{\lvert \pi_{2r} \rtc(A)\rvert} = \lvert \mathbb{W}_{\floor*{\frac{r}{m}}}(\field_q) \rvert \]
for each integer $r$. 
\end{restatable}

See \cite[Theorem B]{Angeltveit2015kTheoryfinitewitt} for a similar result on the topological cyclic homology of $W(\fp)/p^n$ that relies on the oddness of  a relative version of $\textup{TF}(\fp[t]/t^n)$, see also~\cite[Proposition~1.5]{AKN24}. 

 For the cardinalities of the integral algebraic K-theory groups of $k(n)_{\field_q}$ for $\field_q$ a finite field of characteristic $p$, we utilize the computation of the mod $(p,\cdots,v_{n+1})$-syntomic cohomology of $k(n)_{\field_q}$ from the work of the first author, Hahn and Wilson \cite{angelini2024syntomic}. Using \autoref{theo even trivial gives odd} and analyzing the successive $v_i$-Bockstein spectral sequences for the syntomic cohomology of $k(n)_{\field_q}$, we produce our next main result. 

\begin{restatable}[Cardinality of K-theory of connective Morava K-theory]{thmx}{mainthmcardinality}\label{cardinality-k-theory}
Suppose $\field_{q}$ is a finite field of characteristic $p$. There are isomorphisms 
\[ 
\mathrm{K}_{2r}(k(n)_{\field_{q}
})\cong \begin{cases} \mathbb{Z} & r=0  \,,\\ 
    0  &  p^{n+1}-1\nmid r\,, \ 0<r\leq2p^{n+1}-2  \,,\\
          \mathbb{Z}/p^{m(r)} &  p^{n+1}-1\mid r\,,\ 0<r\le 2p^{n+1}-2  \,,\\
    0 & p-1 \nmid r\,, \ r> 2p^{n+1}-2  \,,\\
     A_{r}  &  p-1\mid r\,,  \ r> 2p^{n+1}-2 \,,
     \end{cases}
\]
and we have identifications 
\[ 
|\mathrm{K}_{2r+1}(k(n)_{\field_q})|=\begin{cases}
 |\mathbb{W}_{\lfloor \frac{r}{p^{n}-1}\rfloor}(\field_q)|(q^{r+1}-1) & p^{n+1}-1\nmid r \,, \ 0 \le  r\le 2p^{n+1}-2 \,,\\
       |\mathbb{W}_{\lfloor \frac{r}{p^{n}-1}\rfloor}(\field_q)|(q^{r+1}-1)\cdot p^{m(r)} & p^{n+1}-1\mid r \,, \ 0 \le r\le 2p^{n+1}-2 \,,\\
         |\mathbb{W}_{\lfloor \frac{r}{p^{n}-1}\rfloor}(\field_q)|(q^{r+1}-1) & p-1\nmid r \,, \ r> 2p^{n+1}-2 \,,\\
    |\mathbb{W}_{\lfloor \frac{r}{p^{n}-1}\rfloor}(\field_q)|(q^{r+1}-1)\cdot \lvert A_r\rvert 
    & p-1\mid r \,, \ r>2p^{n+1}-2 \,,
\end{cases}
\]
of cardinalities where $\mathbb{W}_n$ denotes the truncated big Witt vectors of length $n$, $m(r)$ is a positive integer that we leave undetermined and $A_{r}$ is a finite $p$-group that we leave undetermined. 
\end{restatable}

In particular, this determines $\lvert \mathrm{K}_{t}(k(n)_{\field_q})\rvert$ for all $t$ that are  not congruent to $0$ or $1$ modulo $2p-2$. This also determines the cardinalities  $\lvert \mathrm{K}_{t}(k(n)_{\field_q}) \rvert$ for all $t<4p^{n+1}-4$ except for $t = 2p^{n+1}-2$ and $t= 2p^{n+1}-1$ where the cardinality is only determined up to a power of $p$.
In fact, we determine the quotients $\lvert K_{2r+1}(k(n)_{\field_q})\rvert/\lvert K_{2r}(k(n)_{\field_q})\rvert$ for all integers $r$. We have a similar computation in the case of periodic Morava K-theory, see \autoref{cardinality-k-theory-periodic}.  

In \autoref{K1local}, we also compute $\pi_*L_{K(1)}\mathrm{K}(K(1))$ and observe that it is concentrated in even degrees, as a consequence of~\cite{angeltveit2008thhofainftyringspectra,DR25}. This section is independent of the rest of the paper and it is only included for context.

\subsection{On the K-theory of truncated Witt vectors}\label{truncated Witt vectors}
Work of Krause--Senger~\cite{KS24} builds on recent breakthroughs of Antieau--Krause--Nikolaus~\cite{AKN24} and Hahn--Levy--Senger~\cite{HLS24} to determine the exact even vanishing of the algebraic K-theory of $\mathbb{Z}/p^n$; improving on the even vanishing bounds established in \cite[Theorem~1.4]{AKN24}. More generally, quotients of the form $W(\mathbb{F}_{p^k})/p^n$ were considered in work of Angeltveit when $\mathbb{F}_{p^k}$ is a finite field~\cite{Angeltveit2015kTheoryfinitewitt} and an even vanishing result was proven in Antieau--Krause--Nikolaus~\cite{AKN24} to more general quotients of complete DVRs with residue field $\mathbb{F}_{p^k}$. To our knowledge, our work is the first to consider the case where $\field$ is an algebraically closed field. For the following, let $\mathrm{K}(W(\field)/p^n,\field) := \textup{fib}(\mathrm{K}(W(\field)/p^n) \to \mathrm{K}(\field))$. 

\begin{restatable}[Even vanishing of K-theory of truncated Witt vectors]{thmx}{mainthmcardinalitywitt}\label{theo k-theory of truncated witts is odd}
Let $\field$ be an algebraically closed field of characteristic $p$ and let $n>1$. Then 
\[\pi_{2r} \mathrm{K}(W(\field)/p^n,\field) = 0\]
for all $r$. Moreover,
\[\pi_{2r+1}\mathrm{K}(W(\field)/p^n,\field)\]
is infinite for all $r\geq 0$.
\end{restatable}

The proof of this theorem is a direct adaptation of the proof of \autoref{mainthmevenvanish} and it is independent of the methods and results of \cite{KS24} and \cite{AKN24}. For this case, we use the $p$-adic filtration on $W(\field)/p^n$ replacing the Whitehead tower used in the proof of \autoref{mainthmevenvanish}. As the associated graded of this $p$-adic filtration is given by the truncated polynomial ring $\field[t]/t^n$, we rely on the computation of $\tc$, $\ntc$ and $\tp$ of $\field[t]/t^n$ from \cite{hesselholt1997kthryoftruncated} and \cite{speirs2020truncatedpolynomial}. Our approach bears some resemblance to the methods of Brun~\cite{Bru00,Bru01} and Angeltveit~\cite{Angeltveit2015kTheoryfinitewitt}, though we do not directly rely on any of their results. 

In the case of $\fpc$, we can completely compute the cardinalities. 

\begin{restatable}{corox}{maincorcardinalitywitt}\label{cor k-theory of truncated witts is odd}
Let $n> 1$. The groups 
\[\pi_{2r} \mathrm{K}(W(\fpc)/p^n,\fpc) = 0\]
for all $r$. Moreover, the groups 
\[\pi_{2r+1}\mathrm{K}(W(\fpc)/p^n,\fpc)\]
are countably infinite with bounded $p$-torsion for all $r\geq 0$.
\end{restatable}
In sufficiently large degrees, this also follows from~\cite[Theorem~1.4]{AKN24} since algebraic K-theory commutes with filtered colimits.

\subsection{Applications to motivic pullback squares}\label{subsec intro on motivic pullbacks}

In \autoref{sec applications to discrete rings}, we study the consequences of \autoref{mainthmevenvanish} regarding the algebraic $K$-theory of various discrete rings. 
For this, we use the motivic pullback squares constructed by Land and Tamme~\cite{land2023kthrypushouts}. 
Our first example is what Land and Tamme call the arithmetic coordinate axis, i.e.\ the homotopy pullback $W(\field) \times_{\field} W(\field)$ that is isomorphic to the discrete ring $W(\field)[u]/(u(u-p))$. 
Indeed, this is the base-change of the Burnside ring $A(C_p)$ of the cyclic group of order $p$ to $W(\field)$. There is a pullback square~\cite[Example~4.31]{land2023kthrypushouts}:
\begin{equation*}\label{diag motivic pullback square}
\begin{tikzcd}
\mathrm{K}(W(\field) \times_{\field} W(\field)) \ar[r] \ar[d] & \mathrm{K}( W(\field))  \ar[d] \\
\mathrm{K}(W(\field)) \ar[r] & \mathrm{K}(W(\field) \sslash p).
\end{tikzcd}
\end{equation*}
Here, $W(\field)\sslash p$ is a $\z$-algebra (i.e.\ a DGA) with homotopy $\field[t_2]$. Considering this pullback as a fiber sequence identifies:
\begin{equation}\label{eq relative kthry arithmetic coordinate axis}
\mathrm{K}\big(W(\field) \times_{\field}
W(\field), W(\field) \times  W(\field)\big) \simeq  \Omega \mathrm{K}(W(\field)\sslash p).
\end{equation}

For $p=2$, the $\bE_1$-ring $W(\field) \sslash 2$ is equivalent to the free $\bE_1$ $\field$-algebra on $\Sigma^2 \field$ (by \cite[Remark 4.33]{land2023kthrypushouts} or \cite[Theorem B]{bayindir2025towards}). Therefore, its relative algebraic K-theory is known due to \cite{bayindir2020kthryofthh} or \cite[Example 4.29]{land2023kthrypushouts}, which together with \cite{Qui72} computes \eqref{eq relative kthry arithmetic coordinate axis} for $\field$ a finite field of characteristic $2$ or the algebraic closure $\overline{\mathbb{F}}_2$.  

For odd $p$, $W(\field)\sslash p$ is not equivalent to the aforementioned free algebra (\autoref{rema nonformality}) hence the theorem of \cite{bayindir2020kthryofthh} does not apply. Indeed, Davis, Frank and Patchkoria \cite{davis2025cyclichomologyofzmodmodp} computed its  periodic cyclic homology $\textup{HP}^{\z}(W(\fp) \sslash p)$ for odd $p$ considering it as a step towards computing its K-theory; with the motivation coming from \eqref{eq relative kthry arithmetic coordinate axis} above. 

Here, we obtain the cardinalities of the relative algebraic K-theory groups in \eqref{eq relative kthry arithmetic coordinate axis} for $\field = \fpc$ at all primes (\autoref{coro relative k theory of arithmetic coordinate axis}) by applying \autoref{mainthmevenvanish} for  $\mathrm{K}_*(W(\fpc) \sslash p)$. To our knowledge, \autoref{mainthmevenvanish} provides the first K-theory computation in infinitely many degrees for a DGA that is not  equivalent to its homology as an $\bE_1$-ring (i.e.\ not topologically formal in the sense of \cite{bayindir2025towards, dugger2007topologicalequivalences}); c.f.\ \cite{schlicting2002noteontriangulatedcat,dugger2009curious}. We deduce that  
\[\mathrm{K}_s\big(W(\fpc) \times_{\fpc}
W(\fpc), W(\fpc) \times W(\fpc)\big)\]
is trivial for odd $s\geq 1$,  countably infinite with bounded $p$-torsion for even $s \geq 1$ and given by $\mathrm{K}_{s+1}(\fpc)$ for $s<1$. Indeed, this is a special case of our results \autoref{coro cardinalities of kthry for discrete rings} and \autoref{coro absolute k theory is odd}, which we prove in a similar manner. 

\subsection{Notation and Conventions}\label{not-conv} We work in the setting of $\infty$-categories in the sense of~\cite{lurie2016higher}. 
\begin{enumerate}
\item We fix a prime $p$ and write $\field$ for a perfect field of characteristic $p$, $W_n(\field)$ for the $n$-th $p$-typical Witt vectors of $\field$ and $\mathbb{W}_n(\field)$ for the big Witt vectors of $\field$ for the truncation set $\{1,\cdots ,n\}$, which is trivial by convention when $n\le 0$.  
\item We write $\mathrm{Sp}$ for the symmetric monoidal $\infty$-category of spectra with symmetric monoidal unit $\mathbb{S}$ and symmetric monoidal product $\otimes$. We write $B/A$ for the cofiber of a map $A\to B$.
\item Given a symmetric monoidal $\infty$-category $\cC$, we write $\mathrm{Alg}(\cC)$ for the $\infty$-category of $\mathbb{E}_1$ algebras in $\cC$. When $\cC=\Sp$, we call objects in $\mathrm{Alg}(\Sp)$ simply $\mathbb{E}_1$-rings. We write $\mathrm{CAlg}(\cC)$ for the $\infty$-category of $\mathbb{E}_{\infty}$-algebras in $\mathcal{C}$ and when $\mathcal{C}=\Sp$ we refer to them as $\mathbb{E}_\infty$-rings. 
\item Let $\mathbb{Z}^{\textup{op}}$ denote the $1$-category whose objects are integers such that $\mathrm{Map}(a,b)=*$ if $a\ge b$ and $\mathrm{Map}(a,b)=\emptyset$ if $a<b$. Let  $\mathbb{N}$ denote the full subcategory of $\z$ given by nonnegative integers.

 Let $\mathbb{Z}^{\textup{ds}}$ denote the discrete $1$-category whose objects are integers and whose maps are identity maps and let $\mathbb{N}^{\textup{ds}}$ denote the full subcategory of $\z^{\textup{ds}}$ given by the nonnegative integers.

\item Given an $\infty$-category $\cC$, let $\mathrm{Tow}(\cC):=\mathrm{Fun}(\mathbb{N},\cC)$ and $\mathrm{Fil}(\cC):=\mathrm{Fun}(\mathbb{Z},\cC)$. 
Given $I^{\bullet}\in\mathrm{Tow}(\cC)$ or $I^{\bullet}\in\mathrm{Fil}(\cC)$, we write $I^w:=I^{\bullet}(w)$ and refer to $w$ as the \emph{weight} and $I^w$ as the \emph{weight $w$ component}. Note that $\colim I^\bullet=I^0$ if $I^{\bullet}\in \mathrm{Tow}(\mathcal{C})$ so we say that the \emph{colimit is realized as $I^0$} in this case. 
We refer to the objects of $\mathrm{Tow}(\Sp)$ as $\mathbb{N}$-filtered spectra and objects in $I^{\bullet}\in\mathrm{Fil}(\cC)$ as $\mathbb{Z}$-filtered spectra. 

Similarly, let $\mathrm{Gr}(\cC)=\mathrm{Fun}(\mathbb{N}^{\textup{ds}},\cC)$. 
We refer to the objects of $\mathrm{Gr}(\cC)$ as graded objects or $\mathbb{N}$-graded objects of $\cC$. We write $\mathrm{Fun}(\mathbb{Z}^{\textup{ds}},\cC)$ and refer to objects in this category as $\mathbb{Z}$-graded objects in $\mathcal{C}$. In each case, we write $I^i:=I^*(i)$ and refer to $i$ as the \emph{weight} and $I^i$ as the \emph{weight $i$ component}. 

We write $\gr :\mathrm{Tow}(\cC) \to \mathrm{Gr}(\cC)$ and  $\gr :\mathrm{Fil}(\cC) \to \Fun(\mathbb{Z}^{\textup{ds}},\cC)$ for the associated graded functors. 

If $\cC$ is presentably symmetric monoidal, we equip $\mathrm{Tow}(\cC)$,  $\mathrm{Fil}(\cC)$, $\mathrm{Gr}(\cC)$ and $\Fun(\mathbb{Z}^{\textup{ds}},\cC)$ with the Day convolution symmetric monoidal structure~\cite{glasman2016dayconv}. Recall that $\gr$ is symmetric monoidal with respect to Day convolution. 

\item Let $k(n)$ denote a general $\bE_1$-$\bS$-algebra form of $k(n)$ in the sense of~\cite[Definition~A.1.1]{angelini2024syntomic}. 
We then write $k(n)_{\field}:=\bS_{W(\field)}\otimes k(n)$ where $ \bS_{W(\field)}$ is the associated spherical Witt vectors to $\field$, see~\cite[Example~5.2.67]{lurie2018ellipticII} and~\cite[\S~2.1]{BSY22}. We also write $\mathrm{MU}_{W(\field)}=\bS_{W(\field)}\otimes \mathrm{MU}$. 

\item We will refer to objects in $\Sp^{BS^1}$ as $S^1$-spectra and maps in $\Sp^{BS^1}$ as $S^1$-equivariant maps. More genuine notions of $S^1$-spectra will not appear in this paper.
\item We write $\cycsp$ for the $\infty$-category of $p$-cyclotomic spectra in the sense that an object of $\cycsp$ is given by a $S^1$-spectrum $E$ and an $S^1$-equivariant map $\varphi_p :E\to E^{tC_p}$. Henceforth, we will simply refer to objects in $\cycsp$ as cyclotomic spectra. 

We then define $\varphi := \varphi_p^{hS^1}$ and 
\[ 
\TC(E):=\equalizer(
\begin{tikzcd} 
E^{hS^1}
\arrow[r,shift right,"\textup{can}",swap] 
\arrow[r,shift left,"\varphi"]  & (E^{tC_p})^{hS^1}
\end{tikzcd}
)
\]
following~\cite{nikolausscholze2018topologicalcyclic}. When $R$ is an $\mathbb{E}_1$-ring and $E=\THH(R)$, we write $\mathrm{TC}(R):=\mathrm{TC}(\THH(R))$. This notion of topological cyclic homology agrees with the usual one when $\THH(R)$ is bounded below and $R$ is $p$-local, this justifies our notation. In  addition to the standard conventions $\mathrm{TP}=\mathrm{THH}^{tS^1}$ and $\mathrm{TC}^{-}=\mathrm{THH}^{hS^1}$, we write $\TC^{+}:=\THH_{hS^1}$.
\end{enumerate}

\begin{rema}\label{rema on p nmid m generalization}
Outside \autoref{TC of formal DGAs},  we only need the assumption $p \nmid m$ when we use \autoref{theo tc ntc tp of formal dgas}. We expect that the upcoming work of the second author, Land, Tamme and Speirs will generalize \autoref{theo tc ntc tp of formal dgas} to hold without the assumption $p \nmid m$. This generalization will also remove the assumption $p \nmid m$ from all our results that are stated outside of \autoref{TC of formal DGAs}. 
\end{rema}

\subsection{Acknowledgements}
The first author would like to thank Jeremy Hahn for discussions about the motivic filtration.  
The second author would like to thank Markus Land for discussions about the computation of the algebraic K-theory of the free $\field$-algebra $\field[x_{2m}]$ and Maxime Ramzi for suggesting an argument for extending our results to periodic Morava $K$-theories. The authors would like to thank the organizers of the IWoAT workshop in 2025 in Hangzhou, which is where this project first began. The second author was  supported by the Beatriu de Pin\'os Programme (2023 BP 00043) of the Ministry of Research and Universities of the Government of Catalonia and the grant  PID2024-155646NB-I00, Agencia Estatal de Investigaci\'on. The second author holds a  Ram\'on y Cajal grant of the State Research Agency (AEI) of Spain with reference number RYC2024-049560-I.

\section{On the TC of formal DGAs with polynomial homology}\label{TC of formal DGAs}
For this section, let $p \nmid m >0$.  In this section, let $\rthh(\field[x_{2m}])$ denote the fiber of the map $\thh(\field[x_{2m}]) \to \thh(\field)$. Then $\rtc(\field[x_{2m}])$, $\rntc(\field[x_{2m}])$ and $\rtp(\field[x_{2m}])$ are defined in the same way. 
Here, we adapt the root adjunction methods of \cite{ausoni2022adjroot} to obtain a computation of the groups $\pi_*\rtc(\field[x_{2m}])$ using the computation of  $\pi_*\rtc(\field[x_{2}])$ from~\cite{bayindir2020kthryofthh}. Together with the May filtrations on $\rntc$ and $\rtp$ from \autoref{sec:may-filtrations},  this will provide a proof of \autoref{oddness}. 
\begin{rema}
The results of \cite{ausoni2022adjroot}  can be directly used to show that the map $\rtc(\field[x_{2m}]) \to \rtc(\field[x_2])$ induced by the map of $\bE_1$ $\field$-algebras $\field[x_{2m}] \to \field[x_2]$ sending $x_{2m}$ to $x_2^m$ is the inclusion of a summand  whenever $p\nmid m>0$. However, this is not sufficient for us since we need to identify the homotopy groups of this summand.  
\end{rema}

\begin{rema}\label{rema z graded spectra}
Let $\mathcal{C}$ be a presentably symmetric monoidal $\infty$-category. The inclusion $\mathbb{N}^\textup{ds} \to \mathbb{Z}^{\textup{ds}}$ induces a symmetric monoidal left adjoint 
\begin{equation}\label{eq: adjoint N to Z graded}
\Gr(\mathcal{C}) \to \textup{Fun}(\mathbb{Z}^{\textup{ds}}, \mathcal{C}) \,,
\end{equation}
that simply adds zero negative weights~\cite[Corollary 3.8]{nikolaus2016stablemultyoneda} (cf.~\cite[Proposition~2.12]{Kee25}). In particular, this functor commutes with THH. Furthermore, this left Kan extension followed by restriction (along   $\mathbb{N}^\textup{ds} \to \mathbb{Z}^{\textup{ds}}$) is equivalent to the identity functor.
\end{rema}

In this section, we work in the setting of $\z$-graded spectra, following \autoref{rema z graded spectra}, only to be compatible with the setting of \cite{bayindir2020kthryofthh} and \cite{ausoni2022adjroot}. In our usual convention, when we consider $\field[x_{2m}]$ as a $\z$-graded or $\N$-graded $\bE_1$-ring, we place $x_{2m}$ in weight $2m$ as it is often the associated graded of a Whitehead filtration.  Again, only to be compatible with the conventions of  \cite{bayindir2020kthryofthh} and \cite{ausoni2022adjroot}, we also consider the case where $x_{2m}$ is in weight $m$; in this case, we write $\field[x_{2m,m}]$ for the resulting $\z$-graded $\bE_1$-ring. For example, the underlying $\mathbb{Z}$-graded $\field$-module of $\field[x_{2}]$ is
\[  \dots  \qquad  0 \qquad 0 \qquad \field  \qquad 0  \qquad  \field  \qquad 0  \qquad \dots 
\]
whereas the underlying $\mathbb{Z}$-graded $\field$-module of $\field[x_{2,1}]$ is 
\[  \dots \qquad 0 \qquad  0 \qquad \field  \qquad  \field   \qquad  \field \qquad \field \qquad \dots 
\]

We let $f$ denote the map 
\begin{equation}\label{inclusion map}
f\co \field[x_{2m,m}] \to \field[x_{2,1}]
\end{equation}
of $\bE_1$ $\field$-algebras in $\z$-graded spectra  (sending $x_{2m,m}$ to $x_{2,1}^m$). Note that this map exists since $\field[x_{2m,m}]$ is the free $\bE_1$ $\field$-algebra in $\z$-graded spectra on a single generator in degree $2m$ and weight $m$.

Left Kan extending along the functor $\z\to 0$, one obtains the underlying spectrum of a $\z$-graded spectrum; we denote this symmetric monoidal left adjoint functor by $D$ (see \cite[Section 2.2]{ausoni2022adjroot}). For a $\z$-graded spectrum $E^\bullet$, there is an equivalence $D(E^\bullet) \simeq \bigoplus_{i \in \z} E^i$.

Note that $\THH$ can be defined in any presentably symmetric monoidal  stable $\infty$-category, in particular there is a functor
\[ \thh : \Alg(\Fun(\mathbb{Z}^{\textup{ds}},\Sp))\to \Fun(\mathbb{Z}^{\textup{ds}},\Sp)^{BS^1} \,.\]
We define $\rthh(A)$ for a 
$\mathbb{Z}$-graded spectrum $A$ with $A^i\simeq 0$ for $i<0$ to be the fiber of the $\z$-graded map $\THH(A)\to \THH(A^0)$, see~ \cite[Lemma B.0.6 (\romannumeral 4)]{hahn2020redshift}. In this way, we consider $\rntc(A)$ and $\rtp(A)$ as $\z$-graded spectra in the canonical way, i.e.\ by applying $-^{hS^1}$ and $-^{tS^1}$ respectively levelwise to $\rthh(A)$.

We recall that the twisted Frobenius on $\rthh(A)$ is given by $S^1$-equivariant maps
\begin{equation}\label{eq graded frobenius maps}
\varphi_p :\rthh(A)^i \to (\rthh(A)^{tC_p})^{ip}=(\rthh(A)^{ip})^{tC_p}
\end{equation}
where $-^{tC_p}$ is applied levelwise as usual \cite[Appendix A]{antieau2020beilinson}; this is called a twisted cyclotomic structure. We write 
\[\varphi=\varphi_p^{hS^1} :
\rtc(A)^i \to ((\rthh(A)^{ip})^{tC_p})^{hS^1}=((\rthh(A)^{tC_p})^{hS^1})^{ip}
\] 
and note that when $\rthh(A)^i$ is bounded below for each $i$ and $p$-complete the target can be identified with $\tp(A)^{ip}:=(\THH(A)^{ip})^{tS^1}$. 
After left Kan extension along $\mathbb{Z}\to 0$, this equips $D\rthh(A)$ with the structure of a cyclotomic spectrum in a canonical way which agrees with the cyclotomic structure on $\rthh(DA)$ \cite[Example A.16]{antieau2020beilinson}.  
We omit the functor $D$ from our notation when it is clear from the context and when we write $\rtc(A)$, we mean $\rtc(DA)$. 

\begin{rema}\label{rema rm and lm}
Let $R_m$ denote the functor given by the restriction of grading along the map $\z\xrightarrow{\cdot m} \z$ that multiplies by $m$. Left Kan extension along $\cdot m$ gives an adjunction $L_m\dashv R_m$ with 
\[(L_m F)^{im} = F^{i} \textup{\ \ and \ } (L_mF)^{i}= 0 \textup{\ if $i\nmid m$},\]
for a $\z$-graded spectrum $F$, see \cite[Section 2.2]{ausoni2022adjroot}.
In particular, $(L_mR_m F)^i = F^i$ if $m \mid i$ and  $(L_mR_m F)^i \simeq 0$ otherwise. 
By direct observation, one sees that $L_mR_m\rthh(A)$ is also a twisted cyclotomic object where the counit map $L_mR_m\rthh(A) \to \rthh(A)$ respects this structure. 
\end{rema}
\begin{prop}\label{prop root adjunction in thh}
        Let $p \nmid m>0$ and use the notation above. 
\begin{enumerate}
\item \label{it:twisted-cyclotomic-factorization1}
The map of twisted cyclotomic objects in $\z$-graded  spectra $\rthh(f)$ induced by \eqref{inclusion map} factors as 
\[
\rthh(\field[x_{2m,m}]) \xrightarrow{\simeq} L_mR_m \rthh(\field[x_{2,1}]) \to \rthh(\field[x_{2,1}])
\]
where the first map above is an equivalence and the second map is the counit from Remark~\ref{rema rm and lm}.  

\item \label{it:twisted-cyclotomic-factorization2} Applying the functor $D$, the map $D\rthh(f)$ of cyclotomic spectra factors as:
\[D \rthh(\field[x_{2m,m}]) \xrightarrow{\simeq} \bigoplus_{ \substack{i \in \z \\ m \mid i}} \rthh(\field[x_{2,1}])^i \to D \rthh(\field[x_{2,1}]),\]
where the first map is an equivalence. 
\end{enumerate}
\end{prop}

\begin{proof}
Since $\field[x_{2m,m}]$ is concentrated in weights that are divisible by $m$, by direct observation, we deduce that the same holds for $\rthh(\field[x_{2m,m}])$. Then the factorization of $\rthh(f)$ through $L_mR_m\rthh(\field[x_{2m,m}])$ follows immediately, 
and this respects the twisted cyclotomic structures. Upon applying the functor $D$, this gives the second factorization in the proposition. Since $D$ reflects equivalences, it suffices to prove that the first map in the second factorization in the proposition is an equivalence.  

Note that $f = \field \otimes \overline{f}$ for the canonical map of $\bE_1$-algebras in $\z$-graded spectra $\overline{f}\co\sphpl[x_{2m,m}] \to \sphpl[x_{2,1}]$. Then   $D\rthh(f) \simeq \thh(\field)\otimes D(\thh(\overline{f})^{> 0})$ where $(-)^{>0}$ here is the functor that removes the components $\le 0$~\cite[Lemma ~B.0.6]{hahn2020redshift}. 
Then the result follows due to  \cite[Proposition 4.17]{ausoni2022adjroot}; to adapt to the conventions of \cite[Proposition 4.17]{ausoni2022adjroot}, one notes that the  symmetric monoidal left adjoint functor $D^m$ from \cite[Section 2.2]{ausoni2022adjroot} commutes with THH and that  $(D^m F)^0 = \bigoplus_{ \substack{i \in \z \\ m \mid i}} F^i $ for a $\z$-graded spectrum $F$.
\end{proof}

\begin{rema}\label{rema passage to witt vector coefficients for formal}
There is an equivalence of $\z$-graded $\bE_1$-algebras $\sph_{W(\field)} \otimes \fp[x_{2m}] \simeq \field[x_{2m}]$. This gives the first equivalence in the composite 
\[\thh(\field[x_{2m}]) \simeq \thh(\sph_{W(\field)}) \otimes \thh(\fp[x_{2m}]) \simeq \sph_{W(\field)} \otimes \thh(\fp[x_{2m}])\]
of equivalences of $S^1$-equivariant $\z$-graded spectra. 
Here, $\sph_{W(\field)}$ on the right hand side is given the trivial $S^1$-action. The second equivalence follows as $\thh(\sph_{W(\field)}) \to \sph_{W(\field)}$ exhibits $ \sph_{W(\field)}$ as the $p$-completion of $\thh(\sph_{W(\field)})$ (which can be checked after applying $-\otimes \fp$ which is classical, see  \autoref{rema spherical witts as cyclotomic base}) and the fact that both sides are $p$-complete (as they are $\fp$-modules). It follows that this also applies to the relative case to give an equivalence 
\begin{equation}\label{eq for relthh taking witt out in the formal dga}
\rthh(\field[x_{2m}]) \simeq \sph_{W(\field)} \otimes \rthh(\fp[x_{2m}]).
\end{equation}
This gives canonical assembly maps 
\[\sph_{W(\field)} \otimes \rntc(\fp[x_{2m}]) \to \rntc(\field[x_{2m}]) \textup{\ \ and \ } \sph_{W(\field)} \otimes \rtp(\fp[x_{2m}]) \to \rtp(\field[x_{2m}])\]
using the fact that the targets are $\sph_{W(\field)}^{hS^1}$-modules and hence  $\sph_{W(\field)}$-modules through the $\bE_\infty$-map $\sph_{W(\field)} \to \sph_{W(\field)}^{hS^1}$. We prove below that the two maps above are  equivalences when $p\nmid m>0$. 
\end{rema} 

\begin{thm}\label{theo tc ntc tp of formal dgas}
Suppose $m>0$ and $p\nmid m$ and let $\field$ be a perfect field of characteristic $p$. Then we have: 
\begin{enumerate}
 \item    \label{item 1 frobenius is iso}
 $\pis(\rntc(\field[x_{2m}])^i)$ and $\pis (\rtp(\field[x_{2m}])^i)$  are concentrated in odd degrees for each $i$.  If $\field$ is finite then $\pi_{2r+1}(\rntc(\field[x_{2m}])^i)$ and $\pi_{2r+1}(\rtp(\field[x_{2m}])^i)$ are also finite for each $r$ and $i$. 
 \item \label{item 2 frobenius is iso}  If $p \nmid i$, then $\rtp(\field[x_{2m}])^{2i}\simeq 0$.
 \item \label{item 2half frobenius is iso}
 The assembly maps (from \autoref{rema passage to witt vector coefficients for formal}):
\[\sph_{W(\field)} \otimes \rntc(\fp[x_{2m}]) \xrightarrow{\simeq} \rntc(\field[x_{2m}]) \textup{\ \ and \ } \sph_{W(\field)} \otimes \rtp(\fp[x_{2m}]) \xrightarrow{\simeq} \rtp(\field[x_{2m}])\]
 are equivalences. 
 \item\label{item 3 frobenius is iso} The twisted Frobenius map
 \[\pi_k \varphi \co \pi_k\rntc(\field[x_{2m}])^{2i} \to \pi_k\rtp(\field[x_{2m}])^{2ip} \]
 is an isomorphism for every $k \geq 2i+1$. 
 \item \label{item 4 frobenius is iso}  There are isomorphisms \[\pi_{2r+1}\rtc(\field[x_{2m}]) \cong \mathbb{W}_{\floor*{\frac{r}{m}}}(\field) \textup{\ and \ } \pi_{2r}\rtc(\field[x_{2m}]) = 0.\]
\end{enumerate}
\end{thm}
\begin{proof}
As $\field[x_{2m}]$ is a free $\bE_1$-algebra in $\z$-graded spectra, we observe that there is an equivalence of $\z$-graded $\bE_1$-algebras 
\begin{equation}\label{eq grading shift to be compatible with bm}
L_2\field[x_{2m,m}] \simeq \field[x_{2m}],
\end{equation}
see \autoref{rema rm and lm}.
Since $L_2$ is a symmetric monoidal left adjoint, it commutes with $\thh$ and also since limits and colimits are computed levelwise in $\z$-graded spectra, it also commutes with $\ntc$ and $\tp$; in particular, 
\begin{equation}\label{eq ntc and tp for root adjunction}\rntc(\field[x_{2m}])^{2i} \simeq \rntc(\field[x_{2m,m}])^i \textup{\ \ and \ } \rtp(\field[x_{2m}])^{2i} \simeq \rtp(\field[x_{2m,m}])^i \end{equation}
with $\rntc(\field[x_{2m}])^{2i+1}\simeq \rtp(\field[x_{2m}])^{2i+1}\simeq 0$. Again, since $-^{hS^1}$ and $-^{tS^1}$ are applied levelwise in the graded setting, we deduce from \autoref{prop root adjunction in thh} \eqref{it:twisted-cyclotomic-factorization1} that there are equivalences of $\z$-graded spectra 
\begin{equation}\label{eq root adjunction equiv on ntc and ntp}
\rntc(\field[x_{2m,m}]) \simeq L_mR_m \rntc(\field[x_{2,1}])
\textup{\ \ and \ } \rtp(\field[x_{2m,m}]) \simeq L_mR_m \rtp(\field[x_{2,1}]).
\end{equation}
With this in mind, \eqref{item 1 frobenius is iso} in the statement of the theorem follows from \eqref{eq ntc and tp for root adjunction} and \cite[Proposition 6.3]{bayindir2020kthryofthh}, which implies that $\pis \rntc(\field[x_{2,1}])^i$ and $\pis \rtp(\field[x_{2,1}])^i$ are concentrated in odd degrees and that if  $\field$ is finite then they are degreewise finite. 
This also proves \eqref{item 2 frobenius is iso} since \cite[Proposition 6.3]{bayindir2020kthryofthh} also implies that $\rtp(\field[x_{2,1}])^i\simeq 0$ whenever $p\nmid i$. This finishes the proof of  \eqref{item 1 frobenius is iso} and \eqref{item 2 frobenius is iso}. Similarly, \eqref{item 2half frobenius is iso} follows by \cite[Proposition 6.3]{bayindir2020kthryofthh}  using  \eqref{eq ntc and tp for root adjunction} and \eqref{eq root adjunction equiv on ntc and ntp}.

By \cite[Lemma A.14]{antieau2020beilinson}, $L_2$ not only commutes with $\thh$, but it also respects the twisted cyclotomic structure on $\thh$. Therefore, to prove \eqref{item 3 frobenius is iso}, it suffices to show that the twisted Frobenius map 
\begin{equation}\label{eq regrading on frobenius}
\pi_k \varphi\co \pi_k\rntc(\field[x_{2m,m}])^{i} \to \pi_k\rtp(\field[x_{2m,m}])^{pi}
\end{equation}
is an isomorphism for $k\geq 2i+1$ due to \eqref{eq grading shift to be compatible with bm}. Using \autoref{prop root adjunction in thh} \eqref{it:twisted-cyclotomic-factorization1}, we deduce that the equivalences in \eqref{eq root adjunction equiv on ntc and ntp} respect the Frobenius. Therefore, the statement on \eqref{eq regrading on frobenius} boils down to the same statement for $\field[x_{2,1}]$ which is the content of  \cite[Corollary 5.16]{bayindir2020kthryofthh}. This finishes the proof of \eqref{item 3 frobenius is iso}.

    To prove \eqref{item 4 frobenius is iso}, we begin by noting that there is an equivalence of cyclotomic spectra
     \begin{equation}\label{eq cyclotomic splitting of thh from root adjunction}
        D\rthh(\field[x_{2,1}]) \simeq \Big( \bigoplus_{\substack{i\in \z \\ m \mid i }} \rthh(\field[x_{2,1}])^{i} \Big) \prod  \Big(\bigoplus_{\substack{ i\in \z \\m \nmid i }} \rthh(\field[x_{2,1}])^{i} \Big) 
        \end{equation}
      as the Frobenius multiplies the weight by $p$ and $p \nmid m$; i.e.\ Frobenius respects the product splitting above. As a cyclotomic spectrum, the left hand factor is given by
      \begin{equation}\label{eq root adj thh inclusion}
      D \rthh(\field[x_{2m,m}]) \xrightarrow{\simeq} \bigoplus_{ \substack{i \in \z \\ m \mid i}} \rthh(\field[x_{2,1}])^i \,,
      \end{equation}
see~\autoref{prop root adjunction in thh} \eqref{it:twisted-cyclotomic-factorization2}.

      From \eqref{eq cyclotomic splitting of thh from root adjunction} and \eqref{eq root adj thh inclusion}, we deduce that  $\rtc(\field[x_{2m}])$ is a factor of $\rtc(\field[x_{2}])$\footnote{Recall that for a $\z$-graded $\bE_1$-ring $B$, we are writing $\rtc(B)$ for $\rtc(DB)$, in particular there is an equivalence $\rtc(\field[x_{2m,m}]) \simeq \rtc(\field[x_{2m}])$.}. Therefore, the statement $\pi_{2r} \rtc(\field[x_{2m}]) = 0$ in \eqref{item 4 frobenius is iso} follows by  \cite[Theorem 1.2]{bayindir2020kthryofthh} (cf.~\cite[Example~4.29]{land2023kthrypushouts} and~\cite{Hes07}). What remains is to explicitly identify the factor of $\pi_{2r+1}\rtc(\field[x_{2}])$ corresponding to $\pi_{2r+1}\rtc(\field[x_{2m}])$.
      
      The coproducts in \eqref{eq cyclotomic splitting of thh from root adjunction}  are also the products of the underlying $S^1$-spectra as connectivity strictly increases with weight; as a result, the authors of \cite{bayindir2020kthryofthh} observe that these products commute with $-^{hS^1}$ and $-^{tS^1}$ and  that $\rtc(\field[x_{2}])$ is given as the fiber of 
      \[\prod_{\substack{1 \leq m'\\ p \nmid m'}} \prod_{0 \leq v} (\rthh(\field[x_{2,1}])^{p^vm'})^{hS^1} \xrightarrow{\varphi - \can} \prod_{\substack{1 \leq m'\\ p \nmid m'}} \prod_{0 \leq v} (\rthh(\field[x_{2,1}])^{p^vm'})^{tS^1}\]
      where $\varphi$  multiplies the weight by $p$ whereas $\can$ respects the weight; see the discussion after the proof of Proposition 6.3 in \cite{bayindir2020kthryofthh}. This gives a splitting of $\rtc(\field[x_{2}])$ as 
      \[\rtc(\field[x_{2}]) \simeq \prod_{\substack{1 \leq m'\\ p \nmid m'}} \textup{fib}\big(  \prod_{0 \leq v} (\rthh(\field[x_{2,1}])^{p^vm'})^{hS^1} \xrightarrow{\varphi - \can}  \prod_{0 \leq v} (\rthh(\field[x_{2,1}])^{p^vm'})^{tS^1}\big).\]
  
        What we deduce from   \eqref{eq cyclotomic splitting of thh from root adjunction} and \eqref{eq root adj thh inclusion} is that $\rtc(\field[x_{2m}])$ is given by the contribution of the factors for which $m'$ is divisible by $m$. The homotopy group $\pi_{2r+1}$ of the factor corresponding to $m'$ is computed at the end of Section 6 in~\cite{bayindir2020kthryofthh} as $W_s(\field)$ where $s$  is the smallest nonnegative integer $v$ such that $r<p^vm'$. This gives the first isomorphism in the composite 
       \[\pi_{2r+1}\rtc(\field[x_{2m}])\cong \prod_{\substack{1 \leq m' \leq r \\ p\nmid m',\  m \mid m'}} W_s(\field)\cong  \prod_{\substack{1 \leq k \leq \floor*{r/m} \\ p\nmid k}} W_{s'}(\field) \cong \mathbb{W}_{\floor*{\frac{r}{m}}}(\field).\]
    Here, $s'$ is the smallest nonnegative integer $v'$ such that $r<p^{v'}mk$, i.e.\ $\floor*{r/m}<p^{v'}k$. The second isomorphism above is obtained by running the product over $k$ such that $1\leq mk\leq r$ to remove the condition $m \mid m'$. The last isomorphism is \cite[Example 1.11]{hesselholt2015bigderhamwitt}.    
\end{proof}

\begin{rema}
The result above was already known to Land, Tamme and Speirs and work of the second author, Land, Tamme and Speirs will give an independent proof of the result above and we expect that it will generalize it to allow the case $p\mid m$. 
\end{rema}

\begin{rema}\label{rema trivial in odd weights}
In addition to \autoref{theo tc ntc tp of formal dgas} \eqref{item 2 frobenius is iso}, the graded spectra $\rthh(\field[x_{2m}])$, $\rntc(\field[x_{2m}])$ and $\rtp(\field[x_{2m}])$ are trivial in odd weights. This follows since the same is true for $\field[x_{2m}]$, which implies it for $\field[x_{2m}]^{\otimes_{\sph} k}$ for each $k\geq 0$ and hence for $\rthh$, $\rtc$ and $\rtp$ as limits and colimits are computed levelwise in graded spectra.
\end{rema}

\section{The May filtration on \texorpdfstring{$\TC^{-}$}{topological negative cyclic homology} and \texorpdfstring{$\TP$}{topological periodic cyclic homology}}\label{Whitehead}
In this section, we construct a May spectral sequence abutting to topological negative cyclic homology and topological periodic cyclic homology associated to an $\bE_1$-algebra in filtered spectra. In the case of the Whitehead filtration, we use this to prove vanishing in even degrees for relative topological negative cyclic homology and topological periodic cyclic homology of $\bE_1$-rings whose homotopy ring is $\field[x_{2m}]$ where $x_{2m}$ has degree $2m$ and $p\nmid m>0$.

\subsection{May filtrations}\label{sec:may-filtrations}
As mentioned before, $\THH$ can be defined in any presentably symmetric monoidal $\infty$-category, we consider: 
\[ \THH : \Alg(\Tow(\Sp))  \to \Tow(\Sp)^{BS^1}\,, \qquad
\text{ and } \qquad
\THH : \Alg(\Gr(\Sp))  \to \Gr(\Sp)^{BS^1}\,, \]
see~\cite[Definition~4.6,~Proposition 4.8, Definition~5.8]{Kee25} or~\cite[Appendix A]{antieau2020beilinson} for details. In this way, $\ntc$, $\tp$ and $(\thh(-)^{tC_p})^{hS^1}$ upgrade to functors of the form $ \Alg(\Tow(\Sp))  \to \Tow(\Sp)$ and  $ \Alg(\Gr(\Sp))  \to \Gr(\Sp)$. 

Let 
\begin{equation}\label{monoidal-Whitehead}
\tau_{\ge \bullet} : \mathrm{Alg}(\mathrm{Sp})\longrightarrow  \mathrm{Alg}(\mathrm{Fil}(\mathrm{Sp})) 
\end{equation}
denote the multiplicative Whitehead filtration from~\cite[1.4.3.4,7.4.3.11]{lurie2016higher} and \cite[Proposition \RomanNumeralCaps{2}.1.26]{hedenlund2020multiplicative} for example. This provides our main example of an object in $\Alg(\Tow(\Sp))$, namely $\tau_{\ge \bullet}A$ where $A$ is a connective $\mathbb{E}_1$-ring.

\begin{rema}\label{some notations for rel} 
For an  $R^\bullet \in \Alg(\Tow(\Sp))$, we construct a map in $\Alg(\Tow(\Sp))$
\[R^\bullet \to \mathrm{gr}^0R^{\bullet}\]
where (abusing notation)  $\mathrm{gr}^0R^{\bullet}$ is an object of  $\Alg(\Tow(\Sp))$  satisfying $(\mathrm{gr}^0R^{\bullet})^0=\mathrm{gr}^0R^{\bullet}$ and $(\mathrm{gr}^0R^{\bullet})^i=0$ for $i>0$, and the map above induces an equivalence after applying $\grzero$. By \cite[Lemma B.0.6 (\romannumeral 4)]{hahn2020redshift}, there is an $\bE_1$-algebra map of $\N$-graded  spectra $\gr R^\bullet \to (\gr R^\bullet)^0$ that is an equivalence in weight $0$ and the target is trivial elsewhere. Let $\zeta$ denote the right adjoint to $\gr$ described in  \cite[Remark 3.2.11]{raksit2020hochschild}. Then the desired map can be given by the  composite 
\[R^\bullet \to \zeta (\gr R^\bullet ) \to \zeta (\gr R^\bullet)^0 \]
where the first map above is the unit of the monoidal adjunction $\gr \dashv \zeta$. Note that  we sometimes abuse notation and write $\textup{gr}^0 R^\bullet$ for the $\bE_1$-ring corresponding to the $\N$-graded $\bE_1$-ring $\textup{gr}^0 R^\bullet$ described above. 
\end{rema}

\begin{nota}\label{nota relative versions for filtered rings}
Let 
$F\in \{\thh,\ntc,\tp,\tc^{+}, (\thh^{tC_p})^{hS^1},\TC\}$. 
Given a map of $\mathbb{E}_1$-rings $A\to B$, we write 
\[ F(A,B):=\mathrm{fib}\big(F(A)\to F(B) \big)\,. \]
Given $R^{\bullet} \in \Alg(\Tow(\Sp))$, we write 
\[ F_{\textup{rel}}(R^0):=F(R^0,\grzero R^{\bullet})\]
where $\mathrm{gr}^0R$ is regarded as an $\mathbb{E}_1$-ring and the filtration $R^{\bullet}$ is implicit. If $A$ is an $\mathbb{E}_1$-ring, we will implicitly equip $A$ with the Whitehead filtration $\wtw A$ so that 
\[ F_{\textup{rel}}(A):=F(A,\pi_0A) \]
unless a different filtration is explicitly stated. 
 
Now we restrict to 
$F\in \{\thh,\ntc,\tp,\tc^{+}, (\thh^{tC_p})^{hS^1}\}$.
 Given $R^{\bullet}\in \Alg(\Tow(\Sp))$, we write \[F_{\textup{rel}}(R^{\bullet}):=\mathrm{fib}\big(F(R^{\bullet})\to F(\mathrm{gr}^0R^{\bullet})\big)\]  
 where $\mathrm{gr}^0R^{\bullet}$ above is regarded as an $\mathbb{E}_1$-ring in filtered spectra as in \autoref{some notations for rel}. 
Given $R^*\in \Alg(\Gr(\Sp))$, we write $F_{\textup{rel}}(R^*):=\mathrm{fib}(F(R^*)\to F(R^0))$ using $R^* \to R^0$ from \cite[Lemma B.0.6 (\romannumeral 4)]{hahn2020redshift}.

In each case above, we refer to $F_{\textup{rel}}$ as the relative counterpart to $F$.
\end{nota}

\begin{rema}
Note that in the situation of \autoref{nota relative versions for filtered rings}, $F_{\rel}(R^0)$ depends on the tower $R^\bullet$ and not only on $R^0$.
\end{rema}

We now discuss sufficient conditions on $\thh$ of an $\mathbb{E}_1$-algebra in filtered spectra for the May filtration to converge. 
\begin{defi}\label{def:connectivit-assumption}
We say $\mathrm{conn}(X)=n$ if $n$ is the largest integer such that the canonical map $\tau_{\ge n}X\to X$ is an equivalence. We say that $R^{\bullet}\in \Alg(\Tow(\Sp))$ is $\thh$-convergent if we have $\mathrm{conn}(\thh(R^{\bullet})^t)=s(t)$ where $s(t)\to \infty$ as $t\to \infty$. We will say that $R^{\bullet}$ is strongly  $\thh$-convergent if  $\mathrm{conn}(\thh(R^{\bullet})^t)\geq t$ and $\mathrm{conn}(\grzero R^\bullet)\geq 0$. 

\end{defi}

With this definition, $\grzero R^\bullet\in \mathrm{Alg}(\Tow(\Sp))$ is THH-convergent for every $R^\bullet \in \Tow(\Sp)$ simply because $\grzero R^\bullet$ is trivial in positive  weights and therefore so is $\THH(\grzero R^\bullet)$, see~\autoref{some notations for rel}. 

\begin{exam}\label{ex-thh-convergent}
The same argument as in~\cite[Lemma~3.4.7]{AKS18} shows that when $A$ is a connective $\mathbb{E}_1$-ring, then $\tau_{\ge \bullet}A$ is strongly $\thh$-convergent.   Explicitly, note that the weight $t$ component of $(\wtw A)^{\otimes k}$ is a colimit of spectra whose connectivity is at least $t$ \cite[Notation 3.1.2]{lurie2015rotation} and hence the weight $t$ component of $(\wtw A)^{\otimes k}$ has connectivity at least $t$. The claim follows by noting that $\thh(\wtw A)^{t}$ is a colimit over the weight $t$ components of $(\wtw A)^k$.

\end{exam}

\begin{cons}\label{example-p-adic-filtration}
The $p$-adic filtration on $W(\field)/p^n$ gives a commutative monoid in the $1$-category of $\N$-filtered abelian groups; as a result, it gives an $\bE_\infty$-algebra in the corresponding $\infty$-category $\Tow(\mathcal{A}b)$. By the lax symmetric monoidal inclusion of the heart $\mathcal{A}b \to \Sp$, there is a lax symmetric monoidal functor $\Tow(\mathcal{A}b) \to \Tow(\Sp)$  \cite[Corollary 3.7]{nikolaus2016stablemultyoneda}. This gives an $\N$-filtered $\bE_\infty$-algebra in $\Tow(\Sp)$ realizing $W(\field)/p^n$ in the colimit.
\end{cons}

\begin{nota}\label{nota  p adic filtration on truncated witt}
We let $S_{\field}^{\bullet} \in \Alg(\Tow(\Sp))$ denote the $p$-adic filtration on $S_{\field}^0 := W(\field)/p^n$ from \autoref{example-p-adic-filtration} with $\gr S_{\field}^\bullet = \field[t]/t^n$ where $t$ is of weight $1$ and degree $0$. 
\end{nota}

\begin{exam}\label{exam p adic filtration is thh convergent}
The $p$-adic filtration $S_{\field}^{\bullet}$ on $W(\field)/p^n$ is THH-convergent. Since the $p$-adic filtration $S_{\field}^{\bullet}$ on $W(\field)/p^n$ is trivial for $\bullet> n-1$ we deduce that $(S_{\field}^{\bullet})^{\otimes k}$ is trivial in weight > $k(n-1)$. Therefore for a given weight $t$, non-trivial contributions to $\thh(S_{\field}^{\bullet})^t$ can only come from simplicial level $k$ (in the cyclic bar construction) such that $k\geq t/(n-1)$. Since simplicial level $k$ contributes to homotopy degree $k$ and above, we deduce that $\textup{conn}(\thh(S_{\field}^{\bullet})^t) \geq t/(n-1)$ and hence $S_{\field}^{\bullet}$ is THH-convergent.  
\end{exam}

\begin{rema}
In applications, we will only need the filtrations $\wtw A$ (for a connective $\bE_1$-ring $A$) and  $S_{\field}^{\bullet}$ from \autoref{nota  p adic filtration on truncated witt}. 
\end{rema}

\begin{prop}\label{lemm conditional convergence of filtrations}
Let  $R^{\bullet}\in \Alg(\Tow(\Sp))$ be $\thh$-convergent  and let 
  \[
   F\in \{\thh,\ntc,\tp,\tc^+,(\thh^{tC_p})^{hS^1}\} \,. \] 
 Then  
  \[ \lim F(R^{\bullet})\simeq 0 \qquad \text{ and } 
   \qquad \colim F(R^\bullet) = F(R^0), \]
    moreover,  these also hold for the relative counterparts of $F$. 
\end{prop}

\begin{proof}
The functor $\colim \co \Tow(\Sp)\to \Sp$ is a symmetric monoidal left adjoint (as it is left Kan extension along $\N\to 0$) \cite[Corollary~3.7]{nikolaus2016stablemultyoneda} and hence it commutes with $\thh$. Moreover, the functor $\mathrm{colim}: \mathrm{Tow}(\Sp)\to \Sp$ is equivalent to the evaluation functor $\mathrm{ev}_0: \mathrm{Tow}(\Sp)\to \Sp$ that restricts along the canonical inclusion $0\to \N$ it is also a right adjoint; it therefore commutes with fiber sequences as well as homotopy fixed points. 
This proves the statements on colimits using the fiber sequences 
\begin{align}\label{norm-restriction}
    \Sigma \TC^+\to \TC^{-}\to \TP \,.
\end{align}
and 
\begin{align}\label{norm-restriction-cp}
    \thh_{hC_p}\to \thh^{hC_p}\to \thh^{tC_p} 
\end{align} 
and the fact that colimits commute with homotopy orbits. 

Now we prove $\lim F(R^\bullet)\simeq 0$. It is clear from our connectivity hypothesis that $\lim \thh(R^{\bullet})\simeq 0$ and since limits commute with homotopy fixed points it is also clear that $\lim \ntc(R^{\bullet})\simeq 0$. We also know that $\lim \tc^+(R^{\bullet})\simeq 0$ since $S^1$-homotopy orbits preserve connectivity and for $\tp$ it follows from the fiber sequence \eqref{norm-restriction}.
For $(\thh^{tC_p})^{hS^1}$, we first use a similar argument as above using the fiber sequence \eqref{norm-restriction-cp}

 In particular, this holds for $\mathrm{gr}^0R^{\bullet}\in \Alg(\Tow(\Sp))$ as in \autoref{some notations for rel} which is $\thh$-convergent as it is concentrated in weight $0$ and consequently for $F_{\rel}$ as well by passing to fibers. 
\end{proof}

\begin{rema}
Following \autoref{nota relative versions for filtered rings}, the colimit of $F_{\rel} (S_{\field}^\bullet)$ is realized as $F(W(\field)/p^n,\field)$.
\end{rema}

To identify the associated graded of these May filtrations, we need the following folklore result. 

\begin{prop}\label{prop assoc graded is the formal dga}
Let $A$ be a connective $\bE_1$-ring, then $\textup{gr}^*(\tau_{\geq \bullet}A) \in \textup{Alg}(\Gr(\Sp))$ is equivalent to the formal DGA corresponding to the graded ring $\pis A$. 
\end{prop}
\begin{proof}
Let $B$ denote the formal DGA in question. Following \autoref{rema z graded spectra}, it suffices to prove   $\textup{gr}^*(\tau_{\geq \bullet}A)\simeq B $ in $\textup{Alg}(\textup{Fun}(\mathbb{Z}^{\textup{ds}}, \Sp))$.

By \cite[Construction 3.3.2 (b) and Remark 3.3.3]{raksit2020hochschild}, both $\textup{gr}^*(\tau_{\geq \bullet}A)$ and $B$ lie in the heart of $\textup{Fun}(\mathbb{Z}^{\textup{ds}}, \Sp)$ under a $t$-structure that is compatible with its symmetric monoidal structure.  On the heart, this induces the usual symmetric monoidal structure on graded abelian groups (with Koszul sign rule). Applying the connective cover functor and then the $0$-truncation functor (under this $t$-structure), it suffices to show  $\textup{gr}^*(\tau_{\geq \bullet}A)\simeq B$ as $\bE_1$-algebras in the heart. Then this follows by the fact that $\pis \big(\textup{gr}^*(\tau_{\geq \bullet}A)\big) \cong \pis A$ as discrete graded rings.  
\end{proof}
Due to the proposition above, we denote $\textup{gr}^*(\tau_{\geq \bullet}A) \in \textup{Alg}(\Gr(\Sp))$ by $\pi_*A$.

\begin{prop}\label{E1page}
Let $R^{\bullet}\in \Alg(\Fil(\Sp))$. 
There are equivalences 
\begin{equation}\label{eq assoc graded of thh} \gr\thh(R^{\bullet})\simeq \thh(\gr R^{\bullet})
\end{equation}
and 
\begin{equation}\label{eq assoc graded of ntc and tp}
\gr \ntc (R^{\bullet} )\simeq \ntc(\gr R^{\bullet}) \textup{\ \ and \ } 
\gr \tp (R^{\bullet} ) \simeq \tp(\gr R^{\bullet}).
\end{equation}
These equivalences also hold for $\rthh$, $\rntc$ and $\rtp$ as well as $\tc^{+}$ and $\rtc^{+}$.
\end{prop}
\begin{proof}
The first part of the statement follows as in~\cite[Theorem~3.3.10]{AKS18}, the main point being that $\gr$ is a symmetric monoidal left adjoint by \cite[Proposition~3.2.1]{lurie2015rotation} and \cite[Lemma~3.30]{GP18} so it commutes with THH and one therefore obtains an equivalence of $S^1$-spectra. For $\TC^{-}$, it therefore suffices to note that we have a canonical map 
\[ 
\gr\TC^{-}(R^{\bullet})\longrightarrow \TC^{-}(\gr R^{\bullet}).
\]
and it is an equivalence since in weight $i$, $\gr\TC^{-}(R^{\bullet})$ is the cofiber of the map 
\[\TC^{-}(R^{\bullet})^{i+1}\to \TC^{-}(R^{\bullet})^{i}\] 
induced by the structure maps of the filtered object  and this cofiber commutes with homotopy fixed points. Consequently, this also implies the statement for $\TP$ using the cofiber sequence~\eqref{norm-restriction}. The case of $\tc^+$ also follows in the same way. The case of $\mathrm{gr}^0R^{\bullet}\in \Alg(\Tow(\Sp))$ in the sense of \autoref{some notations for rel} is then a special case and the last statement follows since $\gr$ commutes with fiber sequences. 
\end{proof}

We will be using the first two items of the following corollary without further mention. 
\begin{coro}\label{coro frel associated graded and strong convergence}
Let $F\in \{\thh,\ntc,\tp,\tc^{+}, (\thh^{tC_p})^{hS^1}\}$, and $R \in \Alg(\Tow(\Sp))$. We have:
\begin{enumerate}
\item \label{item 1 coro frel associated graded and strong convergence}$\gr F_{\rel}(R^\bullet) \simeq \gr F(R^\bullet)$ for $*>0$ and $\grzero F_{\rel}(R^\bullet) \simeq *$.
\item \label{item 2 coro frel associated graded and strong convergence}$F_{\rel}(R^\bullet)^i \simeq F(R^\bullet)^i$ for $i>0$ and $F_{\rel}(R^\bullet)^0 \simeq F_{\rel}(R^\bullet)^1$.
\item \label{item 3 coro frel associated graded and strong convergence}Assume that $R^\bullet$ is strongly $\thh$-convergent. Then $\mathrm{conn}(\textup{gr}^i\rthh(R^\bullet)) \geq i$ for each $i \geq 0$.
\end{enumerate}
\end{coro}

\begin{proof}
The map $R^\bullet \to \grzero R^\bullet$ induces an isomorphism in the weight zero component of the associated graded and the target is concentrated in weight $0$. Since $\thh(\grzero R^\bullet)$ is also concentrated in weight zero, the same holds for $F(R^\bullet) \to F(\grzero R^\bullet)$ by \autoref{E1page} and the fact that $\thh(B^*)^0 \simeq \thh(B^0)$ for every $B \in \Gr(\Sp)$. This proves \eqref{item 1 coro frel associated graded and strong convergence}. Now  the first part of \eqref{item 2 coro frel associated graded and strong convergence} follows by the aforementioned fact that $F(\grzero R^\bullet)$ is concentrated in weight $0$. The second part of \eqref{item 2 coro frel associated graded and strong convergence} follows by the second part of $\eqref{item 1 coro frel associated graded and strong convergence}$.

By using $\eqref{item 1 coro frel associated graded and strong convergence}$, the proof of \eqref{item 3 coro frel associated graded and strong convergence} boils down to proving that $\mathrm{conn}(\textup{gr}^i\thh(R^\bullet)) \geq i$ for each $i >0$. In the cofiber sequence 
\[\thh(R^\bullet)^{i+1} \to \thh(R^\bullet)^i \to \textup{gr}^i(\thh(R^\bullet))\]
the left hand side and the middle term have $\mathrm{conn}(-) \geq i$ since $R^\bullet$ is strongly THH-convergent. This gives the desired result for the right hand side by considering the induced long exact sequence. 
\end{proof}

\begin{thm}[The May spectral sequence for $\TC^{-}$ and $\TP$]\label{mayss}
Suppose $R^{\bullet}\in \Alg(\Tow(\Sp))$ is $\thh$-convergent. Then there are conditionally convergent spectral sequences
\begin{equation*} 
\begin{split}
 \pi_*(\TC^{-}(\gr R^{\bullet} ))\Longrightarrow& \pi_*(\TC^{-}(R^0)) \qquad \textup{ \ \ \ \ }\qquad  
 \pi_*(\TP(\gr R^{\bullet} )) \Longrightarrow\pi_*(\TP(R^0))\\
 \textup{\ \ \ }\pi_*(\TC^{-}_{\rel}(\gr R^{\bullet} )) \Longrightarrow & \pi_*(\rtc^{-}(R^0))    \qquad \textup{ \ \ \ }\qquad 
 \pi_*(\TP_{\rel}(\gr R^{\bullet} )) \Longrightarrow\pi_*(\rtp (R^0))\,.
\end{split}
\end{equation*}
We also have similar spectral sequences for $\thh$ and $\rthh$.
\end{thm}
\begin{proof}
Conditional convergence follows from \autoref{lemm conditional convergence of filtrations}. The identification of the $\mathrm{E}_1$-pages follows from \autoref{E1page}.
\end{proof}

\begin{prop}\label{lemm thh is p complete when thh gr is}
Let $R^{\bullet}\in \Alg(\Tow(\Sp))$ be $\thh$-convergent and assume that 
$\thh(\gr R^{\bullet})$ is levelwise $p$-complete (meaning $p$-complete at each weight component). Then $\thh(R^{\bullet})$ and $\rthh(R^{\bullet})$ are also levelwise $p$-complete. In particular, $\THH(R^0)$ and $\rthh(R^0)$ are $p$-complete. 
\end{prop}
\begin{proof}
We first prove that $\THH(R^0)$ is $p$-complete. For this, consider the cofiber sequence of $\mathbb{N}$-filtered spectra 
\[\thh(R^{\bullet}) \to \textup{cons}(\thh(R^0)) \to \thh(R^0)/\thh(R^{\bullet}) \]
where the middle term is the constant filtration on $\thh(R^0)$. Since $\gr$ commutes with cofiber sequences, we deduce that the associated graded of $\thh(R^0)/\thh(R^{\bullet})$ is given by $\Sigma \THH(\gr R)$ and since $\lim \thh(R^\bullet)\simeq 0$ by 
\autoref{lemm conditional convergence of filtrations}, we know that 
\[ \lim \THH(R^0)/\THH(R^{\bullet})\simeq \THH(R^0)\,.\]
Since $p$-complete spectra are closed under limits, it suffices to show that $\THH(R^0)/\THH(R^{\bullet})$ is $p$-complete at each level. This follows by induction on the tower $\THH(R^0)/\THH(R^{\bullet})$ since it is trivial in weight $0$ and its associated graded $\Sigma \THH(\gr R)$ is $p$-complete at each weight by hypothesis. This proves that $\THH(R^{\bullet})$ is $p$-complete at weight $0$ by \autoref{lemm conditional convergence of filtrations}. Since its associated graded $\THH(\gr R)$ is $p$-complete by hypotheses, we deduce the $p$-completeness at each weight of $\THH(R^{\bullet})$ by induction over the tower.  Then the $p$-completeness of $\rthh(R^\bullet)$ follows by \autoref{coro frel associated graded and strong convergence} \eqref{item 2 coro frel associated graded and strong convergence}. The last statement follows by noting that $\thh(R^\bullet)^0 \simeq \thh(R^0)$ and $\rthh(R^\bullet)^0 \simeq \rthh(R^0)$ due to \autoref{lemm conditional convergence of filtrations}.
\end{proof}

\begin{rema}\label{rema tcp and then hs1 agree with ts1 when p complete} 

 \autoref{lemm thh is p complete when thh gr is} provides equivalences of $\N$-filtered spectra
\[(\thh(R^{\bullet})^{tC_p})^{hS^1}\simeq \tp(R^{\bullet}) \textup{\ \ and \ }(\rthh(R^{\bullet})^{tC_p})^{hS^1}\simeq \rtp(R^{\bullet})\] 
when $\thh(\gr R^{\bullet})$ is $p$-complete and $R^\bullet$ is $\thh$-convergent. This is due to the discussion in \cite[Section 2.3]{bhatt2019thhandintegralpadichodge} (together with  \cite[Lemma \RomanNumeralCaps{2}.4.2]{nikolausscholze2018topologicalcyclic}) which deduce that $M^{tS^1}$ is $p$-complete as soon as $M$ is bounded below and $p$-complete. Therefore, we use these equivalences without further mention and work with $\tp$ whenever possible. Note that $\thh(\pis A)$ is $p$-complete when $p=0 \in \pi_0A$. This is because, in this situation, the formal DGA $\pis A$ is an $\fp$-algebra and therefore $\thh(\pis A)$ is an   $\fp$-module (see \autoref{prop assoc graded is the formal dga}). Alternatively, it suffices to know that $\mathrm{gr}^*R^{\bullet}$ is an $\field_p$-algebra, as in the case of \autoref{nota  p adic filtration on truncated witt}. Restricting to the weight zero components of the equivalences above, we also observe
\[(\thh(R^0)^{tC_p})^{hS^1}\simeq \tp(R^0) \textup{\ \ and \ }(\rthh(R^0)^{tC_p})^{hS^1}\simeq \rtp(R^0)\] 
whenever $\thh(\gr R^{\bullet})$ is $p$-complete and $R^\bullet$ is $\thh$-convergent. 
\end{rema}

\subsection{Proof of \autoref{oddness}}\label{proof-oddness}
Now we prove \autoref{oddness} and a corollary that we will need later. 
\oddnessofntcandtp*
\begin{proof}[Proof of \autoref{oddness}]
By \autoref{mayss} there are spectral sequences 
\[
\pi_*\TC^{-}_{\rel}(\pi_*A)\implies \pi_*\TC^{-}_{\rel}(A)  \qquad \text{ and }\qquad
\pi_*\TP_{\rel}(\pi_*A)\implies \pi_*\TP_{\rel}(A) 
\]
and the $\mathrm{E}_1$-pages are concentrated in odd stems by \autoref{theo tc ntc tp of formal dgas} and consequently the abutments are trivial in even stems. 
\end{proof}

\autoref{oddness} has the following corollary. 

\begin{cor}\label{cor tc in terms of ker and coker}
    Suppose $A$ is an $\mathbb{E}_1$-ring with $\pi_*A\cong\field[x_{2m}]$  where $p\nmid m>0$.  Then for all $r$ we can identify
     \[\pi_{2r}\rtc(A) = \textup{coker}\big(\pi_{2r+1}(\varphi - \can)\big)\]
     and
     \[\pi_{2r+1}\rtc(A) = \textup{ker}\big(\pi_{2r+1}(\varphi - \can)\big).\]
\end{cor}
\begin{proof}
    This follows from \autoref{oddness} in light of~\cite[Corollary~1.5]{nikolausscholze2018topologicalcyclic}. 
\end{proof}

\section{Cardinality of topological cyclic homology}\label{cardinality}

Here, we construct what we call the orbit filtration for topological cyclic homology and prove \autoref{theo even trivial gives odd} for ring spectra with polynomial homotopy.  
\subsection{Orbit filtration for TC}\label{sec:orbit-filt}

The May filtrations on $\ntc(R^0)$ and $\tp(R^0)$ of an $\mathbb{E}_1$-algebra $R^{\bullet}$ in $\mathrm{Tow}(\Sp)$ do not upgrade to a filtration on $\tc$ immediately. This is because, the Frobenius map does not respect the filtrations on $\ntc(R^{\bullet})$ and $\tp(R^{\bullet})$ in the obvious way; it is rather given by an $S^1$-equivariant map of filtered spectra \cite[Appendix~A]{antieau2020beilinson}
\[ 
 \varphi_p\co \thh(R^{\bullet})\longrightarrow R_p\thh(R^{\bullet})^{tC_p} 
\]
 where $R_p$ denotes restriction along $\cdot p \co \z \to \z$, i.e.\ it is given by $S^1$-equivariant maps 
 \[\varphi_p\co \THH(R^{\bullet})^i\longrightarrow (\THH(R^{\bullet})^{tC_p})^{ip}.\]
 We have the same situation for the the Frobenius on $\rthh(R^\bullet)$.
In contrast, the canonical maps 
\[ 
\textup{can} : \TC^{-}(R^{\bullet}) \longrightarrow \TP(R^{\bullet}) \textup{\ \ and \ } \textup{can} : \TC^{-}_\rel(R^{\bullet}) \longrightarrow \TP_\rel(R^{\bullet})
\]
are  filtered maps without modifications. To work around this issue, we use the following construction.

\begin{con}\label{natural-trans-R}
In the $\mathbb{N}$-filtered setting, it is possible to turn $\varphi_p$ into a map of filtered $S^1$-spectra by postcomposing with the canonical map $ R_p F^\bullet \to  F^\bullet$ that uses the filtration maps to reduce the weight. To make this operation functorial, note that there is a unique natural transformation from $\cdot p \co \mathbb{N} \to \mathbb{N}$ to the identity functor which gives a natural transformation $R_p \Rightarrow \textup{id}$. 

\end{con}

\begin{rem}
The argument above relies on the fact that we are working in the category $\mathrm{Tow}(\Sp)$ and not $\mathrm{Fil}(\Sp)$ since there is no natural transformation from the multiplication by $p$ functor $\cdot p : \mathbb{Z}\to \mathbb{Z}$ on the integers to the identity functor. 
\end{rem}

\begin{defn}\label{defi filtered cyclotomic thh}
For $R^{\bullet}\in \Alg(\Tow(\Sp))$, we let 
 $\thh(R^{\bullet}) \in \mathrm{Tow}(\cycsp)$ 
 denote the tower of cyclotomic spectra whose Frobenius is the canonical composite
 \[\widetilde{\varphi}_p \co \thh(R^{\bullet})  \xrightarrow{\varphi_p} R_p\thh(R^{\bullet})^{tC_p} \to \thh(R^{\bullet})^{tC_p}\]
 where the second map is given by the natural transformation constructed above.  We then write 
 \[ \widetilde{\varphi}:=\widetilde{\varphi}_p^{hS^1}: \mathrm{TC}^{-}(R^{\bullet})\longrightarrow (\mathrm{THH}(R^{\bullet})^{tC_p})^{hS^1}\,.
 \]
In particular, for a connective $\mathbb{E}_1$-ring $A$, 
 $\thh(\tau_{\geq \bullet} A) \in \mathrm{Tow}(\cycsp)$ denotes the tower of cyclotomic spectra whose Frobenius is the canonical composite 
 \[\widetilde{\varphi}_p \co \thh(\tau_{\geq \bullet} A)  \xrightarrow{\varphi_p} R_p\thh(\tau_{\geq \bullet} A)^{tC_p} \to \thh(\tau_{\geq \bullet} A)^{tC_p}\]
 where the second map is given by the natural transformation constructed above. The definitions of  $\rthh(\tau_{\geq \bullet} A)\in  \mathrm{Tow}(\cycsp) $ and $\rthh(R^{\bullet})$ are given in the same way.
\end{defn}

\begin{prop}\label{lemm phi vanish in the associated graded}
Let $R^{\bullet}\in \Alg(\Tow(\Sp))$ and $A$ be a connective $\bE_1$-ring. The Frobenius 
$\widetilde{\varphi}_p$
 on $\rthh(R^{\bullet}) \in \mathrm{Tow}(\cycsp)$ as above induces a null-homotopic map of $S^1$-spectra in its associated graded. The colimit of the map $\widetilde{\varphi}_p$ is the usual Frobenius map $\varphi_p$ on $\rthh(R^0)$. 
 Moreover, the Frobenius map 
 \[ \widetilde{\varphi} : \rntc(R^{\bullet})\to (\mathrm{THH}_{\textup{rel}}(R^{\bullet})^{tC_p})^{hS^1}
 \]
 is null-homotopic on associated graded and it has colimit the usual Frobenius map $\varphi_p^{hS^1}$.
\end{prop}

\begin{proof}
We begin by proving the first statement. Since we are working in the reduced setting, the associated graded is trivial in weight $0$ (both for the target and source of $\widetilde{\varphi}_p$) by \autoref{coro frel associated graded and strong convergence} \eqref{item 1 coro frel associated graded and strong convergence}; therefore, it suffices to prove this for positive weight. 

The (unique) natural transformation  $\cdot p \Rightarrow \textup{id}$ defining $R_p \Rightarrow \textup{id}$ above factors through natural transformations between appropriately defined maps $\mathbb{N} \to \mathbb{N}$ (that we specify at the end),  which in turn factor the  map  $R_p F^\bullet \to F^\bullet$ (for $F^\bullet \in \mathrm{Tow}(\Sp)$) through the towers:
\begin{equation}\label{eq restriction along h1}
\cdots \to F^8 \to F^6 \xrightarrow{\textup{id}} F^6 \to F^4 \xrightarrow{\textup{id}} F^4 \to F^2\to F^0
\end{equation}
and 
\begin{equation}\label{eq restriction along h2}
\cdots \to F^6 \xrightarrow{\textup{id}} F^6 \to  F^4 \xrightarrow{\textup{id}} F^4 \to  F^2 \xrightarrow{\textup{id}} F^2\to F^0.
\end{equation}
The first tower has trivial associated graded in positive even weight and the second has trivial associated graded in positive odd weight. In particular, the associated graded of the map $R_p F^{\bullet} \to F^{\bullet}$ is trivial in the positive degrees. Since this map is used in the definition of $\tilde{\varphi}_p$, this proves the desired result. 

Indeed, \eqref{eq restriction along h1} correspond to restriction along the map $h_1\co \N \to \N$ given by $h_1(0) = 0$, and $h_1(n) = n+2$ for even $n\geq 2$ and $h_1(n) = n+1$ for odd $n \geq 1$. Furthermore, \eqref{eq restriction along h2} is given by restriction along $h_2(n) = n$ for even $n$ and $h_2(n) = n+1$ for $n$ odd. Since there is a unique map between any two objects in $\N$, and since $pn\geq h_1(n)\geq h_2(n)\geq n$, the unique natural transformation $R_p \Rightarrow \textup{id}$ factors uniquely through restriction along $h_1$ and restriction along $h_2$. 

The second statement follows since the colimit is realized as the weight zero part, the Tate construction is applied levelwise and the second map in the composite $\widetilde{\varphi}_p$ is the identity in weight zero. The other statements follow from the first and second statements. 
\end{proof}

\begin{nota}
From now on, we always regard $\mathrm{THH}(R^{\bullet})$ and $\mathrm{THH}_{\textup{rel}}(R^{\bullet})$ as cyclotomic objects in filtered spectra using the cyclotomic structure from \autoref{defi filtered cyclotomic thh}.
\end{nota}

\begin{defn}\label{defi orbitS filtration on tc}
Given $R^{\bullet}\in \Alg(\Tow(\Sp))$, let 
\[\filorb \tc(R^0) := \tc(\thh(R^{\bullet})) \in \mathrm{Tow}(\Sp)\]
where $\tc(-)$ above applies levelwise $\tc$ to the tower of cyclotomic spectra in \autoref{defi filtered cyclotomic thh}. Similarly, we let
\[\filorb \rtc(R^0) := \tc(\rthh(R^{\bullet})) \in \mathrm{Tow}(\Sp) \,.\]
If we write $\filorb \tc(A)$ or $\filorb \rtc(A)$ for a connective $\bE_1$-ring $A$, then we mean $\filorb \tc((\wtw A)^0)$ or $\filorb \rtc((\wtw A)^0)$ respectively unless another filtration is explicitly stated. 
\end{defn}

Equivalently,  $\filorb \rtc(R^0)$ is given by the fiber of the map of filtered spectra $\widetilde{\varphi} -\textup{can}$; i.e. we have:
\begin{equation}\label{eq fiber sequence defining orbit filtration}
\filorb \rtc(R^0) := \textup{Fib}\big(\rntc(R^{\bullet})  \xrightarrow{\widetilde{\varphi}-\textup{can}} (\rthh(R^{\bullet})^{tC_p})^{hS^1})\big)
\end{equation}
and we write 
\begin{equation}\label{eq orbit assoc graded}
\grorbast \rtc(R^0):=\filorbast \rtc(R^0)/\filorbastplusone \rtc(R^0)\,.
\end{equation}

\begin{rema}
Note that $\filorb \tc(R^0)$, $\grorbast \tc(R^0)$, $\filorb \rtc(R^0)$ and $\grorbast \rtc(R^0)$ depend on the filtration $R^\bullet$ not only on $R^0$.
\end{rema}

\begin{thm}[Orbit filtration for relative TC]\label{theo orbit filtration}
Let $R^{\bullet}\in \Alg(\Tow(\Sp_{(p)}))$ be $\thh$-convergent. 
The $\N$-filtered spectrum $\filorb \rtc(R^0)$ satisfies the following:
\begin{enumerate}
\item \label{Orbit-filt-1} the colimit of  $\filorb \rtc(R^0)$ is realized by  $\rtc(R^0)$.
\item \label{Orbit-filt-2} $\filorb \rtc(R^0)$ is complete, i.e.\ 
\[\lim \filorb \rtc(R^0) \simeq 0.\]
\item  \label{Orbit-filt-3}
If $\rthh(\gr R^{\bullet})$ is $p$-complete, then 
\[\grorbast \rtc(R^0) \simeq \Sigma \rtcpl(\gr R^{\bullet}).\]
More generally, there is an equivalence after $p$-completion:
\[
\grorbast (\rtc(R^0))_p\simeq \Sigma \rtcpl(\gr R^{\bullet})_p\, .
\]

\item  \label{Orbit-filt-4} If $\rthh(\gr R^{\bullet})$ is $p$-complete and $R^{\bullet}$ is strongly $\thh$-convergent then 
\[\pi_s \filorbi \rtc(R^0) = 0\]
for every $i\geq s$.
\end{enumerate}
\end{thm}
\begin{proof}
The first statement follows from \autoref{lemm phi vanish in the associated graded}. 

For the second statement, we take the limit of  the fiber sequence in \eqref{eq fiber sequence defining orbit filtration}. By \autoref{lemm conditional convergence of filtrations}, second and the third terms in this fiber sequence have vanishing limits which implies the desired result as fiber sequences commute with limits.

For the third statement, we take the associated graded of the fiber sequence in \eqref{eq fiber sequence defining orbit filtration}. As $\widetilde{\varphi}$ vanishes in the associated graded due to \autoref{lemm phi vanish in the associated graded}, we obtain 
\begin{equation}\label{eq associated graded of the orbit filtration}
\grorbast \rtc(R^0) \simeq \textup{Fib}\big(\rntc(\gr R^{\bullet}) \xrightarrow{-\textup{can}} (\rthh(\gr R^{\bullet})^{tC_p})^{hS^1}\big).
\end{equation}
This fiber is given by $\Sigma \rtcpl(\gr R^{\bullet})$ as the target of $-\textup{can}$ here is equivalent to $\rthh(\gr R^{\bullet})^{tS^1}$ (see \autoref{rema tcp and then hs1 agree with ts1 when p complete}). For the more general case, we know by~\cite[Lemma~II.4.2]{nikolausscholze2018topologicalcyclic} that the target of $-can$ is given by  $(\rthh(\gr R^{\bullet})^{tC_p})^{hS^1}\big)_p\simeq \rthh(\gr R^{\bullet})^{tS^1}_p$ and the $p$-completion of the fiber can be identified with $\Sigma\mathrm{TC}^+_\rel(\gr R^{\bullet})_p$. 

Now we prove the last statement. For this, note that $\pi_s \rthh(\gr R^{\bullet})^j = 0$ whenever $j >s$  by \autoref{E1page} and \autoref{coro frel associated graded and strong convergence} \eqref{item 3 coro frel associated graded and strong convergence}. Since homotopy orbits preserve connectivity, we deduce $\pi_{s} \big(\Sigma \rtcpl (\gr R^{\bullet})^j\big) = 0$ whenever $j\geq s$. By \eqref{Orbit-filt-3}, this implies that the maps
\begin{equation}\label{eq surjective structure maps in the orbit filtration}
\pi_s \filorbjplusone
\rtc(R^0)\to \pi_s \filorbj \rtc(R^0)\end{equation}
are surjective  for $j\geq s$. By \eqref{Orbit-filt-2}, the filtration is complete so the non-derived limit
\[\lim_j \pi_s \filorbj \rtc(R^0)\]
 is also trivial for each integer $s$ by the Milnor sequence. Assume to the contrary that there is $0 \neq x \in \pi_s \filorbi \rtc(R^0)$ for some $i \geq s$, then by the surjectivity of \eqref{eq surjective structure maps in the orbit filtration}, one may inductively choose $x_j \in \pi_s \filorbj \rtc(R^0)$ for $j \geq i$ so that $x_i = x$ and $x_{j+1}$ is sent to $x_j$ under \eqref{eq surjective structure maps in the orbit filtration} which gives a nontrivial element of the non-derived limit $\lim_j \pi_s \filorbj \rtc(R^0)$ and this  contradicts the aforementioned vanishing of this limit.
\end{proof}

\begin{rema}
Without the $p$-completeness assumption on $\thh(\gr R^{\bullet})$, the associated graded of the orbit filtration  is given by: 
\begin{equation}
\grorbast \rtc(R^0) \simeq \textup{Fib}\big(
\rntc(\gr R^{\bullet}) \xrightarrow{-\can}  \rtp(\gr R^\bullet)\to (\rthh(\gr R^{\bullet})^{tC_p})^{hS^1}\big),
\end{equation}
as mentioned in \eqref{eq associated graded of the orbit filtration}. The target of the map above is $\rtp(\gr R^{\bullet})_p^\wedge$ when $\rthh(\gr R^{\bullet})$ is bounded below in each weight component due to \cite[Lemma \RomanNumeralCaps{2}.4.2]{nikolausscholze2018topologicalcyclic}.
\end{rema}

\begin{rema} 
Using \autoref{ex-thh-convergent} and \autoref{theo orbit filtration} \eqref{Orbit-filt-3} in the case $R^{\bullet}=\tau_{\ge \bullet}A$ where $A$ is a connective $\mathbb{E}_1$-ring provides an independent proof that $p$-complete relative topological cyclic homology commutes with filtered colimits of connective $\mathbb{E}_1$-rings, which is also implied by the stronger result of Clausen--Mathew--Morrow~\cite[Corollary~2.15]{CMM21}. 
\end{rema}

\begin{cor}\label{coro orbit spectral sequence}
Let $R^{\bullet}\in \Alg(\Tow(\Sp))$ be $\mathrm{THH}$-convergent and assume that $\rthh(\gr R^{\bullet})$ is $p$-complete. Then there is a conditionally convergent spectral sequence 
\begin{equation}\label{eq orbit spectral sequence}
\mathrm{E}_1=\pi_{*-1}\mathrm{TC}_{\textup{rel}}^{+}(\gr R^{\bullet})\Longrightarrow \pi_*\rtc(R^0) \,.
\end{equation}
\end{cor}

\begin{rema}
We can also produce a spectral sequence by using $\thh$ instead of $\rthh$, but then $\widetilde{\varphi}$ is non-zero in filtration zero and we have a contribution of $\TC(\textup{gr}^0R^\bullet)$ to the $\mathrm{E}_1$-term. 
\end{rema}

\begin{rema}
A similar spectral sequence to \eqref{eq orbit spectral sequence} was considered by Brun~\cite{Bru01} and Angeltveit~\cite{Angeltveit2015kTheoryfinitewitt}, though we have not directly compared their filtrations to ours.
\end{rema}

\subsection{The partial orbit filtration for relative TC}\label{sec partial orbit filtration for tcrel}

Here, we define a quotient of the filtration $\rthh(R^{\bullet})$ for $R^{\bullet}\in \Alg(\Tow(\Sp))$ that allows us to recover $\rtc(R^0)$ in low degrees when $R^{\bullet}$ is strongly THH-convergent. For a stable presentably symmetric monoidal $\infty$-category  $\mathcal{C}$, let 
\[C_{\leq n} \co \mathrm{Tow}(\mathcal{C}) \to \mathrm{Tow}(\mathcal{C})\]
denote the functor given by restriction along the map $c: \mathbb{N} \to \mathbb{N}$ that we define as $c(m) = m$ if $m>n$ and $c(m) = n$ if $m \leq n$. In particular, ${C}_{\leq n}F^\bullet$ is given by
\[\cdots \to  F^{n+2} \to F^{n+1} \to F^n \xrightarrow{\textup{id}} F^{n} \xrightarrow{\textup{id}} \cdots \xrightarrow{\textup{id}} F^n. \]
There is a unique natural transformation $c \Rightarrow \textup{id}$ which provides a natural transformation ${C}_{\leq n} \Rightarrow \textup{id}$. We let ${Q}_n$ denote the cofiber of this tranformation so that we have a cofiber sequence
\begin{equation}\label{eq cofiber sequence for qn}
{C}_{\leq n} F^\bullet \to F^\bullet \to {Q}_n F^\bullet.
\end{equation}
Then ${Q}_n F^\bullet$ is given by 
\[\cdots \to 0 \to 0 \to F^{n-1}/F^n \to F^{n-2}/F^n \to \cdots \to F^0/F^n.\]
We will sometimes write ${Q}_n^i F^\bullet :=F^i/F^n$.\footnote{The construction $Q_nF^{\bullet}$ also appears in~\cite[Construction~3.17]{Ant24} with the notation $\mathrm{gr}_{F}^{[0,n)}$. Note that we are considering $\mathbb{N}$-filtered objects here whereas Antieau considers $\mathbb{Z}$-filtered objects, but these can be compared using \autoref{rema z graded spectra}.}

\begin{lem}\label{lemm associated graded of qn}
In the situation above, We have $\textup{gr}^t(Q_nF^\bullet) \simeq 0$ for $t\geq n$ and $\textup{gr}^t(Q_nF^\bullet) \simeq \textup{gr}^t(F^\bullet)$ otherwise.
\end{lem}
\begin{proof}
Since $\gr$ commutes with cofiber sequences, the result follows by considering the cofiber sequence \eqref{eq cofiber sequence for qn} together with direct observation on the associated graded of $C_{\leq n} F^\bullet$.
\end{proof}

\begin{exam}\label{exam qn and cn of may filtration}
Let $R^{\bullet}\in \Alg(\Tow(\Sp))$. 
Applying the discussion above to the case of $\mathcal{C} = \cycsp$, we have  $Q_n\rthh(R^{\bullet}) \in \mathrm{Tow}(\cycsp)$ and $C_{\leq n}\rthh(R^{\bullet}) \in \mathrm{Tow}(\cycsp)$. 
For the Frobenius, we have the following commuting diagram of $S^1$-spectra.
\begin{equation}\label{eq frobenius of cn and qn}
\begin{tikzcd}
{C}_{\leq n} \rthh(R^{\bullet})\ar[r,"{C}_{\leq n}\varphi_p"]  \ar[d]& \ar[d] C_{\leq n} R_p \rthh(R^{\bullet})^{tC_p} \ar[r]& \ar[d] C_{\leq n} \rthh(R^{\bullet})^{tC_p}\\
\rthh(R^{\bullet}) \ar[r,"\varphi_p"]\ar[d] &  R_p \rthh(R^{\bullet})^{tC_p} \ar[r] \ar[d]& \ar[d] \rthh(R^{\bullet})^{tC_p}\\
Q_n \rthh(R^{\bullet}) \ar[r,"Q_n\varphi_p"]& Q_n  R_p \rthh(R^{\bullet})^{tC_p} \ar[r] & Q_n \rthh(R^{\bullet})^{tC_p} 
\end{tikzcd}
\end{equation}
Here, the vertical composites are cofiber sequences and the unmarked horizontal arrows are induced by the natural transformation from \autoref{natural-trans-R}. Now the bottom horizontal composite gives us a cyclotomic structure map $Q_n \widetilde{\varphi}_p$.
\end{exam}

We will use the following lemma freely without reference. 

\begin{lem}\label{lemma qn of orbit is tc of qn}
Let $R^{\bullet}\in \Alg(\Tow(\Sp))$. Then there is an equivalence 
\[Q_n \filorb \rtc(R^0) \simeq \tc\big(Q_n\rthh(R^{\bullet})\big)\]
where $\tc$ on the right hand side  applies $\tc$ levelwise to $Q_n\rthh(R^{\bullet}) \in \mathrm{Tow}(\cycsp)$.

Similarly, the functors $-^{hS^1}$, $-^{tS^1}$ and $-^{tC_p}$ commute with $Q_n$.
\end{lem}
\begin{proof}
Its clear that the equivalence above holds for $Q_n$ replaced by $C_{\leq n}$, then the result follows by the cofiber sequence in \eqref{eq cofiber sequence for qn} since $\tc(-)\co \cycsp \to \Sp$ commutes with cofiber sequences. The second statement follows in the same manner. 
\end{proof}

The definition of $Q_n$ is justified by the following proposition which states that $Q_n \filorb \rtc(R^0)$ filters the spectrum $Q_n^0 \filorb \rtc(R^0)$  that approximates $\rtc(R^0)$ in sufficiently low  degrees.

\begin{prop}\label{prop partial filtration aprroximate tc}
Let $R^{\bullet}\in \Alg(\Tow(\Sp))$ be strongly $\THH$-convergent and assume that the graded spectrum $\thh(\gr R^{\bullet})$ is $p$-complete. Then the natural map 
\[\pi_s \filorbo\rtc(R^{0}) \to \pi_s Q_n^0 \filorb \rtc(R^0)
\]
is an isomorphism for $s\leq n$ and surjective for $s=n+1$. Since the colimit is realized as the weight zero component there is a natural map
\[\pi_s \rtc(R^0) \to \pi_s Q_n^0 \filorb \rtc(R^0)
\]
that is an isomorphism for $s\leq n$ and surjective for $s=n+1$.
\end{prop}
\begin{proof}
Due to  \autoref{theo orbit filtration} \eqref{Orbit-filt-1}, the weight $0$ component of $Q_n\filorb \rtc(R^0)$ is given by 
\[\rtc(R^0)/\filorbn\rtc(R^0).\]
By \autoref{theo orbit filtration} \eqref{Orbit-filt-4}, $\pi_{s'} \filorbn\rtc(R^0) = 0$ for all $s'\leq n$ which gives the desired result using the cofiber sequence
\[\filorbn\rtc(R^0) \to \rtc(R^0) \to \rtc(R^0)/\filorbn\rtc(R^0)\,.\] 
\end{proof}

\begin{rem}
Although $Q_n \filorb(\rtc(R^0))$ recovers $\rtc(R^0)$ in low degrees for strongly THH-convergent $R^{\bullet}\in \Alg(\mathrm{Tow}(\Sp))$, the same is not generally true for $Q_n \rntc(R^{\bullet})$ and $Q_n \rtp (R^{\bullet})$. For instance,  we observe in \autoref{prop cardinalities of quotients of ntc and ntp} below that the homotopy groups of $Q_n^0 \rntc(\tau_{\geq \bullet} \fp[x_2])$ and  $Q_n^0 \rtp (\tau_{\geq \bullet} \fp[x_2])$ are finite (in each degree) whereas we know that the same does not hold for (the non-graded) $\pis \rntc (D\fp[x_2])$ and $\pis \rtp(D\fp[x_2])$ due to \cite[Proposition 6.3]{bayindir2020kthryofthh}. This is precisely the reason why we need this quotient filtration; it allows us to compute TC using a fiber sequence involving spectra with smaller homotopy groups. 
\end{rem}

\subsection{Surjectivity of Frobenius}

Here, our goal is to prove the surjectivity of the Frobenius map in sufficiently large degrees on the quotient filtrations on $\rntc$ and $\rtp$ in the case of an $\bE_1$-ring with homotopy $\field[x_{2m}]$. This is accomplished in \autoref{coro frobenius is surjective in homotopy}, which follows from the next two results. 

\begin{lemm}\label{lemm weight reducing map is surjective for tp}
Let $A$ be an $\bE_1$-ring with $\pis A \cong \field[x_{2m}]$ where $p \nmid m>0$. Then after applying $-^{hS^1}$ and restricting to weight 0, the horizontal map in the lower right corner of \eqref{eq frobenius of cn and qn} 
\begin{equation}\label{lower right corner of eq frobenius of cn and qn}
Q_n^0R_p \rtp(\wtw A) \to Q_n^0 \rtp(\wtw A)
\end{equation}
is surjective on $\pi_*$.
\end{lemm}
\begin{proof}
By the definitions of $Q_n$ and $R_p$, this map is given by the quotient map  
\[
\rtp(\wtw A)^{0}/\rtp(\wtw A)^{pn} \to \rtp(\wtw A)^{0}/\rtp(\wtw A)^{n} \,.
\]
This map sits in a cofiber sequence 
\[\rtp(\wtw A)^{0}/\rtp(\wtw A)^{pn} \to \rtp(\wtw A)^{0}/\rtp(\wtw A)^{n} \to \Sigma \rtp(\wtw A)^{n}/\rtp(\wtw A)^{pn}.\]
Therefore, it suffices to prove that the left hand and the middle terms above are concentrated in odd degrees and the right hand side above is concentrated in even degrees. 

The oddness of $\rtp(\wtw A)^{0}/\rtp(\wtw A)^{pn} $ follows from the fact that this quotient has a finite filtration given by $Q_{np}\rtp(\wtw A)$ whose associated graded is concentrated in odd degrees due to \autoref{lemm associated graded of qn} and \autoref{theo tc ntc tp of formal dgas} \eqref{item 1 frobenius is iso}. The same argument using  $Q_{n}\rtp(\wtw A)$ instead of $Q_{np}\rtp(\wtw A)$ proves that the middle term in the cofiber sequence above is also concentrated in odd degrees. 

What remains is to prove that $\rtp(\wtw A)^{n}/\rtp(\wtw A)^{np}$ is concentrated in odd degrees. This follows in the same way but by starting with the $\N$-filtration $F^\bullet$ on $\rtp(\wtw A)^{n}$ given by $F^i := \rtp(\wtw A)^{i+n}$ and using the filtration $Q_{np-n}F^{\bullet}$ on $ \rtp(\wtw A)^{n}/\rtp(\wtw A)^{pn}$.

\end{proof}

\begin{nota}\label{Qnvarphi}
Let $Q_n^0\varphi$ denote $(Q_n^0\varphi_p)^{hS^1}$ where $Q_n^0\varphi_p$ is the weight zero component of the map $Q_n\varphi_p$ as in \autoref{exam qn and cn of may filtration}.
\end{nota}
\begin{lem}\label{lemm frobenius is surjective on the quotient}
Let $A$ be an $\bE_1$-ring with $\pis A \cong \field[x_{2m}]$ and $p \nmid m>0$. The map
\begin{equation}\label{map described in notation}
Q_n^0 \varphi \co Q_n^0 \rntc(\tau_{\geq \bullet }A) \to Q_n^0R_p\rtp(\tau_{\geq \bullet }A)
\end{equation}
described in \autoref{Qnvarphi} is an isomorphism in $\pi_k$ for $k\geq n-1$. 
\end{lem}
\begin{proof}

In general, the functor $R_p$ (restriction along $ \N \xrightarrow{\cdot p} \N$) may not commute with taking the associated graded. However, it does in the situation above. This is because the associated graded of $\rtp(\wtw A)$, i.e.\ $\rtp(\field[x_{2m}])$, is trivial in odd weights (\autoref{rema trivial in odd weights}) and also trivial in even weights $2i$ for which $p \nmid i$ by \autoref{theo tc ntc tp of formal dgas} \eqref{item 2 frobenius is iso}. In particular, the structure maps 
\[\rtp(\wtw A)^{i+1} \to \rtp(\wtw A)^i\]
are equivalences except when $i = 2pk$ for some $k>0$. In summary, we deduce
\begin{equation}\label{eq rp commutes with associated graded}
\gr(R_p\rtp(\wtw A)) \simeq R_p\rtp(\field[x_{2m}])\end{equation}
where  $R_p$ on the right hand side denotes restriction along $\N^{\textup{ds}} \xrightarrow{\cdot p} \N^{\textup{ds}}$.

With this in mind, both   $Q_n^0 \rntc(\tau_{\geq \bullet }A)$ and $Q_n^0R_p\rtp(\tau_{\geq \bullet }A)$ have finite $\N$-filtrations given by   $Q_n \rntc(\tau_{\geq \bullet }A)$ and $Q_nR_p\rtp(\tau_{\geq \bullet }A)$ respectively where the associated graded is concentrated in odd degrees due to \autoref{lemm associated graded of qn} and \autoref{theo tc ntc tp of formal dgas} \eqref{item 1 frobenius is iso}. Then the corresponding spectral sequences are concentrated in finitely many filtrations 
and they are concentrated in odd total degrees; we deduce that  both $Q_n^0R_p\rtp(\tau_{\geq \bullet }A)$ and $Q_n^0 \rntc(\tau_{\geq \bullet }A)$ are concentrated in odd degrees. Therefore, we may assume that $k$ is odd. 

In more detail, the map between the $\mathrm{E}_1$-pages 
\begin{equation}\label{map-E1-Q} \pi_*\gr Q_n \rntc(\tau_{\geq \bullet }A)\to \pi_*\gr Q_nR_p\rtp(\tau_{\geq \bullet }A)
\end{equation}
of these spectral sequences is given by the Frobenius 
\begin{equation}\label{eq map of E1pages}
\pi_k\rntc(\field[x_{2m}])^{j} \to \pi_k \rtp(\field[x_{2m}])^{pj}
\end{equation}
at the $(k,j)$-spot where $k$ is total degree and $j$ is weight, 
when $j\leq n-1$ and both source and target of the map \eqref{map-E1-Q} of are trivial for $j \geq n$ by \autoref{lemm associated graded of qn}.
Again due to  \autoref{rema trivial in odd weights} (and \autoref{theo tc ntc tp of formal dgas} \eqref{item 2 frobenius is iso} for $p=2$), both source and target of \eqref{eq map of E1pages} are trivial for odd $j$, so we only need to consider 
\[\pi_k\rntc(\field[x_{2m}])^{2i} \to \pi_k \rtp(\field[x_{2m}])^{2pi}\]
for  $2i\leq n-1$.  Since $k\geq n-1$ by hypothesis we obtain $k\geq 2i$, but since $k$ is odd this means $k\geq 2i+1$. Now this map is an isomorphism for $k \geq 2i+1$ by \autoref{theo tc ntc tp of formal dgas} \eqref{item 3 frobenius is iso}. Therefore, the induced map of spectral sequences is an isomorphism in total degree $k$. Both spectral sequences collapse on the first page and are concentrated in finitely many filtrations, so both source and target of the map of spectral sequences strongly converge and we deduce the desired result.

\end{proof}

\begin{nota}
We let $Q_n\widetilde{\varphi}$ denote $(Q_n\widetilde{\varphi}_p)^{hS^1}$ and similarly let  $Q_n^0\widetilde{\varphi}$ denote the weight zero component $(Q_n^0\widetilde{\varphi}_p)^{hS^1}$; see \autoref{exam qn and cn of may filtration}.
\end{nota}
\begin{prop}\label{coro frobenius is surjective in homotopy}
Let $A$ be an $\bE_1$-ring with $\pis A \cong \field[x_{2m}]$ and $p\nmid m>0$. Then the map 
\[Q_n^0\widetilde{\varphi} \co Q_n^0\rntc(\wtw A) \to Q_n^0 \rtp(\wtw A)\]
is surjective in $\pi_k$ for $k \geq n-1$.
\end{prop}
\begin{proof}
This map is the composite of the maps \eqref{lower right corner of eq frobenius of cn and qn} and 
\eqref{map described in notation} so the result follows from 
\autoref{lemm weight reducing map is surjective for tp} and \autoref{lemm frobenius is surjective on the quotient}.
\end{proof}

\subsection{Cardinality of TC of \texorpdfstring{$\bE_1$-rings}{ring spectra} with polynomial homotopy}

Here, our goal is to prove \autoref{theo even trivial gives odd}.

\begin{prop}\label{prop cardinalities of quotients of ntc and ntp}
Let  $A$ be an  $\bE_1$-ring with $\pi_*A \cong \field[x_{2m}]$ where $p\nmid m>0$ and $\field$ be a finite field of characteristic $p$. In this situation, we have 
\begin{enumerate}
\item $\pi_* Q_n^0 \rntc(\tau_{\geq \bullet}A)$ and $\pi_* Q_n^0 \rtp(\tau_{\geq \bullet}A)$ are concentrated in odd degrees\footnote{The reader may notice that these results were already proven within the proof of \autoref{lemm weight reducing map is surjective for tp} and \autoref{lemm frobenius is surjective on the quotient} however, we include and prove them here in order to improve the flow of the exposition.} and they are finite in each odd degree. 
\item We have the following equalities:
\begin{equation*}
\begin{split}
\lvert \pi_\ell Q_n^0 \rntc(\tau_{\geq \bullet}A)\rvert =& \lvert \pi_\ell Q_n^0 \rntc(\tau_{\geq \bullet} \field[x_{2m}]) \rvert\\
\lvert \pi_\ell Q_n^0 \rtp(\tau_{\geq \bullet}A)\rvert =& \lvert \pi_\ell Q_n^0 \rtp(\tau_{\geq \bullet} \field[x_{2m}]) \rvert \,.
\end{split}
\end{equation*}

\end{enumerate}
\end{prop}
\begin{proof}
We prove the statements on topological negative cyclic homology, the ones on topological periodic homology follow in the same manner. 

By \autoref{lemm associated graded of qn}, $\gr(Q_n \rntc(\tau_{\geq \bullet}A))$ is trivial in weights $\geq n$ and given by
\[
\gr(\rntc(\tau_{\geq \bullet}A))\simeq \rntc(\field[x_{2m}])
\]
in weights $< n$ by \autoref{E1page}. Therefore, 
$\pis \gr(Q_n \rntc(\tau_{\geq \bullet}A))$ 
is concentrated in odd degrees by \autoref{theo tc ntc tp of formal dgas} \eqref{item 1 frobenius is iso}. 
In particular, the corresponding spectral sequence is concentrated in odd total degrees, hence it collapses on the first page, and  this proves the first part of \eqref{item 1 frobenius is iso}. 
Moreover, this spectral sequence is concentrated in finitely many filtrations (due to the definition of $Q_n$) and hence, for odd $\ell$,  one obtains a filtration
\begin{equation}\label{eq filtration on homotopy of qn ntc}
0 \subseteq M^{n-1} \subseteq M^{n-2} \subseteq \cdots \subseteq M^1 \subseteq M^0:= \pi_\ell Q_n^0\big(\rntc(\tau_{\geq \bullet} A)\big)
\end{equation}
of $\pi_\ell Q_n^0\big(\rntc(\tau_{\geq \bullet} A)\big)$. Indeed, this is just $\pi_\ell-$ applied levelwise to the tower $Q_n(\rntc(\wtw A))$. Here, we know that the structure maps $M^{i} \to M^{i-1}$ are inclusions by \autoref{theo tc ntc tp of formal dgas} \eqref{item 1 frobenius is iso}; i.e.\ because the associated graded of $Q_n(\rntc(\wtw A))$ is concentrated in odd degrees, see \autoref{lemm associated graded of qn}. Moreover, we can identify 
\begin{equation}\label{eq quotient of the filtration on homotopy of qn}
M^{i-1}/M^{i}\cong \pi_{\ell} (\rntc(\field[x_{2m}])^{i-1})
\end{equation}
for $i\leq n$ where $\rntc(\field[x_{2m}])^{i-1}$ denotes the $i-1$ weight component of $\rntc(\field[x_{2m}])$. In particular, each quotient $M^{i-1}/M^{i}$ is finite by \autoref{theo tc ntc tp of formal dgas} \eqref{item 1 frobenius is iso}; this finishes the proof of \eqref{item 1 frobenius is iso}. By a decreasing induction over the tower $M^\bullet$, we deduce
\[\lvert \pi_\ell Q_n^0\big(\rntc(\tau_{\geq \bullet} A)\big)\rvert  = \prod_{i \leq n} \lvert \pi_\ell \rntc(\field[x_{2m}])^{i-1}\rvert\,.\]
Replacing $A$ above with $\field[x_{2m}]$, gives
\[\lvert \pi_\ell Q_n^0\big(\rntc(\tau_{\geq \bullet} \field[x_{2m}])\big)\rvert  = \prod_{i \leq n} \lvert \pi_\ell \rntc(\field[x_{2m}])^{i-1}\rvert\,.\]
Since the right hand side of these equalities are the same, we deduce \eqref{item 2 frobenius is iso}.
\end{proof}

We will use the following elementary observation for our cardinality results.

\begin{lem}\label{lemm general counting}
    Let $f \co G \to H$ be a homomorphism between finite abelian groups. Then 
    \[
    \lv \kernel(f) \rv = \frac{\lv G \rv}{\lv H \rv} \lv \coker(f) \rv 
    \]
\end{lem}
\begin{proof}
    There are short exact sequences 
    \[0 \to \kernel (f) \to G \to \im(f) \to 0\]
    and 
    \[0 \to \im(f) \to H \to \coker(f)\to 0.\]
    The first one gives $\lv G \rv = \lv \kernel (f)\rv \lv \im(f)\rv$ and the second gives $\lv H \rv = \lv \im(f) \rv \lv \coker(f) \rv$. 
\end{proof}

\begin{lemm}\label{lemm odd and even tc group cardinality in terms of quotients of quotient ntc and tp}
Let $A$ be an $\bE_1$-ring with $\pis A \cong \field[x_{2m}]$ where $p \nmid m >0$  and $\field$ be a finite field of characteristic $p$. 
Then both $\lvert \pi_{2r+1}\rtc (A) \rvert$ and $\lvert \pi_{2r} \rtc (A)\rvert$ are finite and we have:
\[
\lvert \pi_{2r+1}\rtc (A) \rvert = \frac{\lvert \pi_{2r+1} Q_n^0 \rntc(\tau_{\geq \bullet} \field[x_{2m}]) \rvert}{\lvert \pi_{2r+1} Q_n^0 \rtp(\tau_{\geq \bullet} \field[x_{2m}]) \rvert} \lvert \pi_{2r} \rtc (A)\rvert 
\]
as soon as $2r+1\leq n$.
\end{lemm}
\begin{proof}
Let $n\geq2r+1$, due to \autoref{prop partial filtration aprroximate tc}, it suffices to prove this result for $Q_n^0 \filorb \rtc(A)$ instead of  $\rtc(A)$. By \autoref{lemma qn of orbit is tc of qn}, there is a fiber sequence 
\[Q_n^0 \filorb \rtc(A) \to Q_n^0 \rntc(\tau_{\geq \bullet} A) \xrightarrow{Q_n^0(\widetilde{\varphi}-\can)} Q_n^0 \rtp(\wtw A)\]
for each $A$ as in the statement of the lemma (since $\thh(\pis A)$ is $p$-complete). 

By \autoref{prop cardinalities of quotients of ntc and ntp}, the center and the right hand terms above have homotopy groups concentrated in odd degrees. This provides the following. 
\begin{equation*}
\begin{split}
\pi_{2r+1}\big( Q_n^0 \filorb \rtc(A)\big) =& \textup{ker}\big(\pi_{2r+1}(Q_n^0(\widetilde{\varphi}-\can)\big) \\
\pi_{2r}\big( Q_n^0 \filorb \rtc(A)\big) = &\textup{coker}\big(\pi_{2r+1}(Q_n^0(\widetilde{\varphi}-\can)\big)
\end{split}
\end{equation*}

Again by \autoref{prop cardinalities of quotients of ntc and ntp},  the  homotopy groups of $Q_n^0 \rntc(\tau_{\geq \bullet} A)$ and $Q_n^0 \rtp(\tau_{\geq \bullet} A)$ are degreewise finite. Therefore, we can apply \autoref{lemm general counting} for $\pi_{2r+1}(Q_n^0(\widetilde{\varphi}-\can))$ together with \autoref{prop cardinalities of quotients of ntc and ntp}  to obtain the desired result.
\end{proof}

\theoeventrivialgivesodd*

\begin{proof}

 Applying \autoref{lemm odd and even tc group cardinality in terms of quotients of quotient ntc and tp} to $\field_{q}[x_{2m}]$, we have 
\[\lvert \pi_{2r+1}\rtc (\field_q[x_{2m}]) \rvert = \frac{\lvert \pi_{2r+1} Q_n^0 \rntc(\tau_{\geq \bullet} \field_q[x_{2m}]) \rvert}{\lvert \pi_{2r+1} Q_n^0 \rtp(\tau_{\geq \bullet} \field_q[x_{2m}]) \rvert} \lvert \pi_{2r} \rtc (\field_q[x_{2m}])\rvert \,.\]
By \autoref{theo tc ntc tp of formal dgas} \eqref{item 4 frobenius is iso}, we know that $\lvert \pi_{2r} \rtc (\field_q[x_{2m}])\rvert = 1$ and $\lvert \pi_{2r+ 1}\rtc (\field_q[x_{2m}]) \rvert  = \lvert \mathbb{W}_{\floor*{\frac{r}{m}}}(\field_q)\rvert$. We deduce that 
\[ \lvert \mathbb{W}_{\floor*{\frac{r}{m}}}(\field_q)\rvert=\frac{\lvert \pi_{2r+1} Q_n^0 \rntc(\tau_{\geq \bullet} \field_q[x_{2m}]) \rvert}{\lvert \pi_{2r+1} Q_n^0 \rtp(\tau_{\geq \bullet} \field_q[x_{2m}]) \rvert}\,.\]
Combining this last equality with \autoref{lemm odd and even tc group cardinality in terms of quotients of quotient ntc and tp} gives the desired result.

\end{proof}
The following corollary is immediate from \autoref{theo even trivial gives odd}.  
\begin{coro}
Let $A$ be an $\bE_1$-ring with $\pis A \cong \field[x_{2m}]$ where $\field$ is a finite field of characteristic $p$ and assume $p \nmid m>0$. Then 
\[ \lvert \pi_{2r+1}\rtc(A)\rvert \ge  \lvert \mathbb{W}_{\floor*{\frac{r}{m}}}(\field) \rvert  \,.
\]
\end{coro}

\section{Coefficients in algebraically closed fields}

Here, our goal is to prove \autoref{mainthmevenvanish}. Throughout this section, we let  $A$ denote an $\bE_1$-ring with $\pis A \cong \fp[x_{2m}]$ for  $p\nmid m>0$.

\subsection{Spherical Witt vectors and topological cyclic homology}
In this section, we prove \autoref{prop tc fiber sequence for the witt case} which allows us to study $\rtc(\sphwf\otimes A)$ only using the cyclotomic structure on $\sphwf$ and $\filorb \rtc(A)$.

We first discuss how to produce the structure of a cyclotomic $\bE_\infty$-ring on $\mathbb{S}_{W(\field)}$ in a compatible way with the usual Frobenius map on $W(\field)$. 
\begin{rema}\label{cyclotomic-structure-sphwf}
By \cite[Proposition 6.6]{yuan2023integralmodels} the natural map (with the trivial $S^1$-action on $\sphwf$) 
\[c \co \sphwf \to \sphwf^{hC_p} \to \sphwf^{tC_p}
\] 
is  an equivalence of  $\bE_\infty$-rings in $S^1$-spectra. Let $\phi$ denote the Tate valued Frobenius~\cite[Definition~IV.1.1]{nikolausscholze2018topologicalcyclic} on $\sphwf$ and consider
\[F \co \sphwf \xrightarrow{\phi} \sphwf^{tC_p} \xrightarrow[\simeq]{c^{-1}} \sphwf.\]
By \cite[Example 6.14]{yuan2023integralmodels}, $F$ is the map of $\bE_\infty$-rings induced by the Frobenius of $W(\field)$. Equipping with the trivial action, $F$ lifts to an $S^1$-equivariant map of  $\bE_\infty$-rings and therefore the $\bE_\infty$-ring map   $\phi$ also lifts to an  $S^1$-equivariant map of $\bE_\infty$-rings. This equips $\sphwf$ with the structure of an $\bE_\infty$-algebra in cyclotomic spectra. Abusing notation, we often consider the structure map of this cyclotomic structure as the map $F \co \sphwf \to \sphwf$. 
\end{rema}

The remark above and the following proposition are well known, cf. \cite[Proposition 2.10]{LW22}.
\begin{prop}\label{rema spherical witts as cyclotomic base}
The counit map $\thh(\sphwf) \to \sphwf$ is a map of  $\bE_\infty$ algebras in cyclotomic spectra and it exhibits $\sphwf$ as the $p$-completion of $\thh(\sphwf)$.
\end{prop}

\begin{proof} 
For the first statement, we need to prove that the square of $S^1$-equivariant $\bE_\infty$-algebras within the diagram below commutes (see \autoref{cyclotomic-structure-sphwf}). 
\[
\begin{tikzcd}
\sphwf \arrow{r}   & \THH(\sphwf) \arrow{r} \arrow{d}{\varphi_p} & \sphwf  \arrow[d,"\phi"] \\  
 &  \THH(\sphwf)^{tC_p}  \arrow{r} & \sphwf^{tC_p}  \,
\end{tikzcd}
\]
By the universal property of $\thh$, it suffices to prove that the two composites in this square  agree as maps of $\bE_\infty$-rings  after precomposing with the top left horizontal arrow. Since the top horizontal composite is the identity map, this follows by \cite[Corollary~\RomanNumeralCaps{4}.2.4]{nikolausscholze2018topologicalcyclic}.

To see that the map $\THH(\sphwf) \rightarrow \sphwf $ is a $p$-complete equivalence, it suffices to note that it is an equivalence on mod $p$ homology which is the map $\thh(\field/\mathbb{F}_p)\xrightarrow{\simeq} \field$ that  is known to be an equivalence.
\end{proof}

By left Kan extension along the canonical inclusion $0 \to \mathbb{N}$, one obtains  an $\bE_\infty$-algebra in $\mathbb{N}$-filtered spectra that we also denote by $\sphwf$. Given an $\mathbb{N}$-filtered spectrum $F$, we write $\sphwf\otimes F$ for the Day convolution tensor product of these two $\mathbb{N}$-filtered spectra, which satisfies  $(\sphwf\otimes F)^i=\sphwf\otimes F^i$. We can similarly regard $\sphwf$ as an $\mathbb{E}_\infty$-algebra in graded spectra and given a graded spectrum $G$ we can consider the Day convolution $\sphwf\otimes G$, which satisfies $(\sphwf\otimes G)^i=\sphwf\otimes G^i$. 

\begin{rema}\label{rema base change for p adic filtration on the truncated witt}
As left Kan extension along $0\to \N$ is left adjoint to restriction, and since $S_{\field}^0 = W(\field)/p^n$ receives an $\bE_\infty$-map from $\sphwf$, we deduce by this adjunction that there is an $\bE_\infty$-map  $\sphwf \to S_{\field}^{\bullet}$ in $\Tow(\Sp)$; in particular, $S_{\field}^\bullet$ is an $\bE_\infty$ $\sphwf$-algebra. By direct observation, we see that there is an  equivalence  $\bE_\infty$ $\sphwf$-algebras
\begin{equation}\label{eq base change for p adic filtration}
\sphwf \otimes S_{\fp}^{\bullet} \xrightarrow{\simeq} S_{\field}^\bullet 
\end{equation}
given by extending scalars along the canonical map $S_{\fp}^{\bullet} \to S_{\field}^\bullet$.

\end{rema}
\begin{lemm}
Let $F$ be an $\mathbb{N}$-filtered spectrum. 
Then the canonical map
\begin{equation}\label{eq associated graded and sphwf otimes}
\sphwf \otimes \gr(F) \overset{\simeq}{\longrightarrow} \gr\big(\sphwf \otimes F\big) 
\end{equation}
is an equivalence. When $F$ is a $\mathbb{N}$-filtered twisted cyclotomic spectrum, then this is an equivalence of twisted cyclotomic graded spectra.
\end{lemm}
\begin{proof}
Applying $\sphwf \otimes -$ in $\N$-filtered spectra applies $\sphwf \otimes -$ levelwise which commutes with cofibers and hence with the associated graded. Since $\gr$ is compatible with twisted cyclotomic spectra by~\cite[Proposition~C.5.2]{hahn2020redshift} this is also an equivalence of twisted cyclotomic spectra when $F$ is a twisted cyclotomic spectrum. 
\end{proof}

\begin{lemm}\label{lemm sphwf otimes thhwtwa is pcomplete}
Let $A$ be an $\bE_1$-ring with $\pis A \cong \fp[x_{2m}]$ for some $m>0$. Then the $\N$-filtered spectrum 
\[\sphwf \otimes \rthh(\wtw A)\]
is (levelwise) $p$-complete. 

Similarly, $\sphwf \otimes \rthh(S_{\fp}^\bullet)$ is levelwise $p$-complete where $S_{\fp}^\bullet$ is as in \autoref{nota  p adic filtration on truncated witt}.
\end{lemm}
\begin{proof}
We begin by showing that  $\pi_k \rthh(A)$ has bounded $p$-torsion for each $k$. The first page of the May spectral sequence computing $\pis \rthh(A)$ is given by  $\pis \rthh(\fp[x_{2m}])$, see \autoref{mayss}. There are isomorphisms 
\begin{equation}\label{eq thh of formal dga}
\begin{split}
\pi_*\thh(\fp[x_{2m}]) \cong \ & \pis\big( \thh(\fp) \otimes \thh(\sph[x_{2m}]) \big)\\
\cong \ & \pis \big(\thh(\fp) \otimes_{\fp} (\fp \otimes \thh(\sph[x_{2m}]))\\
\cong \ & \pis \big(\thh(\fp) \otimes_{\fp} \textup{HH}^{\fp}(\fp[x_{2m}])\big)\\
\cong \ & \fp[u_2] \otimes \fp[x_{2m}] \otimes \Lambda (y_{2m+1})\,.
\end{split}
\end{equation}
In the last line above, subscripts denote total degrees and removing the sub $\fp$-vector space $\fp[u_2]$ gives $ \pis \rthh(\fp[x_{2m}])$. In each total degree, this spectral sequence is finite, and hence, we deduce that $\pi_k\rthh(A)$ has bounded $p$-torsion. By the flatness of $\sphwf$, we deduce the same for $\sphwf \otimes \rthh(A)$; in particular,  $\sphwf \otimes \rthh(A)$ is $p$-complete.

By \eqref{eq associated graded and sphwf otimes}, the associated graded of $\sphwf \otimes \rthh(\wtw A)$ is $p$-complete at each weight since it is given by the $\fp$-module $\sphwf \otimes \rthh(\fp[x_{2m}])$.  As we already proved that the weight zero component of $\sphwf \otimes \rthh(\wtw A)$ (given by $\sphwf \otimes \rthh(A)$) is $p$-complete, the $p$-completeness of the associated graded now allows us to inductively deduce that $\sphwf \otimes \rthh(\wtw A)$ is $p$-complete at each weight. 

The proof of the last sentence follows the same lines. To see that $\rthh(S_{\fp}^0)$ has bounded $p$-torsion,  note that the corresponding May spectral sequence has $\mathrm{E}_1$-page $\pis \rthh(\fp[t]/t^n)$ as $\gr S_{\fp}^\bullet = \fp[t]/t^n$. Letting $\Pi_n$ be the pointed topological monoid $\{0,1,t...,t^{n-1}\}$, we have 
\begin{equation}\label{eq splitting for thh of truncated polynomial}
    \begin{split}
        \pis \thh(\fp[t]/t^n) &\cong \pis \big(\thh(\fp\otimes \Sigma^{\infty} \Pi_n)\big) \\
        &\cong \pis\big(\thh(\fp) \otimes \thh(\Sigma^{\infty} \Pi_n)\big)\\
        &\cong \pis\big(\thh(\fp) \otimes_{\fp} ( \fp \otimes \thh(\Sigma^{\infty} \Pi_n)) \big)\\
        &\cong \pis \big(\thh(\fp) \otimes_{\fp} \thh^{\fp}(\fp[t]/t^n) \big).
    \end{split}
\end{equation}

Then by \cite[Proposition 3]{speirs2020truncatedpolynomial} and B\"okstedt periodicity, this is finite for each $*$ and hence $\pis (\thh_{\rel}(S_{\fp}^0))$ has bounded $p$-torsion in each degree and so does $\pis(\sphwf \otimes \rthh(S_{\fp}^0))$; in particular, $\sphwf \otimes \rthh(S_{\fp}^\bullet)$ is $p$-complete in weight $0$. As before, $\gr (\sphwf\otimes \rthh(S_{\fp}^\bullet)) = \sphwf \otimes \rthh(\fp[t]/t^n)$ is an $\fp$-module and hence $p$-complete. Then the result follows by an induction over the tower $\sphwf \otimes \rthh(S_{\fp}^\bullet)$.
\end{proof}

\begin{lemm}\label{lemm twisted cyclotomic moving witt out}
Let $A$ be an $\bE_1$-ring with $\pis A \cong \fp[x_{2m}]$ where $m>0$. Then there is an equivalence of twisted cyclotomic objects:
\[\rthh\big(\tau_{\geq \bullet}(\sphwf \otimes A)\big) \simeq \sphwf \otimes \rthh(\tau_{\geq \bullet }A).\]
Similarly, there is an equivalence of twisted cyclotomic objects 
\[\rthh(S_{\field}^{\bullet}) \simeq \sphwf \otimes \rthh(S_{\fp}^\bullet).\]
\end{lemm}
\begin{proof}
Restricting $\tau_{\geq \bullet}\sphwf$ through $0 \to \N$ gives the $\bE_\infty$-ring $\sphwf$. Since left Kan extension along $0\to \N$ is left adjoint to this restriction, we obtain a map of $\bE_\infty$-algebras in $\N$-filtered spectra $\sphwf \to \tau_{\geq \bullet} \sphwf$.

Using this map, and the lax symmetric monoidal structure of the functor  $\tau_{\geq \bullet}\co \Sp_{\geq 0} \to\Fil(\Sp)$, we obtain a map of $\bE_1$-algebras 
\begin{equation}\label{eq intermediate 10}
\sphwf  \otimes \tau_{\geq \bullet} A\xrightarrow{\simeq} \tau_{\geq \bullet }(\sphwf \otimes A)
\end{equation}
which can be seen to be a levelwise equivalence by induction over these towers since it is clearly an equivalence in weight $0$ and it is an equivalence on the associated graded due to $\eqref{eq associated graded and sphwf otimes}$ and the flatness of $\sphwf$. From this equivalence, we obtain an equivalence of twisted cyclotomic objects in $\N$-filtered spectra $\thh(\tau_{\geq \bullet}(\sphwf \otimes A)) \simeq \thh(\sphwf) \otimes \thh(\wtw A)$. This can be shown to upgrade to the relative setting to give an equivalence of twisted cyclotomic objects:
\begin{equation}\label{eq symmetric monoidal thh}
\rthh(\tau_{\geq \bullet}(\sphwf \otimes A)) \simeq \thh(\sphwf) \otimes \rthh(\wtw A).
\end{equation}

Therefore, it suffices to show that the canonical map of twisted cyclotomic objects 
\[\thh(\sphwf) \otimes \rthh(\wtw A) \xrightarrow{\simeq} \sphwf \otimes \rthh(\wtw A)\]
is an equivalence. Now $\thh(\sphwf) \to \sphwf$ is an equivalence after $p$-completion by  \autoref{rema spherical witts as cyclotomic base}. Therefore, the map above is also an equivalence (levelwise) after $p$-completion and hence, it suffices to prove that both sides above are $p$-complete. The $p$-completeness of the left hand side follows by \autoref{lemm thh is p complete when thh gr is} and \eqref{eq symmetric monoidal thh}. The $p$-completeness of the right hand side follows by \autoref{lemm sphwf otimes thhwtwa is pcomplete}. 

The proof of the last sentence  is similar. We begin with using the equivalence $\sphwf \otimes S_{\fp}^{\bullet} 
\simeq S_{\field}^{\bullet}$ from \eqref{eq base change for p adic filtration} instead of \eqref{eq intermediate 10}. This gives an equivalence of twisted cyclotomic objects $\thh(S_{\field}^\bullet) \simeq \thh(\sphwf) \otimes \thh(S_{\fp}^\bullet)$ which upgrades to an equivalence of twisted cyclotomic objects $\rthh(S_{\field}^{\bullet}) \simeq \thh(\sphwf) \otimes \rthh(S_{\fp}^\bullet)$. Therefore, it suffices to prove that the canonical map 
\[\thh(\sphwf) \otimes \rthh(S_{\fp}^\bullet) \xrightarrow{\simeq} \sphwf \otimes \rthh(S_{\fp}^\bullet)\]
is an equivalence. This map is an equivalence after $p$-completion  and the right hand side is $p$-complete due to \autoref{lemm sphwf otimes thhwtwa is pcomplete}.  The left hand side is $p$-complete by \autoref{lemm thh is p complete when thh gr is} whose hypothesis is satisfied due to \autoref{exam p adic filtration is thh convergent} and the fact that $\gr \thh(S_{\field}^\bullet) \simeq \thh(\field[t]/t^n)$ is an $\fp$-algebra. 
\end{proof}

Given $R^{\bullet}\in \Alg(\Tow(\Sp))$, there are assembly maps 
\begin{equation}\label{eq moving witts out of fixed points}
\begin{split}
\textup{ass}^{h}  : \sphwf \otimes \rthh(R^{\bullet})^{hS^1} \to &\big(\sphwf \otimes \rthh(R^{\bullet})\big)^{hS^1} \textup{\ \ and \ } \\ 
\textup{ass}^{t}  : \sphwf \otimes \rthh(R^{\bullet})^{tS^1} \to &\big(\sphwf \otimes \rthh(R^{\bullet})\big)^{tS^1}
\end{split}
\end{equation}
that comes from the lax symmetric monoidal structures of $-^{hS^1}$ and $-^{tS^1}$ and the $\mathbb{E}_\infty$-ring maps $\sphwf \to \sphwf^{hS^1}$ and $\sphwf \to \sphwf^{tS^1}$. Moreover, there are commuting diagrams 
\[
\begin{tikzcd}
\sphwf \otimes \rthh(R^{\bullet})^{hS^1} \arrow[r,"\textup{ass}^{h}" ] \arrow[d,"\textup{id} \otimes \textup{can} " ] &  \big(\sphwf \otimes \rthh(R^{\bullet})\big)^{hS^1}  \arrow[d,"\textup{can}"]\\ 
\sphwf \otimes \rthh(R^{\bullet})^{tS^1} \arrow[r,"\textup{ass}^{t}" ] &  \big(\sphwf \otimes \rthh(R^{\bullet})\big)^{tS^1} 
\end{tikzcd}
\]
and 
\[
\begin{tikzcd}
\sphwf \otimes \rthh(R^{\bullet})^{hS^1} \arrow[r,"\textup{ass}^{h}" ] \arrow[d,"F \otimes \widetilde{\varphi} " ] &  \big(\sphwf \otimes \rthh(R^{\bullet})\big)^{hS^1}  \arrow[d,"\widetilde{\varphi}"]\\ 
\sphwf \otimes \rthh(R^{\bullet})^{tS^1} \arrow[r,"\textup{ass}^{t}" ] &  \big(\sphwf \otimes \rthh(R^{\bullet})\big)^{tS^1}  
\end{tikzcd}
\]
where $F$ is defined as in \autoref{cyclotomic-structure-sphwf}. 

\begin{lemm}\label{lemm moving witt out of fixed points}
Let $A$ be an $\bE_1$-ring with $\pis A\cong \fp[x_{2m}]$ for some $p\nmid m>0$. Then for each $e$, the assembly maps in \eqref{eq moving witts out of fixed points} for $R^{\bullet}=\tau_{\ge \bullet}A$ induce equivalences after applying $Q_e^0$. 

The same statement holds for  $R^\bullet = S_{\fp}^\bullet$, see \autoref{nota  p adic filtration on truncated witt}.
\end{lemm}

\begin{proof}
We only prove this for the first map in $\eqref{eq moving witts out of fixed points}$, the proof for the second map follows in the same way. Recall from \eqref{eq associated graded and sphwf otimes} that $\sphwf \otimes -$ commutes with the associated graded and that homotopy fixed points also commute with the associated graded functor since it commutes with cofibers. From this, we deduce that the associated graded of $\textup{ass}^h$ is given by the map 
\[\gr(\textup{ass}^h)\co \sphwf \otimes \rthh(\fp[x_{2m}])^{hS^1} \xrightarrow{\simeq} \big(\sphwf \otimes \rthh(\fp[x_{2m}]) \big)^{hS^1}\]
in \autoref{rema passage to witt vector coefficients for formal} (by using \eqref{eq for relthh taking witt out in the formal dga}) which is an equivalence due to \autoref{theo tc ntc tp of formal dgas} \eqref{item 2half frobenius is iso}. 

By \autoref{lemm associated graded of qn}, $\gr(Q_e \textup{ass}^h)$ is a weight truncation of $\gr(\textup{ass}^h)$ and hence, $\gr(Q_e \textup{ass}^h)$ is also an equivalence. Since $Q_e \textup{ass}^h$ is a finite filtration of $Q_e^0 \textup{ass}^h$, we deduce (by induction over the tower $Q_e \textup{ass}^h$) that $Q_e^0 \textup{ass}^h$ is an equivalence. 

For the last sentence, the result again boils down to showing that $\gr \textup{ass}^h$  is an equivalence. As mentioned earlier, $\gr \rthh(S_{\fp}^\bullet) = \rthh(\fp[t]/t^n)$ and hence this follows by \cite[Propositions 12 and 13]{speirs2020truncatedpolynomial}. 
\end{proof}

\begin{prop}\label{prop tc fiber sequence for the witt case}
Let $A$ be an $\bE_1$-ring with $\pis A\cong \fp[x_{2m}]$ for $p\nmid m >0$. Then for each $e\geq 0$, there is a fiber sequence:
\begin{equation}\label{Qn0tc}
Q_e^0 \filorb \rtc(\sphwf \otimes A) \to \sphwf \otimes Q_e^0\rntc(\wtw A) \xrightarrow{F \otimes Q_e^0 \widetilde{\varphi} - \textup{id} \otimes \textup{can}} \sphwf\otimes  Q_e^0 \rtp(\wtw A) \,.
\end{equation}

Similarly, there is a fiber sequence 
\begin{equation*}
Q_e^0 \filorb \rtc(S_{\field}^0) \to \sphwf \otimes Q_e^0\rntc(S_{\fp}^\bullet) \xrightarrow{F \otimes Q_e^0 \widetilde{\varphi} - \textup{id} \otimes \textup{can}} \sphwf\otimes  Q_e^0 \rtp(S_{\fp}^\bullet) \,.
\end{equation*}
\end{prop}

\begin{proof}
We begin with the first statement. Since $\sphwf \otimes -$ is applied levelwise in $\N$-filtered spectra and since it commutes with cofibers, it also commutes with the functor $Q_e$. Therefore, the map $F \otimes Q_e^0\widetilde{\varphi} - \textup{id} \otimes \textup{can}$ above is given by a map
\[Q_e^0\big( \sphwf \otimes \rntc(\wtw A) \big)\xrightarrow{}  Q_e^0 \big(\sphwf\otimes  \rtp(\wtw A)\big).\]
By \autoref{lemm moving witt out of fixed points} (together with the discussion before) and  \autoref{lemm twisted cyclotomic moving witt out}, this boils down to \autoref{lemma qn of orbit is tc of qn}.

The second statement follows in the same way by using the equivalence $S_{\field}^\bullet\simeq \sphwf \otimes S_{\fp}^\bullet$ from \eqref{eq base change for p adic filtration}.
\end{proof}

\subsection{Algebraic lemmas}

The proof of Theorem \ref{mainthmevenvanish} \eqref{item 1 thmevenvanishing} boils down to proving that for algebraically closed $\field$, the second map in the composite \eqref{Qn0tc} of \autoref{prop tc fiber sequence for the witt case} is surjective in $\pi_{k}$ for each odd $k$  and $e$ such that $k \leq e$. For this, we prove here a few purely algebraic lemmas. \autoref{lemm H  is surjective before modding by p} is the key lemma that we will use later.  

\begin{nota}
We write $F$ for each of the Frobenius maps $ W(\field) \to W(\field)$ and $\field \to \field$. It will be clear from the context which one is meant. 
\end{nota}

\begin{rema}
Since we work in the $\infty$-category of spectra, all our tensor products denote the corresponding derived tensor product. On the other hand, when we write $W(\field) \otimes_{\z} M $ for a discrete $\z$-module $M$, we note that this agrees with the corresponding underived tensor product since $W(\field)$ is flat. Similarly, for discrete $\fp$-modules, we do not need to distinguish between the derived and the underived versions for the tensor product $\otimes_{\fp} $. 
\end{rema}

\begin{lemm}\label{lemm H mod p is surjective in fp closure}
Let  $\phi\co W \to V$ be a surjective map between finite discrete $\fp$-modules and $c\co W\to V$ be a map of $\fp$-modules. Then the map 
\[W(\field) \otimes_{\mathbb{Z}} W \xrightarrow{F \otimes \phi - \textup{id} \otimes c} W(\field) \otimes_{\mathbb{Z}}  V\]
is surjective whenever $\field$ is an algebraically closed field of characteristic $p$. 
\end{lemm}

\begin{proof}
Let $V'\subset W$ be a subspace  so that the restriction $\phi_{\mid V'}$ is an isomorphism. Such a subspace exists since $W\to V$ is a surjective map of finite discrete $\mathbb{F}_p$-modules and therefore it is split surjective. 
With this, we consider
\[H \co W(\field) \otimes_{\mathbb{Z}} V' \xrightarrow{F \otimes \textup{id} - \textup{id} \otimes c'}W(\field)  \otimes_{\mathbb{Z}} V'\]
where $c'=\phi_{|{V'}}^{-1}\circ c_{|{V'}}$ and $H=(\textup{id}\otimes \phi_{|V'}^{-1})\circ (F\otimes\phi_{|{V'}}-\textup{id}\otimes c_{|{V'}})$. By the commuting diagram 
\[
\begin{tikzcd}
W(\field) \otimes_{\mathbb{Z}} V' \ar[r,"\textup{id}\otimes \textup{inc}"] \ar[d,"H",swap] & W(\field) \otimes_{\mathbb{Z}} W \ar[d,"F\otimes \phi -\textup{id}\otimes c"] \\ 
W(\field) \otimes_{\mathbb{Z}} V'\ar[r,"\textup{id}\otimes \phi_{|{V'}}"]\ar[r,swap,"\cong"] &  W(\field) \otimes_{\mathbb{Z}} V 
\end{tikzcd}
\]
it suffices to show $H$ is a surjection. 

Since $V'$ is an $\fp$-module, this map is also written as 
\[H \co \field \ot_{\fp} V' \xrightarrow{F \otimes \textup{id} - \textup{id} \otimes c'}\field \ot_{\fp} V'\]
where $F\co \field \to \field$ above also denotes the Frobenius. To prove that this map is surjective, we pick a basis for $V'$ which gives an $\field$-basis for $\field\otimes_{\fp} V'$ for which $F \otimes \textup{id}$ carries a vector $(x_1,\cdots,x_k)$ to $(x_1^p,\cdots,x_k^p)$. As it is $\field$-linear, $\textup{id}\otimes c$ carries $(x_1,\cdots,x_k)$ to $(g_1,\cdots,g_k)$ for some linear polynomials $g_i \in \field[x_1,\cdots,x_k]$. In particular, 
\[H(x_1,\cdots,x_k) = (x_1^p-g_1(x_1,\cdots,x_k),\cdots,x_k^p-g_k(x_1,\cdots,x_k)).\]
Given a vector $(c_1,\cdots,c_k)$ in the target of $H$, showing that $H$ hits $(c_1,\cdots,c_k)$ boils down to the solution of the system of polynomial equations given in \autoref{lemm system of polynomials solution} by letting $f_i = g_i+c_i$ which proves that $(c_1,\cdots,c_k)$ is in the image of $H$. Therefore, $H$ is surjective as desired.
\end{proof}

\begin{lem}\label{lemm system of polynomials solution}
Let $\field$ be an algebraically closed field of characteristic $p$ and let $f_i(x_1,\cdots,x_k)$ for $1\leq i\leq k$ be degree 1 polynomials in $\field$. Then the system equations
\[x_i^p-f_i(x_1,\cdots,x_k) = 0\]
over $1\leq i\leq k$ have at least one solution.
\end{lem}

We expect that this is known, but we nevertheless give a direct proof.
\begin{proof}
Let $g_i(-)$ for $1\leq i\leq k$ be polynomials in single variable such that $\textup{deg}(g_k)>0$ and   $\textup{deg}(g_i)< \textup{deg}(g_k)$ for each $i <k$. We inductively prove the stronger statement that the system of equations 
\begin{equation}\label{eq system of equations}
x_i^p + g_i(x_k^p) + f_i(x_1,\cdots,x_k) = 0
\end{equation}
over $i\leq k-1$ and 
\begin{equation}\label{eq last of the system of equations}
g_k(x_k^p) + f_k(x_1,\cdots,x_k) = 0
\end{equation}
admits a solution.

The base case of $k=1$ follows since $\field$ is algebraically closed. Now assume such systems have solutions for all $k'<k$ and we prove this for $k$. 

If $f_k(x_1,..,x_k) = c_1x_k + c_2$ for some $c_i \in \field$, then the last equation is a single variable polynomial in $x_k$, which has a solution. Plugging the solution to the previous equations, one deduce to the induction hypothesis. 

Therefore, assume (without loss  of generality) that a unit multiple of $x_{k-1}$ appears as a nontrivial summand  in $f_k(x_1,\cdots,x_k)$. Then we can solve for $x_{k-1}$  in \eqref{eq last of the system of equations} and plug in to the equations in \eqref{eq system of equations}. After this, one uses \eqref{eq system of equations} for $i<k-1$ to deduce the $i= k-1$ case of \eqref{eq system of equations} to a form compatible with the induction hypothesis. Indeed, this gives  a system of $k-1$ equations compatible with the induction hypothesis which finishes the proof.
\end{proof}

\begin{lemm}\label{lemm H  is surjective before modding by p}
Let  $\phi\co M \to N$ be a surjective map between finite discrete $\zpl$-modules and $c\co M\to N$ be a map of $\zpl$-modules. Then the map of discrete $\zpl$-modules
\[
W(\field)  \otimes_{\mathbb{Z}} M \xrightarrow{F \otimes \phi - \textup{id} \otimes c} W(\field)  \otimes_{\mathbb{Z}}N
\]
is surjective whenever $\field$ is an algebraically closed field of characteristic $p$.  Here, $F \co W(\field) \to W(\field)$ denotes the Frobenius map.
\end{lemm}

\begin{rema}
For the proof of this lemma, one might consider applying Nakayama's lemma for $(p) \subset W(\field)$ to reduce to the case of \autoref{lemm H mod p is surjective in fp closure}. However,  Nakayama's lemma  doesn't apply immediately in this case since $F \otimes \phi-\textup{id}\otimes c$ may not be a map of $W(\field)$-modules. 
\end{rema}
\begin{proof}[Proof of \autoref{lemm H  is surjective before modding by p}]
For this, we use the $p$-adic filtration on $M$
\[
0 \subset p^\ell M \subset p^{\ell-1}M \subset \cdots \subset pM \subset M
\]
which eventually terminates on the left hand side as $M$ is a finite $\zpl$-module (hence it only has $p$-torsion). We have the same situation for the $p$-adic filtration of $N$. Note that $c$ and $\phi$ respect these $p$-adic filtrations (as they are maps of abelian groups) and moreover, it is a straightforward observation that the restricted and corestricted map 
\begin{equation}\label{eq phi is surjective on the padic filtration}
\phi \co p^t M \to p^t N
\end{equation}
is also surjective (for each $t \geq 0$) as $\phi \co M \to N$ is surjective. Also writing $c$ for the map $c\co p^t M\to p^tN$, we have maps:
\[W(\field)  \otimes_{\mathbb{Z}} p^t M \xrightarrow{F \otimes \phi - \textup{id} \otimes c} W(\field)  \otimes_{\mathbb{Z}} p^tN,\]
filtering the map $F \otimes \phi - \textup{id} \otimes c$ given in the lemma.

Since the $p$-adic towers of $M$ and $N$ terminate, we do a decreasing induction on $t$ showing that the map above is surjective for each $t\geq 0$. The base case of the induction  follows by the fact that these filtrations are finite. Now assume that the map above is surjective for $t$ and we need to prove it for $t-1$. We have the following commuting diagram
\begin{equation*}
\begin{tikzcd}
W(\field) \otimes_{\mathbb{Z}}p^t M \ar[d,"F\otimes_{\mathbb{Z}} \phi - \textup{id} \otimes c"] \ar[r] & W(\field) \otimes_{\mathbb{Z}} p^{t-1} M \ar[r] \ar[d,"F \otimes_{\mathbb{Z}}\phi - \textup{id} \otimes c"]  & W(\field) \otimes_{\mathbb{Z}} p^{t-1}M /(p^{t} M) \ar[d,"F \otimes_{\mathbb{Z}} \phi' - \textup{id} \otimes c'"]\ar[r] & 0 \\
W(\field) \otimes_{\mathbb{Z}} p^t N\ar[r] &W(\field) \otimes_{\mathbb{Z}} p^{t-1} N \ar[r] & W(\field) \otimes_{\mathbb{Z}}  p^{t-1}N /(p^{t} N) \ar[r] & 0
\end{tikzcd}
\end{equation*}
where the quotients on the right hand column are the non-derived quotients and $\phi'$ and $c'$ are the induced maps on quotients. Since $W(\field)\otimes_{\mathbb{Z}}-$ is right exact, we get that the  horizontal rows are exact. The induced map $\phi' \co p^{t-1}M/p^tM \to p^{t-1}N/p^tN$ is also surjective (since $\phi \co p^{t-1}M \to p^{t-1}N$ is surjective, see \eqref{eq phi is surjective on the padic filtration}) and it is a map between finite $\fp$-modules. Therefore, \autoref{lemm H mod p is surjective in fp closure}  applies to prove that the right vertical arrow is surjective. The left vertical arrow is surjective due to the induction hypothesis. Now a simple diagram chase proves the surjectivity of the middle vertical map by using the exactness of the rows and the surjectivity of the left and  right  vertical arrows. Or alternatively, since $W(\field)$ is flat,  the rows above are short exact sequences and we can use the snake lemma to deduce the surjectivity of the middle vertical arrow. This finishes the induction whose $t=0$ case is  the desired result.
\end{proof}

\subsection{Finalizing the proof of \autoref{mainthmevenvanish}} 
We are now prepared to prove \autoref{mainthmevenvanish}.
\mainthmevenvanishing*

\begin{proof}
We use~\cite[Theorem~7.2.2.1]{DGM13} and prove the same statement for $\rtc$. We begin with proving \eqref{item 1 thmevenvanishing}. For a given even integer $\ell>0$, we need to show that $\pi_{\ell}\rtc(\sphwf \otimes A) \cong 0$. By \autoref{prop partial filtration aprroximate tc}, it suffices to prove that 
\[\pi_{\ell} Q_{\ell}^0 \filorb\rtc(\sphwf \otimes A) \cong 0. \]

\autoref{prop tc fiber sequence for the witt case} gives a fiber sequence of spectra:
\[Q_{\ell}^0 \filorb \rtc(\sphwf \otimes A) \to \sphwf \otimes Q_{\ell}^0\rntc(\wtw A) \xrightarrow{F \otimes Q_{\ell}^0 \widetilde{\varphi} - \textup{id} \otimes \textup{can}} \sphwf\otimes  Q_{\ell}^0 \rtp(\wtw A).\]

By the flatness of $\sphwf$ and  \autoref{prop cardinalities of quotients of ntc and ntp}, the middle and the right hand terms above are concentrated in odd homotopy groups. Therefore, it suffices to prove that  the second map above is surjective in $\pi_{\ell+1}$ since $\ell$ is even. Again by the flatness of $\sphwf$, this map is given by 
\[
W(\field) \otimes_\z \pi_{\ell+1} Q_{\ell}^0\rntc(\wtw A) \xrightarrow{F \otimes \pi_{\ell+1} Q_{\ell}^0 \widetilde{\varphi} -\textup{id} \otimes \pi_{\ell+1} \textup{can}} W(\field) \otimes_\z \pi_{\ell + 1}Q_{\ell}^0\rtp(\wtw A) \,.
\]

By \autoref{coro frobenius is surjective in homotopy},  $ Q_{\ell}^0 \widetilde{\varphi}$ is surjective on $\pi_{\ell+1}$. Moreover,  $Q_{\ell}^0\rntc(\wtw A) $ and $ Q_{\ell}^0 \rtp(\wtw A)$  have degreewise finite homotopy groups due to \autoref{prop cardinalities of quotients of ntc and ntp} and they are $p$-local since $A$ is. Therefore, we can apply \autoref{lemm H  is surjective before modding by p} to deduce that the map above is surjective as desired.  This finishes the proof of \eqref{item 1 thmevenvanishing}.

Now we prove the statement in \eqref{item 2 thmevenvanishing} on the cardinalities of the relative algebraic K-theory groups. For $k< m$, the vanishing of these groups follow by the fact that algebraic K-theory increases connectivity by $1$ by Waldhausen~\cite{Wal78} (cf.~\cite[Lemma~2.4]{LT19}). 

For the case $k\geq m$, let $K \to K'$ be a ring map between perfect fields of characteristic $p$. We consider the following  map of fiber sequences (\autoref{prop tc fiber sequence for the witt case}):

\begin{equation}\label{eq proving odd groups are infinite}
\adjustbox{scale=.9,center}{
\begin{tikzcd}
Q_{2k+1}^0\filorb\rtc(\sph_{W(K)}\otimes A) \ar[r] \ar[d] &\sph_{W(K)} \otimes Q_{2k+1}^0 \rntc(\wtw A) \ar[r] \ar[d]& \sph_{W(K)} \otimes Q_{2k+1}^0 \rtp(\wtw A) \ar[d]\\
Q_{2k+1}^0\filorb\rtc(\sph_{W(K')}\otimes A) \ar[r] & \sph_{W(K')} \otimes Q_{2k+1}^0 \rntc(\wtw A) \ar[r] & \sph_{W(K')} \otimes Q_{2k+1}^0 \rtp(\wtw A)  \,.
\end{tikzcd}
}
\end{equation}

Applying $\pi_{2k+1}$ to this diagram, and letting $M:= \pi_{2k+1 }Q_{2k+1}^0 \rntc(\wtw A)$ and letting $N:= \pi_{2k+1 }Q_{2k+1}^0 \rtp(\wtw A)$, we obtain another commuting diagram
\begin{equation}\label{eq proving odd groups are infinite 2} 
\begin{tikzcd}
\pi_{2k+1}\rtc(\sph_{W(K)} \otimes A) 
\ar[d,hook] \ar[r,hook]& W(K) \otimes_\z M \ar[r] \ar[d,hook]&  W(K) \otimes_\z N\ar[d] \\
\pi_{2k+1}\rtc(\sph_{W(K')} \otimes A) 
 \ar[r,hook] &W(K') \otimes_\z M \ar[r]  &W(K') \otimes_\z N  
\end{tikzcd}
\end{equation}
where the left hand side is as given above due to  \autoref{prop partial filtration aprroximate tc}. Since the middle and the right hand terms in \eqref{eq proving odd groups are infinite} are concentrated in odd homotopy groups by \autoref{prop cardinalities of quotients of ntc and ntp}, we deduce that the left horizontal maps are the inclusions of the kernels of the respective right hand horizontal maps.  

Now $M$ and $N$ are finite due to \autoref{prop cardinalities of quotients of ntc and ntp} and $p$-local since $A$ is. By the classification of finite abelian groups, we deduce that the   middle vertical arrow is a coproduct of the inclusions $W(K)/p^l \to W(K')/p^l$ which implies that the middle vertical  arrow is injective. From this, we deduce that the left hand vertical arrow is also injective. 

If $K$ is finite, then there is a map $K\to \field$ since $\field$ is algebraically closed. By the injectivity of the left vertical map in \eqref{eq proving odd groups are infinite 2} for the case $K' = \field$, we deduce:
\[\lv \pi_{2k+1}\rtc(\sph_{W(K)} \otimes A)\rv \leq \lv \pi_{2k+1}\rtc(\sph_{W(\field)} \otimes A)\rv.\]
 By \autoref{theo even trivial gives odd}, the left hand side above can be chosen to be  arbitrarily  large  proving the infiniteness of  \eqref{item 2 thmevenvanishing}.

 For the last statement of \eqref{item 2 thmevenvanishing}, we use the orbit spectral sequence for $\rtc(\sphwf \otimes A)$ from \autoref{theo orbit filtration} \eqref{Orbit-filt-3} noting $\thh_{\rel}(\pis (\sphwf \otimes A))$ is $p$-complete as it is an $\field$-module. For this, we first consider the $S^1$-homotopy orbit spectral sequence computing $\rtc^+(\pis (\sphwf \otimes A))$. This is a first quadrant spectral sequence whose first page is given by $\pis \big(\rthh(\pis (\sphwf \otimes A))\big)[z]$ with $z$ in total degree $2$. Each entry of the first page of this spectral sequence is a   $p$-group  with bounded $p$-torsion. We deduce that the same is true for the $\mathrm{E}_\infty$-page as this is invariant under taking kernels and cokernels. Hence, $\pi_k\rtc^+(\pis (\sphwf \otimes A))$  has a finite filtration with quotients given by $p$-groups with bounded $p$-torsion. It follows by an induction over this finite filtration that $\pi_k\rtc^+(\pis (\sphwf \otimes A))$ is also a $p$-group with bounded $p$-torsion. 

 Now applying the same argument for the orbit spectral sequence for $\rtc(\sphwf \otimes A)$ gives the desired result as this is also a first quadrant spectral sequence and as we demonstrated above, each entry of its first page is a $p$-group with bounded $p$-torsion.
\end{proof}

For $\field = \fpc$, we have control over the cardinalities of the infinite groups above. 

\begin{coro}\label{coro countable relative k theory groups}
Let $A$ be an $\bE_1$-ring with $\pis A \cong \fp[x_{2m}]$ where $p \nmid m >0$. Then  each \[\pi_{k}\mathrm{K}_{\rel}(\sph_{W(\fpc)} \otimes A)\] is countable.
\end{coro}
\begin{proof}
This is similar to the last part of the proof of \autoref{mainthmevenvanish}.  One observes that the $\mathrm{E}_2$-page of the $S^1$-homotopy orbit spectral sequence given by $\pis \big(\rthh(\pis (\sphwf \otimes A))\big)[z]$ with $z$ in total degree $2$  consists of entries that are countable by \eqref{eq thh of formal dga}\footnote{note that \eqref{eq thh of formal dga} also applies to $\fpc$ in place of $\fp$}, so does the $\mathrm{E}_\infty$-page giving a finite filtration of each homotopy group $\pi_k\rtc^{+}(\pis (\sphwf \otimes A))$ with countable quotients. An induction over this finite filtration shows that these homotopy groups are also countable. Now the same argument applies with the orbit spectral sequence from \autoref{theo orbit filtration} computing the groups $\pi_k \rtc(\sphwf \otimes A)$. The result then follows from~\cite[Theorem~7.0.0.2]{DGM13}. 
\end{proof}

In the case $\field = \fpc$, we can use Quillen's computation of $\mathrm{K}(\fpc)$ to deduce the cardinalities of the following absolute K-theory groups.

\begin{coro}\label{coro absolute k theory is odd}
Let $A$ be an $\bE_1$-ring with $\pis A \cong \fp[x_{2m}]$ where $p \nmid m >0$. Then $\mathrm{K}_*(\sph_{W(\fpc)} \otimes A)$ is  trivial for even  $*\geq 2$ and countably infinite with bounded $p$-torsion for odd $* \geq 1$ and $*=0$.  
\end{coro}
\begin{proof}
The first claim follows by considering the fiber sequence 
\[\mathrm{K}_{\rel}(\sph_{W(\fpc)}\otimes A) \to \mathrm{K}(\sph_{W(\fpc)}\otimes A) \to  \mathrm{K}(\fpc) \]
since both extremes are trivial in even positive degrees due to \autoref{mainthmevenvanish}  and \cite[Section~12]{Qui72}. This also means applying $\pi_k-$ to the sequence above for odd $k$, one obtains a short exact sequence. Both extremes of this short exact sequence are countable due to  \cite[Section~12]{Qui72} and \autoref{coro countable relative k theory groups}; this proves that the central term is also countable.  Moreover, $\mathrm{K}_*(\fpc)$ is  $p$-torsion free in each positive odd degree and vanishes $p$-locally. 
Since $\mathrm{K}_{\rel}(\sph_{W(\fpc)}\otimes A)\simeq \mathrm{TC}_{\rel}(\sph_{W(\fpc)}\otimes A)$, we know that it is $p$-local this implies that the short exact sequence in $\pi_k-$ splits. Together with \autoref{mainthmevenvanish} \eqref{item 2 thmevenvanishing}, this implies the boundedness of $p$-torsion. 
\end{proof}

The following corollary is immediate. 

\cardoverFpbar*

\section{On the algebraic K-theory of truncated witt vectors}

The goal of this section is to prove \autoref{theo k-theory of truncated witts is odd}. We first discuss results related to the the algebraic K-theory truncated polynomial algebras,  which follow from work of Speirs~\cite{speirs2020truncatedpolynomial}. We then adapt some of results from the previous section to the case of the $p$-adic filtration on $W(\field)/p^n$ whose associated graded is the truncated polynomial algebra $\field[t]/t^n$ in order to prove \autoref{theo k-theory of truncated witts is odd}. 

\subsection{Recollections on the algebraic K-theory of the truncated polynomial algebra}
The computation of the algebraic K-theory of $\field[t]/t^n$ is due to Hesselholt--Madsen \cite{hesselholt1997kthryoftruncated}; however, here we closely follow the more recent approach to this computation by Speirs \cite{speirs2020truncatedpolynomial}.

Let $n>1$ be an integer. We regard $\field[t]/t^n$ as an $\mathbb{N}$-graded ring with $t$ in weight $1$ and degree $0$, which results in  $\N$-gradings of its $\thh$, $\ntc$ and $\tp$. Consequently, following  \autoref{nota relative versions for filtered rings}, we define
\[\rthh(\field[t]/t^n):= \textup{fib} \big(\thh(\field[t]/t^n)\to \thh(\field)\big)\] 
which simply removes the weight zero component of $\thh(\field[t]/t^n)$. The same applies for $\rntc$ and $\rtp$. 

Using the pointed monoid $\Pi_n:= \{0,1,t,...,t^{n-1}\}$ we have $\field[t]/t^n  = \field[\Pi_n]$.  Following the notation of \cite{speirs2020truncatedpolynomial}, we let $m$ denote the weight and let 
\[d:= \Big\lfloor{\frac{m-1}{n}\Big\rfloor}\,. \]

\begin{prop}\label{prop speirs can iso}
The map 
\[\pi_{2r+1} \rntc(\field[t]/t^n)^m \xrightarrow{\pi_{2r+1} \textup{can}} \pi_{2r+1} \rtp(\field[t]/t^n)^m\]
is an isomorphism for 
\[d>r \textup{\ \ if \  } n \nmid m\]
and it is an isomorphism for 
\[d \geq r  \textup{\ \ if  \ } n \mid m\,.\]
\end{prop}
\begin{proof}
For $n\nmid m$, it follows from \cite[Lemma 4]{speirs2020truncatedpolynomial} that $\pis \thh^\z(\z[\Pi_n])^m$ is given by $\z$ in degrees $2d$ and $2d+1$ and trivial elsewhere. Therefore, 
\begin{equation}\label{chain-of-equiv}
\thh(\field[\Pi_n])^m \simeq \thh(\field) \otimes \thh(\sph[\Pi_n])^m \simeq \thh(
\field) \otimes_{\z}\thh^\z(\z[\Pi_n])^m
\end{equation}
is trivial on $\pis$ for $*<2d$. Then  $\Sigma \tc^+(\field[\Pi_n])^m$, the fiber of $\textup{can}$, is trivial on $\pis $ for $*<2d+1$; which proves the $n \nmid m$ case. 

For $n \mid m$, $\pi_*(\rthh^\z(\z[\Pi_n])^m)$ is concentrated in degree $2d+1$ where it is given by $\z/n$ by \cite[Lemma 4]{speirs2020truncatedpolynomial}. By the chain of equivalences \eqref{chain-of-equiv}, one sees that $\pis \rthh(\field[\Pi_n])^m$ is trivial for $*\leq 2d$ and this gives the desired result in the same way.
\end{proof}

\begin{prop}[{\cite[Corollary 9]{speirs2020truncatedpolynomial}}]\label{prop speirs frobenius iso}
The Frobenius map 
\[\pi_{2r+1} \rntc(\field[t]/t^n)^m \xrightarrow{\pi_{2r+1} \varphi} \pi_{2r+1} \rtp(\field[t]/t^n)^{pm}\]
is an isomorphism for 
\[d \leq r \textup{\ \ if  \ } n \nmid m\]
and an isomorphism for 
\[d < r \textup{\ \ if \  } n \mid m \,.\]
\end{prop}

\begin{coro}\label{coro phi and can for truncated witt}
Let $r \geq 0$, then the canonical map 
\[\pi_{2r+1} \rntc(\field[t]/t^n)^m \xrightarrow{\pi_{2r+1} \textup{can}} \pi_{2r+1} \rtp(\field[t]/t^n)^m\]
is an isomorphism for $m \geq n(r+1)$ and the Frobenius
\[\pi_{2r+1} \rntc(\field[t]/t^n)^m \xrightarrow{\pi_{2r+1} \varphi} \pi_{2r+1} \rtp(\field[t]/t^n)^{pm}\]
is an isomorphism whenever $pm <n(r+1)$.
\end{coro}
\begin{proof}
For the first isomorphism, note that for $m = n(r+1)$, $n \mid m$ and $d= r$ so the isomorphism follows by \autoref{prop speirs can iso}. For $m > n(r+1)$, we have $d>r$ hence the isomorphism again follows by \autoref{prop speirs can iso}. 

For the second isomorphism, note that for $r = 0$, when $pm <n(r+1)$ one observes that $d \leq r$ and  therefore the result follows by \autoref{prop speirs frobenius iso} since $n \nmid m$. 
For $r>0$, one computes that $d<r$ under the hypothesis $pm <n(r+1)$ hence the isomorphism is again given by \autoref{prop speirs frobenius iso}. To see that $d<r$, one needs to show $\frac{m-1}{n}<r$ for which it suffices to show that $\frac{pm-p}{n}<pr$, i.e.\ $pm-p<prn$  holds under the assumption $pm<n(r+1)$ as soon as $r>0$. This assumption gives the first line below
\begin{equation*}
\begin{split}
pm -p&< n(r+1)-p\\
& <n(r+1)\\
& \leq pnr
\end{split}
\end{equation*}
where the last line follows since $r>0$. 
\end{proof}

\subsection{The $p$-adic filtration on $W(\field)/p^n$}
Recall from \autoref{nota  p adic filtration on truncated witt}, that there is the $p$-adic filtration $S_{\field}^\bullet \in \mathrm{CAlg}(\Tow(\Sp))$   on $W(\field)/p^n$ whose associated graded is given by $\field[t]/t^n$ with $t$ in weight $1$ and degree $0$ as before. 
Also, recall that the THH-convergence (\autoref{def:connectivit-assumption}) of $S_{\field}^{\bullet}$ is established in \autoref{exam p adic filtration is thh convergent}.

Recall that $\ntc(W(\field)/p^n,\field)$  denotes the fiber of the map 
\[\ntc(W(\field)/p^n) \to \ntc(\field)\]
and similarly for $\ntc$, $\tp$ and $\tc$. By \autoref{nota relative versions for filtered rings} and due to \autoref{lemm conditional convergence of filtrations},  $\rntc(S_{\field}^{\bullet})^0 \simeq \ntc(W(\field)/p^n,\field)$, we have the same for $\tp$.
\begin{prop}\label{prop speirs odd ntc and ntp}
The homotopy groups of $\rntc(\field[t]/t^n)$ and $\rtp(\field[t]/t^n)$ are concentrated in odd degrees. As a result, 
the homotopy groups of $\ntc(W(\field)/p^n,\field)$ and $\tp(W(\field)/p^n,\field)$ as well as $Q_{\ell}^0\rntc(S_{\field}^\bullet)$ and $Q_{\ell}^0 \rtp(S_{\field}^{\bullet})$ are also concentrated in odd degrees. 

\end{prop}
\begin{proof}
The first sentence is \cite[Proposition 12]{speirs2020truncatedpolynomial}. The second sentence follows by the first and the spectral sequences corresponding to the filtered spectra $\rntc(S_{\field}^\bullet)$ and $\rtp(S_{\field}^\bullet)$ from \autoref{mayss}. For the third statement, note that due to \autoref{lemm associated graded of qn} the corresponding  spectral sequences are weight truncations of the spectral sequences for $\rntc(S_{\field}^\bullet)$ and $\rtp(S_{\field}^\bullet)$ hence they are also concentrated in odd total degrees. 
\end{proof}

As in \autoref{defi filtered cyclotomic thh}, $\rthh(S_{\field}^{\bullet})$ has the structure of a filtered object in cyclotomic spectra 
\[\rthh(S_{\field}^{\bullet}) \in \Tow(\cycsp)\]
whose Frobenius map is given by 
\[\widetilde{\varphi} \co \mathrm{TC}^{-}(S_{\field}^{\bullet})\longrightarrow R_p (\mathrm{THH}(S_{\field}^{\bullet})^{tC_p})^{hS^1} \longrightarrow (\mathrm{THH}(S_{\field}^{\bullet})^{tC_p})^{hS^1}\]
where the second map above is given by the natural transformation $R_p \Rightarrow \mathrm{id}$ from \autoref{natural-trans-R} and the first map comes from the twisted cyclotomic structure on $\rthh(S_{\field}^{\bullet})$~\cite[Appendix~A]{antieau2020beilinson}. Then by \autoref{defi orbitS filtration on tc}, there is a filtration 
\[\filorb \rtc(S_{\field}^0)\]
of $\tc(W(\field)/p^n,\field)$  by \autoref{theo orbit filtration} \eqref{Orbit-filt-1}.

Since $S_{\field}^{\bullet}$ may not be strongly THH-convergent, we need to prove the following analogue of \autoref{prop partial filtration aprroximate tc}.

\begin{prop}\label{prop tc apprxmation with qn for truncated witt}
There is an isomorphism 
\[\pi_{*}Q_{n(r+1)}^0 \filorb  \rtc(S_{\field}^0) \cong \pi_{*} \tc(W(\field)/p^n,\field).\]
for all $*\leq 2r+1$.
\end{prop}
\begin{proof}

There is a fiber sequence (see \eqref{eq cofiber sequence for qn}) 
\[C_{\leq n(r+1)}\filorb\rtc(S_{\field}^0) \to \filorb\rtc(S_{\field}^0) \to Q_{n(r+1)}\filorb \rtc(S_{\field}^0)\]
where each filtration above have vanishing limit by \autoref{theo orbit filtration} \eqref{Orbit-filt-2}. By construction, the associated graded of the left hand side above is  given by the associated graded of $\filorb\rtc(S_{\field}^0)$ in weights $\geq n(r+1)$ and is trivial otherwise. 
In particular, the associated graded is given by $\Sigma \rtc^+(\field[t]/t^n)^m$ for $m \geq n(r+1)$. This vanishes in $\pis$ for $*\leq 2r+1$ due to \autoref{coro phi and can for truncated witt} and the fact that $\rntc(\field[t]/t^n)^m$ and $\rtp(\field[t]/t^n)^m$ are concentrated in odd degrees (\autoref{prop speirs odd ntc and ntp}). In other words, the associated graded of $C_{\leq n(r+1)}\filorb\rtc(S_{\field}^0)$ vanishes in homotopy degree $2r+1$ and below which implies by the completeness of this filtered spectrum that $\pis C_{\leq n(r+1)}^0\filorb\rtc(S_{\field}^0) = 0$ for $*\leq 2r+1$. Now the result follows by the weight $0$ component of the fiber sequence above. 
\end{proof}  

Next, we need to prove an analogue of \autoref{coro frobenius is surjective in homotopy} to establish the surjectivity of the map $\pi_{2r+1}Q_{n(r+1)}^0\widetilde{\varphi}$. The argument in this case is slightly different since the bounds we establish for the bijectivity of $\textup{can}$ and $\varphi$ in \autoref{coro phi and can for truncated witt} is different than those we used in the proof of \autoref{mainthmevenvanish}. Because of this, we consider the following factorization of $\pi_{2r+1}Q_{n(r+1)}^0\widetilde{\varphi}$.

\begin{rema}\label{rema factoring frobenius for truncated witt}
We factor
\begin{equation}\label{map-to-be-factored}
Q_{n(r+1)} \widetilde{\varphi} \co Q_{n(r+1)}\rntc(S_{\field}^{\bullet}) \to Q_{n(r+1)}R_p\rtp(S_{\field}^{\bullet}) \to Q_{n(r+1)}\rtp(S_{\field}^{\bullet})
\end{equation}
as 
\begin{multline}\label{eq long composite frobenius truncated witt}
Q_{n(r+1)}\rntc(S_{\field}^{\bullet}) \to Q_{n(r+1)}R_p\rtp(S_{\field}^{\bullet}) \to Q_{n(r+1)} R_p Q_{n(r+1)}\rtp(S_{\field}^{\bullet})\\ \to Q_{n(r+1)}Q_{n(r+1)}\rtp(S_{\field}^{\bullet}) \simeq Q_{n(r+1)}\rtp(S_{\field}^{\bullet})
\end{multline}
(up to homotopy) where the second map in the composite \eqref{eq long composite frobenius truncated witt} uses the natural transformation $\mathrm{id} \Rightarrow Q_{n(r+1)} $, the third map uses natural transformation $R_p \Rightarrow \mathrm{id}$ and the equivalence on the right follows from the construction of $Q_{n(r+1)}$. The fact that the map \eqref{map-to-be-factored} factors as \eqref{eq long composite frobenius truncated witt} again uses the natural transformation $\mathrm{id} \Rightarrow Q_{n(r+1)}$. 
\end{rema}

\begin{lemm}\label{lemm truncated witt a quotient is surjective}
The weight zero component
\[ Q_{n(r+1)}^0 R_p Q_{n(r+1)}\rtp(S_{\field}^{\bullet}) \xrightarrow{\simeq} Q_{n(r+1)}^0Q_{n(r+1)}\rtp(S_{\field}^{\bullet})\]
of  the third map in \eqref{eq long composite frobenius truncated witt} is an equivalence. 
\end{lemm}

\begin{proof}
Firstly, $R_p \Rightarrow \mathrm{id}$ is the identity map in the weight zero component. Then the result follows by noting that   $R_p Q_{n(r+1)}\rtp(S_{\field}^{\bullet})$ is trivial in weight $n(r+1)$.
\end{proof}

\begin{prop}\label{prop truncated witt surjective frobenius}
The map
\[\pi_{2r+1} Q_{n(r+1)}^0\widetilde{\varphi} \co \pi_{2r+1}Q_{n(r+1)}^0\rntc(S_{\field}^{\bullet}) \to \pi_{2r+1}Q_{n(r+1)}^0\rtp(S_{\field}^{\bullet})\]
induced by the Frobenius is surjective. 
\end{prop}
\begin{proof}
By the factorization  \eqref{eq long composite frobenius truncated witt} and \autoref{lemm truncated witt a quotient is surjective}, it suffices to prove that the composite map 
\begin{equation}\label{eq a shorter composite for truncated witt frobenius}
Q_{n(r+1)}^0\rntc(S_{\field}^{\bullet}) \to Q_{n(r+1)}^0R_p\rtp(S_{\field}^{\bullet}) \to Q_{n(r+1)}^0 R_p Q_{n(r+1)}\rtp(S_{\field}^{\bullet})\end{equation}
induced by the first two maps in \eqref{eq long composite frobenius truncated witt} is surjective in $\pi_{2r+1}$. 

We first note that by \cite[Propositions 12 and 13]{speirs2020truncatedpolynomial}, \begin{equation}\label{eq truncated witt tate vanish prime to p}
\rtp(\field[t]/t^n)^m \simeq 0 
\end{equation}
when  $p \nmid m$. Therefore, both $R_p$ above commute with the associated graded functor; see the discussion in the proof of \autoref{lemm frobenius is surjective on the quotient}.

 Also recall that the functor $Q_{\ell}$ truncates the associated graded at  weight $\ell$ due to \autoref{lemm associated graded of qn}. Therefore, the spectral sequence corresponding to the domain of \eqref{eq a shorter composite for truncated witt frobenius} is trivial in weight $m\geq n(r+1)$ and given by $\pis \rntc(\field[t]/t^n)^m$ otherwise.
 The  spectral sequence corresponding to the target of the composite \eqref{eq a shorter composite for truncated witt frobenius} is trivial in weight $m$ for which $pm\geq n(r+1)$ and given by $\pis \rtp(\field[t]/t^n)^{pm}$ in weight $m$ with $pm<n(r+1)$.  Both  spectral sequences are concentrated in odd total degrees due to \autoref{prop speirs odd ntc and ntp} and they come from finite filtrations. Therefore, it suffices to show that this map of spectral sequences is surjective on the first page in total degree $2r+1$. This follows by the fact that the Frobenius 
\begin{equation}\label{eq map of spectral sequences for frobenius for truncated witt}
\pis \rntc(\field[t]/t^n)^m \to \pis \rtp(\field[t]/t^n)^{pm}\end{equation}
is an isomorphism for $* = 2r+1$ and $pm<n(r+1)$ due to \autoref{coro phi and can for truncated witt}.
\end{proof}

\begin{prop}\label{prop finiteness for ntc and tp of truncated witt}
Let $\field$ be finite, then the following abelian groups
\[ \pi_{k}Q_\ell^0 \rntc(S_{\field}^{\bullet}) \textup{\ \ and \ } \pi_{k}Q_\ell^0 \rtp(S_{\field}^{\bullet}) \]
are finite for each $\ell$.
\end{prop}
\begin{proof}
We prove the $\rntc$ case, the $\rtp$ case follows in the same way. By \autoref{lemm associated graded of qn}, the spectral sequence corresponding to $Q_\ell \rntc(S_{\field}^{\bullet})$ computing $Q_\ell^0 \rntc(S_{\field}^{\bullet})$ is given by a weight truncation of $\pis \rntc(\field[t]/t^n)$ hence it has finitely many entries contributing to total degree $k$. Furthermore, this spectral sequence is concentrated in odd total degrees due to \autoref{prop speirs odd ntc and ntp} and each entry is finite due to \cite[Propositions 12 and 13]{speirs2020truncatedpolynomial}. This proves the desired result. 
\end{proof}

\begin{rema}\label{rema angeltveit quotient}
In \cite{Angeltveit2015kTheoryfinitewitt}, Angeltveit proves the following analogous formula to \autoref{theo even trivial gives odd}. 
\[\frac{\lvert \pi_{2r+1}\tc(W(\field_q)/p^n,\field)\rvert }{\lvert \pi_{2r} \tc (W(\field_q)/p^n,\field)\rvert } = q^{(n-1)(r+1)} \]
Here, $r\geq 0$ and  $q$ is a power of $p$ and as noted in \cite{Angeltveit2015kTheoryfinitewitt},  this is given by $\lvert \pi_{2r+1}\tc(\field_q[t]/t^n,\field_q)\rvert $ due to \cite{hesselholt1997kthryoftruncated}. Indeed, we expect that it is also possible to prove this result using our methods from \autoref{theo even trivial gives odd}, but we do not pursue this here since an independent proof is given in \cite[Proposition 1.5]{AKN24}.
\end{rema}

\mainthmcardinalitywitt*
\begin{proof}
This follows as in \autoref{mainthmevenvanish}. Note that by~\cite[Theorem~7.2.2.1]{DGM13}, the trace map $\mathrm{K}(W(\field)/p^n,\field) \simeq \tc(W(\field)/p^n,\field)$ is an equivalence.

We  begin with the first statement. 
By \autoref{prop speirs odd ntc and ntp} and \autoref{prop tc apprxmation with qn for truncated witt} it suffices to prove that the map $F \otimes Q_{n(r+1)}^0 \widetilde{\varphi} - \textup{id} \otimes \textup{can}$ in the $e = n(r+1)$ case of \autoref{prop tc fiber sequence for the witt case} is surjective in $\pi_{2r+1}$. This now follows by \autoref{prop truncated witt surjective frobenius} and \autoref{lemm H  is surjective before modding by p} whose finiteness condition is satisfied by \autoref{prop finiteness for ntc and tp of truncated witt}.

The proof of the second statement follows as in the proof of the second part of \autoref{mainthmevenvanish} by replacing $\sphwf \otimes A$ with $S_{\field}^0 \simeq \sphwf \otimes S_{\fp}^{0}$, $\wtw A$ with $S_{\fp}^\bullet$, $Q_{2k+1}$ with $Q_{n(r+1)}$ and using \autoref{rema angeltveit quotient} instead of \autoref{theo even trivial gives odd}.
\end{proof}

\maincorcardinalitywitt*

\begin{proof}
This follows as in the proofs of \autoref{coro countable relative k theory groups} and \autoref{mainthmevenvanish} \eqref{item 2 thmevenvanishing}, as we now explain. For a fixed integer $k$, the $\fp$-module $\pi_{k} \rthh(\fpc[t]/t^n)^m$ is countable and $\tau_{\leq k}\rthh(\fpc[t]/t^n)^m \simeq *$ for sufficiently large $m$; this follows by \eqref{eq splitting for thh of truncated polynomial} and \cite[Proposition 3]{speirs2020truncatedpolynomial}. By the  $S^1$-homotopy orbit spectral sequence 
\[\pis \big(\rthh(\fpc[t]/t^n)^m\big)[z] \implies \pis \rtc^+(\fpc[t]/t^n)^m\]
with $\lvert z\rvert = 2$, we deduce that each 
$\pi_k \rtc^+(\fpc[t]/t^n)^m$ is countable, has bounded $p$-torsion and trivial for sufficiently large $m$. Since  $\pi_k \rtc^+(\fpc[t]/t^n)^m$ is the $\mathrm{E}_1$-page of the orbit spectral sequence computing $\tc(W(\fpc)/p^n,\field)$, this gives the desired result. 
\end{proof}

\section{On the algebraic K-theory of motivic pullback squares}\label{sec applications to discrete rings}

The motivic pullback squares of Land and Tamme \cite{land2023kthrypushouts} show that certain relative algebraic K-theory groups for discrete rings are given by the algebraic K-theory of a corresponding $\z$-algebra $\odot$ with homotopy $\field[t_2]$. When the $\bE_1$-ring $\odot$ is the free $\field$-algebra on $\Sigma^2 \field$ (i.e.\ $\odot$ is topologically formal \cite{bayindir2025towards}), then \cite{bayindir2020kthryofthh} provides the algebraic K-theory of $\odot$ and hence the relative algebraic K-theory groups of the discrete rings in question; for this reason, the formality question for $\odot$ is studied extensively in \cite{bayindir2025towards}. However, there are also examples where $\odot$ is not topologically formal \cite[Example~4.31]{land2023kthrypushouts}.  To our knowledge, \autoref{mainthmevenvanish} provides the first  K-theory computation in infinitely many degrees for a $\z$-algebra that is not topologically formal.

For the following, $u$ denotes a class of degree $0$ so that $\zp[u]$ denotes the usual  discrete polynomial ring.  Given a map $A\to B$ of $\mathbb{E}_1$-rings, we write $\mathrm{K}(A,B)$ for the fiber of the map $\mathrm{K}(A\to B)$ and refer to this as relative algebraic K-theory. The following is a special case of the motivic pullback squares constructed by Land and Tamme in \cite{land2023kthrypushouts}.

\begin{coro}\label{coro general motivic pullback square}
Let $g \in \zp[u]$ be a polynomial with $g(0) = p$.
There is a pullback square of spectra
\begin{equation*}
\begin{tikzcd}
\mathrm{K}(W(\field)[u]/(ug)) \ar[d] \ar[r] & \mathrm{K}(W(\field)[u]/(g)) \ar[d]\\
\mathrm{K}(W(\field)) \ar[r] & \mathrm{K}(\odot)
\end{tikzcd}
\end{equation*}
where $\odot$ is an $\bE_1$-ring with $\pis \odot \cong \field[x_2]$ and $\odot \simeq \sph_{W(\field)}\otimes \odot'$ for another $\bE_1$-ring $\odot'$ with $\pis \odot' \cong \fp[x_2]$. In particular,
\[\mathrm{K}(W(\field)[u]/(ug),W(\field) \times W(\field)[u]/(g)) \simeq \Omega \mathrm{K}(\odot) \,.\]
\end{coro}
\begin{proof}
Recall from \cite{land2023kthrypushouts} that a motivic pullback square of $\bE_1$-rings results in a pullback square of spectra after applying algebraic K-theory. 
One begins with the motivic pullback square given in \cite[Lemma~4.30]{land2023kthrypushouts} for $R = \zp[u]$, $y = u$ and $x = g$ and this gives a motivic pullback square that results in the one above for $\field = \fp$. Applying $\sphwf \otimes -$ to this square gives another motivic pullback square by \cite[Proposition 2.16]{land2023kthrypushouts} which is precisely the one given in the corollary. The last statement follows by the definition of the relative algebraic K-theory and by translating the pullback square in the corollary into the fiber sequence.
\end{proof}

In particular, \autoref{coro absolute k theory is odd} applies to $ \odot$ above to deduce the following when $\field = \fpc$. 

\begin{coro}\label{coro cardinalities of kthry for discrete rings}
In the situation of \autoref{coro general motivic pullback square}, 
\[\mathrm{K}_*\big(W(\fpc)[u]/(ug),W(\fpc) \times W(\fpc)[u]/(g)) \big)\]
is trivial for odd $*\geq 1$,  infinite for even $* \geq 1$ and given by $\mathrm{K}_{*+1}(\fpc)$ for $*<1$.
\end{coro}

\begin{exam}
We consider the case $g = u+p$. Then $W(\fpc)[u]/u(u+p)$, which is isomorphic to the Burnside ring $W(\fpc)[u]/u(u-p)$ of the cyclic group of order $p$ over $W(\fpc)$, is given by the homotopy pullback $W(\fpc) \times_{\fpc} W(\fpc)$, see \cite[Example~4.31]{land2023kthrypushouts}. 
\end{exam}

We deduce the following from \autoref{coro cardinalities of kthry for discrete rings}.

\begin{coro}\label{coro relative k theory of arithmetic coordinate axis}
The relative algebraic K-theory groups
\[\pis \mathrm{K}\big(W(\fpc)\times_{\fpc} W(\fpc), W(\fpc) \times W(\fpc) \big) \]
are trivial for odd $*\geq 1$,  infinite for even $* \geq 1$ and given by $\mathrm{K}_{*+1}(\fpc)$ for $*<1$. 
\end{coro}

\begin{rema}\label{rema nonformality}
As mentioned in \autoref{subsec intro on motivic pullbacks}, the $\mathbb{E}_1$-ring $\odot$  from \autoref{coro general motivic pullback square}
corresponding to the  case of \autoref{coro relative k theory of arithmetic coordinate axis} is  $W(\fpc)\sslash p$. This is defined by the pushout of $\bE_1$-algebras:
\begin{equation}\label{eq defining mod mod p guy}
W(\field) \sslash p := W(\field) \coprod_{W(\field)[z]} W(\field)
\end{equation}
where we use two maps $W(\field)[z] \to W(\field)$ sending $z$ to $0$ and $p$ respectively.

We claim that  for every perfect field $\field$ of characteristic $p>2$, the $\bE_1$-ring $W(\field) \sslash p$ is not equivalent to the free $\bE_1$ $\field$-algebra on $\Sigma^2\field$ and hence the results of~\cite{bayindir2020kthryofthh} do not apply to compute its algebraic K-theory. This follows as in~\cite[Remark 4.33]{land2023kthrypushouts}. We have $W(\field) \sslash p:= W(\field) \otimes_{\z} \z\sslash p$, from this and \cite[Corollary 2.6]{davis2025cyclichomologyofzmodmodp}, we deduce that $\hh^{\z}_2(W(\field) \sslash p) = W(\field)/p^2$. However, by~\cite[Proposition IV.4.2]{nikolausscholze2018topologicalcyclic},  
 we also  have $\thh_2(W(\field) \sslash p) = W(\field)/p^2$, which means that the $\bE_1$-ring $W(\field)\sslash p$ cannot be equivalent to an $\field$-algebra since $p\neq 0 \in \thh_2(W(\field)\sslash p)$. 
\end{rema}

\section{Bockstein spectral sequences and syntomic cohomology}\label{syntomic}

\subsection{Background}
Here we introduce motivic filtrations on topological Hochschild homology and related invariants as in \cite{hahn2022motivic} and \cite{angelini2024syntomic}. For simplicity, we will only define these filtrations for certain $\mathrm{MU}$-algebras since that will suffice for our purposes. Recall that $\MU_{W(\field)}:=\mathbb{S}_{W(\field)}\otimes \MU$ and $k(n)_{\field}:=\mathbb{S}_{W(\field)}\otimes k(n)$. 

\begin{defi}\label{defin-motivic}
For $R\in \{\mathrm{MU}_{W(\field)},k(n)_{\field},W(\field),\field\}$. We define 
\[ \mathrm{fil}_{\textup{mot}}^{*}F(R)_p=\lim_{\Delta}\tau_{\ge 2*}F(R/\mathrm{MU}^{\otimes \bullet+1})_p
\] 
and we write 
\[ 
\mathrm{gr}_{\textup{mot}}^{*}F(R)_p= \mathrm{fil}_{\textup{mot}}^{*}F(R)_p/ \mathrm{fil}_{\textup{mot}}^{*+1}F(R)_p
\]
for $F\in \{\mathrm{THH},\mathrm{TC}^{-}, \mathrm{TP}\}$. When $R\in \{k(n)_{\field},\field\}$, then in light of \autoref{lemm thh is p complete when thh gr is} and \autoref{rema tcp and then hs1 agree with ts1 when p complete} we omit the $p$-completion from the notation and simply write $\mathrm{fil}_{\textup{mot}}^{*}F(R)$ and $\mathrm{gr}_{\textup{mot}}^{*}F(R)$. 
\end{defi}

\begin{rema}
To translate notations, note that $\THH(\MU)\to \MU$ is an evenly free,  map of $S^1$-spectra by~\cite[Example~4.2.3]{hahn2022motivic}; see~\cite[Definition 2.2.15]{hahn2022motivic} for the definition of an evenly free map. Let $F\in \{\THH, \TC^{-},\TP\}$ and $q\ge 0$. Since $\THH(\MU)\to \MU$ is evenly free, we know that $F(\MU/\mathrm{MU}^{\otimes q+1})_p$ is even and by direct computation it has bounded $p$-power torsion. By the same arguments as~\cite[Proposition~2.2.18,~Theorem~4.5.2]{angelini2024syntomic}, we determine that 
$F(k(n)_{\field} /\mathrm{MU}^{\otimes q+1})_p$ is even and has bounded $p$-power torsion. Adapting the argument from~\cite[Theorem~E]{hahn2020redshift}, we deduce that the spectra $F(\field/\mathrm{MU}^{\otimes q+1})_p$ and $F(W(\field)/\mathrm{MU}^{\otimes q+1})_p$ are even and have bounded $p$-power torsion. 
Consequently, by~\cite[Corollary~A.2.6]{hahn2022motivic},
we have equivalences (cf.~\cite[Definition~4.2.1]{hahn2022motivic} and \cite[\S~2.1.1]{angelini2024syntomic})
\begin{align*}
\mathrm{fil}_{\textup{mot}}^*\mathrm{THH}(R)_p  &\simeq \mathrm{fil}_{\textup{ev}/\mathrm{THH}(\mathrm{MU})_p,p}^*\mathrm{THH}(R)_p  \,,  \\
\mathrm{fil}_{\textup{mot}}^*\mathrm{TC}^{-}(R)_p  &\simeq \mathrm{fil}_{\textup{ev}/\mathrm{THH}(\mathrm{MU})_p,p,hS^1}^*\mathrm{THH}(R)_p  \,, \\
\mathrm{fil}_{\textup{mot}}^*\mathrm{TP}(R)_p   &\simeq \mathrm{fil}_{\textup{ev}/\mathrm{THH}(\mathrm{MU})_p,p,tS^1}^*\mathrm{THH}(R)_p \,. 
\end{align*}
\end{rema}

\begin{rema}
Let $R\in \{\mathrm{MU}_{W(\field)},k(n)_{\field},W(\field),\field\}$. Since $\mathrm{MU}$ is chromatically quasisyntomic~\cite[Example~1.3.3,~Definition~4.1.12]{hahn2022motivic}, we can determine that there are filtered maps 
\[ 
\widetilde{\varphi},\widetilde{\mathrm{can}} : \mathrm{fil}_{\textup{mot}}^*\mathrm{TC}^{-}(R)_p\to \mathrm{fil}_{\textup{mot}}^*\mathrm{TP}(R)_p
\]
that converge to the usual $\varphi$ and $\mathrm{can}$ from~\cite{nikolausscholze2018topologicalcyclic} by~\cite[Corollary~A.2.10,~Remark~4.2.10]{hahn2022motivic}.  We will abuse notation and write $\mathrm{can}$ and $\varphi$ for these maps and the induced maps 
\[ 
\varphi,\mathrm{can} : \mathrm{gr}_{\textup{mot}}^*\mathrm{TC}^{-}(R)_p\to \mathrm{gr}_{\textup{mot}}^*\mathrm{TP}(R)_p \,.
\]

\end{rema}
\begin{defi}\label{syntomic-coh-def}
Let $R\in \{\mathrm{MU}_{W(\field)},k(n)_{\field},W(\field),\field\}$. We write 
\[ \mathrm{fil}_{\mot}^*\mathrm{TC}(R)_p \]
for the equalizer 
\[ 
\mathrm{eq}(
\begin{tikzcd}
\mathrm{fil}_{\textup{mot}}^*\mathrm{TC}^{-}(R)_p
\arrow[r,shift right,"\textup{can}",swap] 
\arrow[r,shift left,"\varphi"]  & \mathrm{fil}_{\textup{mot}}^*\mathrm{TP}(R)_p
\end{tikzcd}
) 
\]
and we write 
\[ \mathrm{gr}_{\mot}^*\mathrm{TC}(R)_p= \mathrm{fil}_{\mot}^*\mathrm{TC}(R)_p/ \mathrm{fil}_{\mot}^{*+1}\mathrm{TC}(R)_p\,.
\]
As before, we omit the $p$-completion from the notation in the case of $R\in \{k(n)_{\field},\field\}$ in light of \autoref{lemm thh is p complete when thh gr is} and \autoref{rema tcp and then hs1 agree with ts1 when p complete}. 
\end{defi}

In this section, we use the following terminology throughout. 
\begin{term}
Let $M$ be a graded spectrum and consider $x\in \pi_sM^w$, then we write 
\[\| x\|=(s,2w-s)\] 
and refer to $s$ as the \emph{degree}, $w$ as the \emph{weight} and $2w-s$ as the \emph{Adams weight}. 
\end{term}

\begin{rema}\label{rema:vn}
Recall that by~\cite[Corollary~2.2.21]{hahn2022motivic}
\[\mathrm{fil}_{\textup{ev}}^*\mathbb{S}:=\lim_{\Delta} \tau_{\ge 2*}\mathrm{MU}^{\otimes \bullet+1} 
\]
acts on all of the filtered spectra from \autoref{defin-motivic} and consequently the associated graded 
\[\mathrm{gr}_{\textup{ev}}^*\mathbb{S}:=\mathrm{fil}_{\textup{ev}}^*\mathbb{S}/\mathrm{fil}_{\textup{ev}}^{*+1}\mathbb{S}
\]
acts on all of the graded spectra from \autoref{defin-motivic}. Since there is a symmetric monoidal equivalence between $(\mathrm{gr}_{\textup{ev}}^*\mathbb{S})_p$-modules in graded spectra and the stable derived category of 
$\mathrm{BP}_*\mathrm{BP}$-comodules, by \cite[Corollary~1.1.6]{hahn2022motivic} and \cite[Corollary~1.2]{GWZ21},
we can associate an $\mathbb{E}_{\infty}$ algebra in $\mathrm{gr}_{\textup{ev}}^*\mathbb{S}_p$  modules $\mathrm{gr}_{\textup{ev}}^*\mathbb{S}/(p,\cdots ,v_m)$ to the commutative algebra in $\mathrm{BP}_*\mathrm{BP}$-comodules $\mathrm{BP}_*/(p,\cdots ,v_m)$. We can then compute that 
\[ \pi_*\mathrm{gr}_{\textup{ev}}^*\mathbb{S}/(p,\cdots ,v_m)= \mathrm{Ext}_{\mathrm{BP}_*\mathrm{BP}}^*(\pi_*\mathrm{BP},\pi_*\mathrm{BP}/(p,\cdots ,v_m))\]
which contains $\mathbb{F}_p[v_{m+1}]$ as a subring in cohomological degree zero, see~\cite[ Theorem~4.3.2]{Rav86}, which corresponds to Adams weight. Let $R\in\{\mathrm{MU}_{W(\field)},k(n)_{\field},W(\field),\field\}$. From this discussion, we know that $v_{m+1}$ is in Adams weight zero and it acts on 
\[
\mathrm{gr}_{\textup{mot}}^*F(R)/(p,\cdots ,v_m) := \mathrm{gr}_{\textup{mot}}^*F(R)_p\otimes_{\mathrm{gr}_{\textup{ev}}^*\mathbb{S}}\mathrm{gr}_{\textup{ev}}^*\mathbb{S}/(p,\cdots ,v_m)
\]
for $F\in \{\mathrm{THH},\mathrm{TC}^{-},\mathrm{TP},\mathrm{TC}\}$ and $m\ge 0$. 
\end{rema} 

\begin{defi}
For $R\in \{\mathrm{MU}_{W(\field)},k(n)_{\field},W(\field),\field\}$ and $m\ge 0$, we refer to 
\[ \pi_*\grmot \mathrm{TC}(R)/(p,\cdots ,v_m)\]
as the \emph{mod $(p,\cdots ,v_m)$ syntomic cohomology} of $R$.\footnote{This is justified by~\cite[Theorem~5.0.3]{hahn2022motivic}.} 
\end{defi}

\subsection{Syntomic cohomology}
We first recall the computation of the mod $(p,\cdots,v_{n+1})$ syntomic cohomology of $k(n)_{\field}$ and $\field$. First, we need a preliminary definition. 

\begin{defi}\label{def:epsilon}
Let $i\ge 1$ be an integer. 
Suppose $X^{*}$ is a $\mathrm{gr}_{\textup{ev}}^*\mathbb{S}$-module such that $v_i$ acts trivially on $X^{*}/(p,\cdots ,v_{i-1})$ and 
\[\pi_*(X^*/(p,\cdots ,v_{i-1}))\] 
vanishes in degree $2p^{i}-1$ and Adams weight $-1$. 
Then we define $\overline{\varepsilon}_i$ to be the unique class that maps to $1$ under the boundary map 
\[ \pi_*X^{*}/(p,\cdots ,v_{i})\to \pi_{*-2p^i+1}X^{*}/(p,\cdots ,v_{i-1})\]
of the long exact sequence induced by the cofiber sequence 
\[\Sigma^{2p^i-2}\mathrm{gr}_{\textup{ev}}^*\mathbb{S}/(p,\cdots ,v_{i-1})\overset{v_i}{\longrightarrow}  \mathrm{gr}_{\textup{ev}}^*\mathbb{S}/(p,\cdots ,v_{i-1})\to  \mathrm{gr}_{\textup{ev}}^*\mathbb{S}/(p,\cdots ,v_{i}) \,.
\]
\end{defi}

\begin{nota}
For a field $k$, when we write $k\langle x_1,\cdots,x_m\rangle$ we mean the exterior algebra over $k$ with generators $x_1,\cdots,x_m$. Given a $k$-vector space $M$, we write $M\langle x_1,\cdots ,x_m\rangle:=M\otimes_{k}k\langle x_1,\cdots ,x_m\rangle$. 
\end{nota}
\begin{nota}
Let $\field_{\mathrm{Frob}}$ denote the coinvariants of the action of Frobenius on $\field$. 
\end{nota}
\begin{rema}\label{Z-algebra}
Note that the $\mathrm{gr}_{\textup{ev}}^*\mathbb{S}$-module structure on $\mathrm{gr}_{\textup{mot}}^*\mathrm{TC}(\field)$ factors through  $\mathrm{gr}_{\textup{ev}}^*\mathbb{Z}_p$ since $\mathrm{K}(\mathbb{F}_p)_p\simeq \mathbb{Z}_p$ and consequently $v_i$ acts trivially on  $\mathrm{gr}_{\textup{mot}}^*\mathrm{TC}(\field)/(p,v_1,\cdots ,v_{i-1})$ for all $i\ge 1$. 
\end{rema}

\begin{prop}[{\cite{HM97},\cite{bhatt2019thhandintegralpadichodge}}]\label{syntomic field}
The mod $p$ syntomic cohomology of $\field$ is $\mathbb{F}\oplus \field_{\mathrm{Frob}}\{\partial\}$. Moreover, the mod 
$(p,\cdots ,v_{n+1})$ syntomic cohomology of $\field$ is 
\[ 
\mathbb{F}_p\langle \overline{\varepsilon}_1,\overline{\varepsilon}_2,\cdots ,\overline{\varepsilon}_{n+1} \rangle \oplus \mathbb{F}_{\mathrm{Frob}}\langle \overline{\varepsilon}_1,\overline{\varepsilon}_2,\cdots ,\overline{\varepsilon}_{n+1} \rangle \{\partial\}
\] 
where $\|\partial\|=(-1,1)$ and $\|\overline{\varepsilon}_i\|=(2p^i-1,-1)$. 
In the $v_i$-Bockstein spectral sequence there are differentials $d_1(\overline{\varepsilon}_i)=v_i$ and differentials $d_1(b_{\alpha}\partial\overline{\varepsilon}_i)=b_{\alpha}\partial v_i$ for each basis element $b_{\alpha}$ of $\mathbb{F}_{\mathrm{Frob}}$ as a $\mathbb{F}_p$-vector space
for each $1\leq i\leq n+1$ and no further differentials except those generated by the Leibniz rule.
\end{prop}
\begin{proof}
By~\cite[Theorem~B]{HM97}, we know that 
\[ \mathrm{TC}(\field)\simeq \mathbb{Z}_p\oplus \Sigma^{-1}\mathrm{W}(\field)_{\mathrm{Frob}}\]
where $\mathrm{W}(\field)_{\mathrm{Frob}}$ is the coinvariants of the action of the Frobenius map on $\mathrm{W}(\field)$ 
and it is a $\mathbb{Z}_p$-module so that 
$\pi_*\mathrm{TC}(\field)/p\cong \mathbb{F}_p\oplus \field_{\mathrm{Frob}}\{\partial\}$ where $|\partial|=-1$. By \cite[\S~6]{bhatt2019thhandintegralpadichodge}, we know that 
\[ 
\grmot\mathrm{TC}(\field)/p\cong \mathbb{F}_p\oplus \mathbb{F}_{\mathrm{Frob}}\{\partial\}
\]
where $\|\partial \|=(-1,1)$. 
By \autoref{Z-algebra} and \autoref{def:epsilon}, we can determine that there is a unique class, which we call 
$\overline{\varepsilon}_i$, that maps to $1$ via the boundary map
\[ 
\pi_*\grmot\mathrm{TC}(\field_q)/(p,\cdots ,v_{i})\to \pi_{*-2p^i+1}\grmot\mathrm{TC}(\field_q)/(p,\cdots ,v_{i-1})
\]
of the long exact sequence induced by the cofiber sequence 
\[ \Sigma^{2p^i-2}\mathrm{gr}_{\textup{ev}}^*\mathbb{S}/(p,\cdots ,v_{i-1})\to \mathrm{gr}_{\textup{ev}}^*\mathbb{S}/(p,\cdots ,v_{i-1})\to \mathrm{gr}_{\textup{ev}}^*\mathbb{S}/(p,\cdots ,v_{i})
\]
of $\mathrm{gr}_{\textup{ev}}^*\mathbb{S}$-modules for bidegree reasons. Consequently, the bidegree of $\overline{\varepsilon}_i$ is $(2p^i-1,-1)$. The argument for differentials in the Bockstein spectral is implied since we know the computation mod $p$. 
\end{proof}

\begin{thm}[{\cite[Theorem~3.5.7,~Theorem~4.5.3]{angelini2024syntomic}}]\label{AKHW}
The mod $(p,\cdots ,v_{n+1})$ syntomic cohomology of $k(n)_{\field}$  is isomorphic to
\begin{align}\label{eq: syntomic}
\mathbb{F}_p\langle \bar{\varepsilon}_{1},\cdots,  \bar{\varepsilon}_{n} , \lambda_{n+1} \rangle \oplus \field_{\textup{Frob}}\langle \bar{\varepsilon}_{1},\cdots,  \bar{\varepsilon}_{n} , \lambda_{n+1} \rangle\{\partial\}\oplus  \bigoplus_{S\subset \{1,\cdots,n\}}M_S
\end{align}
as a $\mathbb{F}_p\langle \lambda_{n+1} \rangle$-module where $\field_{\textup{Frob}}$ is the coinvariants of the action of Frobenius on $\field$. Given $S\subset \{1,\cdots ,n\}$, we write 
\[ 
M_{S}:=
	\field\{t^d\bar{\varepsilon}_S\lambda_{n+1} : -f(S) < d\le p^{n+1}-p^n-f(S)\}
\]
where $f(S)=\sum_{s\in S} p^{s-1}-p^s$. 
The bidegrees are 
\begin{align*}
\|\lambda_{n+1}\|& =(2p^{n+1}-1,1)\\
\| \partial \| & = (-1,1)  \\
\| \bar{\varepsilon}_i\| & = (2p^i-1,-1)  \,, \text{ for } 1\le i \le n \\
\|t^d\bar{\varepsilon}_S\lambda_{n+1}\| & =(\sum_{s\in S}2p^s-|S|-2d+2p^{n+1}-1,1-|S|) 
\end{align*}
for $S\subset \{1,\cdots ,n\}$. Moreover, the map from the mod $(p,\cdots ,v_{n+1})$ syntomic cohomology of $k(n)_{\field}$ to the  mod $(p,\cdots ,v_{n+1})$ syntomic cohomology of $\field$ sends all classes in $\bigoplus_{S\subset \{1,\cdots,n\}}M_S$ to zero, all $\lambda_{n+1}$-multiples to zero and it sends products of classes of the form $b_\alpha \cdot \partial^{a_0}\bar{\varepsilon}_{1}^{a_1}\cdot \dots \bar{\varepsilon}_{n+1}^{a_{n+1}}$ for $a_i\in \{0,1\}$ for $0\le i\le n+1$ to the classes with the same name where $b_{\alpha}$ runs over basis elements for $\field_{\mathrm{Frob}}$ as an $\mathbb{F}_p$-vector space.
\end{thm}

\begin{rema}
It will be crucial that the syntomic cohomology above is concentrated in Adams weights $\leq 2$ and the Adams weight $2$ component is given by $\field_{\textup{Frob}}\{\partial\lambda_{n+1}\}$. Moreover, the lowest degree element in the syntomic cohomology is $\partial$ that is of bidegree $(-1,1)$.
\end{rema}

\begin{nota}
For $F\in \{\TC^{-},\TP,\TC\}$, we define 
\[ \fil_{\mot}^* F_{\textup{rel}}(k(n)_{\field})=\fib  \big ( \fil_{\mot}^*F(k(n)_{\field})\to  \fil_{\mot}^*F({\field}) \big )
\]
with associated graded 
\[ \grmot F_{\textup{rel}}(k(n)_{\field})=\fil_{\mot}^* F_{\textup{rel}}(k(n)_{\field})/\fil_{\mot}^{*+1} F_{\textup{rel}}(k(n)_{\field})\,.
\]
\end{nota}

\begin{defi}\label{defi motivic spectral sequence}
For $R\in \{\mathrm{MU}_{W(\field)},k(n)_{\field},W(\field),\field\}$ and $F\in \{\THH,\TC^{-},\TP,\TC\}$, we refer to the spectral sequence 
\[ \pi_*\grmot F(R)\implies \pi_*F(R)
\]
associated to the filtered spectrum  $\fil_{\mot}^* F(R)$ as the motivic spectral sequence. We define similar spectral sequences for $F_{\textup{rel}}$. The differential $d_r$ raises Adams weight by $r$ and decreases degree by one.  Moreover, by construction it only has differentials for odd $r \geq 3$. 
\end{defi}

\begin{cor}\label{Adams weights}
The bigraded homotopy groups 
\[ 
\pi_*\grmot\TC^{-}_{\rel}(k(n)_{\field}) \qquad \text{ and } \qquad \pi_*\grmot\TP_{\rel}(k(n)_{\field})
\]
are concentrated in Adams weight $1$. Consequently, we can conclude that $\pi_*\grmot\TC_{\rel}(k(n)_{\field})$ is concentrated in Adams weights $[1,2]$. Moreover, the map 
\[\pi_*\grmot\TC_{\rel}(k(n)_{\field})\to \pi_*\grmot\TC(k(n)_{\field}) \]
is an isomorphism in Adams weight $2$. 
\end{cor}

\begin{proof}
By \cite[Corollary~2.2.11]{angelini2024syntomic} and the fact that the generator $t$ of the group cohomology $H^2(S^1,\mathbb{F}_p)$  is in Adams weight zero, we also know that 
\[ 
\pi_*\grmot\TC^{-}(k(n)_{\field}) \qquad \text{ and } \qquad \pi_*\grmot\TP(k(n)_{\field})
\]
are concentrated in Adams weights $[0,1]$ using the algebraic $t$-Bockstein spectral sequence~\cite[Construction~2.1.14]{hahn2022motivic} (cf.~\cite[\S 3]{AKAR23}). 
Since $\thh(\field)$ is an even $\bE_\infty$-ring, its motivic filtration is given by its Whitehead tower. Hence,  $\grmot\thh(\field)$ is concentrated in Adams weight $0$, see also \cite[Proposition 2.2.8]{angelini2024syntomic}, and the same holds for $\grmot \ntc(\field)$ and $\grmot \tp(\field)$, so we deduce that 
\[
\pi_*\grmot\TC^{-}_{\textup{rel}}(k(n)_{\field}) \qquad \text{ and } \qquad \pi_*\grmot\TP_{\textup{rel}}(k(n)_{\field})
\]
are also concentrated in Adams weights $[0,1]$. Since the $0$-line is concentrated in even degrees and the $1$-line is concentrated in odd degrees  by the definition of Adams weight and the motivic spectral sequence collapses at the $\mathrm{E}_2$-page, see~\autoref{defi motivic spectral sequence}, we therefore know that the $\mathrm{E}_2$-page of the motivic spectral sequence for $\TC^{-}_{\textup{rel}}(k(n)_{\field})_p$ and $\TP_{\textup{rel}}(k(n)_{\field})$ must vanish in Adams weight zero by \autoref{oddness}. This proves the first statement. 

Since there is a fiber sequence 
\[ \grmot\TC_{\textup{rel}}(k(n)_{\field})\to \grmot\TC^{-}_{\textup{rel}}(k(n)_{\field})\overset{\varphi-\mathrm{can}}{\longrightarrow} \grmot\TP_{\textup{rel}}(k(n)_{\field})\]
the second to last statement follows.

By the case of the fiber sequence above for $\grmot \tc(\field)$, we deduce that  $\pis \grmot \tc(\field)$ is concentrated in Adams weights $[0,1]$. Hence, this motivic spectral sequence collapses and since $\pis \tc(\field) \cong W(\field) \oplus W(\field)_{\mathrm{Frob}}\{\partial\}$, we must also have $\pis \grmot \tc(\field) \cong W(\field) \oplus W(\field)_\mathrm{Frob}\{\partial\}$ where $\|\partial \| = (-1,1)$, in particular, $\partial$ is of weight $0$. 

 Since $\pi_\ell \tc(k(n)_\field) \to \pi_\ell \tc(\field)$ is an isomorphism for $\ell \leq 0$, we deduce that the map of  spectral sequences computing this map $\pi_\ell \grmot \tc(k(n)_\field) \to \pi_\ell \grmot \tc(\field)$ is surjective as the target is concentrated in $\ell\leq 0$ and  weight $0$. Then the last statement follows since $\grmot \rtc(k(n)_{\field})$ is the fiber of the map 
\[ \grmot\TC(k(n)_{\field})\to \grmot\TC(\field)
\]
which is surjective in $\pis$ as mentioned and whose target is trivial in Adams weights $>1$.
\end{proof}

\subsection{Bockstein spectral sequences}
We begin with some preliminary results regarding syntomic cohomology. The goal is to show that certain classes act on the Bockstein spectral sequences for $k(n)_{\field}$.

\begin{lemm}\label{THH-MU-lem}
There is an isomorphism
\[ \pi_*\mathrm{THH}(\MU_{W(\field)})/p\cong\pi_*\mathrm{THH}(\MU)/p\otimes \field
\]
and consequently class $\lambda_k$ for $k\ge 1$, which generate $\pi_{2p^k-1}\grmot\mathrm{THH}(\MU_{W(\field)})/p$  as a $\field$-module and are in the image of the class $\lambda_k\in \pi_{2p^k-1}\grmot\mathrm{THH}(\mathrm{MU})/(p)$ from~\cite[Definition~2.3.2]{angelini2024syntomic}. 
\end{lemm}
\begin{proof}
By \autoref{rema spherical witts as cyclotomic base}, the map $\mathrm{THH}(\mathbb{S}_{W(\field)})\to \sphwf$ exhibits  $\sphwf$ as the $p$-completion of $\mathrm{THH}(\mathbb{S}_{W(\field)})$. Consequently, there is an equivalence
\[
\mathrm{THH}(\MU_{W(\field)})/p\simeq \mathrm{THH}(\mathrm{MU})/p\otimes \mathbb{S}_{W(\field)}\,.\]
Since $\mathbb{F}_p\otimes \mathbb{S}_{W(\field)}\simeq \field$ and 
$\mathrm{THH}_*(\mathrm{MU})\cong \mathrm{MU}_*\langle \lambda_i' : i\ge 1\rangle$
the homotopy groups of $\mathrm{THH}(\mathrm{MU})/(p)\otimes \mathbb{S}_{W(\field)}$ are 
$\pi_*\THH(\mathrm{MU})/p\otimes \field$. 
For the last statement, we note that $\lambda_i=\lambda_{p^i-1}'$ and the identification 
\[ \pi_*\THH(\MU_{W(\field)})/p\simeq \pi_*\THH(\MU)/p\otimes \field
\]
can be easily upgraded to an isomorphism
\[ \pi_*\grmot\THH(\MU_{W(\field)})/p\simeq \pi_*\grmot\THH(\MU)/p\otimes \field,
\]
adapting the proofs of \cite[Proposition~2.2.6,~Proposition~2.3.1]{angelini2024syntomic}. 
\end{proof}

\begin{prop}\label{lambda-class}
There is a class 
\[
\lambda_k\in \pi_{2p^k-1}\grmot\TC(\MU_{W(\field)})/p
\]
in bidegree $(2p^k-1,1)$ that is a lift of the class 
$\lambda_k\in \pi_{2p^{k}-1}\grmot\THH(\mathrm{MU}_{W(\field)})/p$ defined in \autoref{THH-MU-lem} and is in the image of the class $\lambda_k\in \pi_{2p^k-1}\grmot\mathrm{TC}(\mathrm{MU})/p$ from \cite[Definition 3.2.4]{angelini2024syntomic}.
\end{prop}

\begin{proof}
By~\cite[Lemma~3.2.5]{angelini2024syntomic}, there is a lift $\lambda_k\in \pi_{2p^k-1}\grmot\mathrm{TC}(\MU)/p$ of the generator $\lambda_k\in \pi_{2p^k-1}\grmot\mathrm{THH}(\mathrm{MU})/p$. Therefore, by the commutative diagram 
\[
\begin{tikzcd}
\pi_{2p^k-1}\grmot\mathrm{TC}(\mathrm{MU})/p\arrow{d} \arrow{r} &  \pi_{2p^k-1}\grmot\mathrm{TC}(\MU_{W(\field)})/p \arrow{d} \\ 
\pi_{2p^k-1}\grmot\mathrm{THH}(\mathrm{MU})/p \arrow{r} & \pi_{2p^k-1}\grmot\mathrm{THH}(\mathrm{MU}_{W(\field)})/p \,,
\end{tikzcd}
\]
and \autoref{THH-MU-lem}, which implies that the bottom arrow is injective, the result follows. 
\end{proof}

\begin{prop}\label{lambdas} 
There is a class 
$\lambda_{n+1}\in \pi_{2p^{n+1}-1}\grmot\TC(k(n)_{\field})/(p)$ in bidegree $(2p^{n+1}-1,1)$ in the image of the $\mathbb{E}_0$-algebra structure map 
\[\grmot\TC(\mathrm{MU}\otimes \mathbb{S}_{W(\field)})/p\to\grmot\TC(k(n)_{\field})/p\] 
and maps to the nontrivial class $\lambda_{n+1}\in \pi_{2p^{n+1}-1}\grmot\TC(k(n)_{\field})/(p,\cdots v_{n+1})$ from \autoref{AKHW} up to a unit in $\mathbb{F}_p$. 
\end{prop}

\begin{proof}
We know that $\lambda_{n+1}$ maps to $\lambda_{n+1}$ in the left vertical map of the commutative diagram
\[
\begin{tikzcd}
\pi_{2p^{n+1}-1}\grmot\mathrm{TC}(\mathrm{MU}_{W(\field)})/p\arrow{dd} \arrow{r} &  \pi_{2p^{n+1}-1}\grmot\mathrm{TC}(k(n)_{\field})/p \arrow{d} \\ 
 &  \pi_{2p^{n+1}-1}\grmot\mathrm{TC}(k(n)_{\field})/(p,\cdots ,v_{n+1}) \arrow{d} \\ 
\pi_{2p^{n+1}-1}\grmot\mathrm{THH}(\mathrm{MU}_{W(\field)})/(p,\cdots,v_{n+1}) \arrow{r} & \pi_{2p^{n+1}-1}\grmot\mathrm{THH}(k(n)_{\field})/(p,\cdots ,v_{n+1}) \,,
\end{tikzcd}
\]
by \autoref{lambda-class} and the fact that $\lambda_{n+1}$ is nontrivial in the image of the map 
\[ \pi_{2p^{n+1}-1}\grmot\mathrm{THH}(\mathrm{MU}_{W(\field)})/p\to \pi_{2p^{n+1}-1}\grmot\mathrm{THH}(\mathrm{MU}_{W(\field)})/(p,\cdots,v_{n+1})\]
by \autoref{THH-MU-lem} and \cite[Proposition~2.2.6]{angelini2024syntomic}\,. 

We also have a commutative diagram 
\[
\begin{tikzcd}
\pi_{2p^{n+1}-1}\grmot\mathrm{THH}(\mathrm{MU}_{p})/(p,\cdots ,v_{n+1})\arrow{d} \arrow{r} &  \pi_{2p^{n+1}-1}\grmot\mathrm{THH}(k(n))/(p,\cdots, v_{n+1}) \arrow{d} \\ 
\pi_{2p^{n+1}-1}\grmot\mathrm{THH}(\mathrm{MU}_{W(\field)})/(p,\cdots ,v_{n+1}) \arrow{r} & \pi_{2p^{n+1}-1}\grmot\mathrm{THH}(k(n)_{\field})/(p,\cdots ,v_{n+1}) \,,
\end{tikzcd}
\]
where the vertical maps are given by applying the natural transformation $-\otimes_{\fp} \fp \Rightarrow - \otimes_{\fp} \field$ to the top arrow in the diagram by \autoref{THH-MU-lem} and  \cite[Theorem~4.5.2]{angelini2024syntomic}. 
By \cite[Theorem~2.5.12]{angelini2024syntomic} the class $\lambda_{n+1}\in \pi_{2p^{n+1}-1}\grmot\mathrm{THH}(\mathrm{MU})/(p,\cdots ,v_{n+1})$ maps to the class $\lambda_{n+1}\in \pi_{2p^{n+1}-1}\grmot\mathrm{THH}(k(n))/(p,\cdots ,v_{n+1})$ and consequently the result follows from the fact that there is a unique class of bidegree $(2p^{n+1}-1,1)$ up to a unit in $\field_p$ in the mod $(p,\cdots,v_{n+1})$ syntomic cohomology of $k(n)_{\field}$. 
Here, note that that $v_{n+1}$-acts trivially on $\grmot\mathrm{THH}(k(n))$ and therefore 
\[\pi_*\grmot \mathrm{THH}(k(n))/(p,\cdots ,v_{n+1})\cong \pi_*\grmot \mathrm{THH}(k(n))/(p,\cdots ,v_{n})\langle \varepsilon_{n+1}\rangle  
\]
where $\|\varepsilon_{n+1}\|=(2p^{n+1}-1,-1)$, see the proof of \cite[Theorem~2.5.14]{angelini2024syntomic} for details. 
\end{proof}

\begin{lemm}\label{lem-partial}
Let $\mathbb{F}$ be finite. 
We have an inclusion 
$\field_p\subset \pi_{-1}\mathrm{gr}_{\textup{mot}}^0\mathrm{TC}(\mathrm{MU}_{W(\field_q)})/p$ such that the composite 
\[\field_p\subset \pi_{-1}\mathrm{gr}_{\textup{mot}}^0\mathrm{TC}(\mathrm{MU}_{W(\field)})/p\to \pi_{-1}\mathrm{gr}_{\textup{mot}}^0\mathrm{TC}(W(\field))/p\]
is an isomorphism. We call the generator of this sub-$\mathbb{F}_p$-vector space $\partial$. 
\end{lemm}
\begin{proof}
By~\cite[Proposition~3.22]{Kee25} for example, the map $\mathrm{TC}(\mathrm{MU}_{W(\field)})/p\to \mathrm{TC}(W(\field))/p$ is also an isomorphism in degrees $\le 1$. By Hesselholt~\cite[Theorem~A]{Hes97}, we know $\pi_{-1}\mathrm{TC}(W(\field))/p\cong W(\field)_{\mathrm{Frob}}/p$ and it is well-known that $\pi_{-s} \tc(W(\field))/p = 0$ for $s\le 2$. This implies that \[\pi_{-1}\mathrm{TC}(\mathrm{MU}_{W(\field)})/p=\mathbb{F}_p\]
since $W(\field)_{\mathrm{Frob}}/p=\mathbb{F}_{\mathrm{Frob}}$ and $\mathbb{F}_{\mathrm{Frob}}=\mathbb{F}_p$ when $\mathbb{F}$ is finite. By~\cite[Theorem~1.5]{LW22}, we know $\pi_{-1}(\grmot\mathrm{TC}(W(\field))/p)=\mathbb{F}_p$, for which we fix a generator and denote it  as $\partial$, concentrated in Adams weight $1$, i.e.\ in weight zero\footnote{The class $\partial$ is called $\lambda$ in their notation, but that conflicts with the other $\lambda$ classes that appear in our work, so we choose the more standard name}. This implies the desired result as $\pis \grmot \tc(R)/p$ is the first page of a spectral sequence computing $\tc(R)/p$ for $R \in \{\MU_{W(\field)}, W(\field)\}$. 

In more detail, let 
\[ 
\mathrm{Fil}_{\mot}^{f}\pi_s(\mathrm{TC}(R)/p) := \mathrm{im}\left (\pi_s\mathrm{fil}_{\textup{mot}}^f \mathrm{TC}(R)/p\to \pi_s\mathrm{TC}(R)/p\right ) \]
for $R\in \{ \mathrm{MU}_{W(\field)}, W(\field)\}$. 
Since $\pi_{-1}\grmot\mathrm{TC}(W(\field))$ is concentrated in weight $0$~\cite[Theorem~1.5]{LW22}, there is an isomorphism
\[ 
\mathrm{Fil}_{\textup{mot}}^0\pi_{-1}\mathrm{TC}(W(\field))/p\xrightarrow{\cong} \pi_{-1}\mathrm{gr}_{\textup{mot}}^0\mathrm{TC}(W(\field))/p\,.
\] 
We know that 
\[ 
\pi_{-1}\mathrm{TC}(\MU_{W(\field)})/p\cong \pi_{-1}\mathrm{TC}(W(\field))/p
\]
So we have a commutative diagram 
\[ 
\begin{tikzcd}
\Fil_{\mot}^1\pi_{-1}\mathrm{TC}(\mathrm{MU}_{W(\field)})/p   \arrow{d}  \arrow{r} & \Fil_{\mot}^0\pi_{-1}\mathrm{TC}(\mathrm{MU}_{W(\field)})/p=\pi_{-1}\mathrm{TC}(\mathrm{MU}_{W(\field)})/p \arrow{d}{\cong} \\ 
0 \arrow{r} & \Fil_{\mot}^0\pi_{-1}\mathrm{TC}(W(\field))/p=\pi_{-1}\mathrm{TC}(W(\field))/p
\end{tikzcd}
\] 
and therefore $\partial$ on the bottom right corner does not lift  to a class in higher weight. Therefore, we have shown that $\partial$ lifts to a a generator of a sub $\mathbb{F}_p$-vector space of $\pi_{-1}\mathrm{gr}_{\textup{mot}}^{0}\mathrm{TC}(\mathrm{MU}_{W(\field)})$.
\end{proof}

\begin{prop}\label{lambda-linearity}
Let  $\field$ be finite. The differentials in the $v_{i+1}$-Bockstein spectral sequence 
\[ \pi_*\grmot\TC(k(n)_{\field})/(p,\cdots ,v_{i+1})[v_{i+1}]\implies \pi_*\grmot\TC(k(n)_{\field})/(p,\cdots ,v_{i})\]
are $\partial$ and $\lambda_{n+1}$-linear for $i\ge 0$. 
\end{prop}
\begin{proof}
It suffices to show that $\partial$ and $\lambda_{n+1}$ are in a ring that acts on this spectral sequence. For $\partial$, this follows from \autoref{lem-partial}. 
For $\lambda_{n+1}$ this follows from \autoref{lambdas}. 
\end{proof}

\begin{rema}
In the $v_i$-Bockstein spectral sequences for syntomic cohomology, such as the ones in the proposition above, the classes $v_i$ are of Adams weight $0$, see \autoref{rema:vn}. 
\end{rema}

\begin{lemm}\label{vn-bockstein}
Let $\field$ be finite. 
The classes in
$\pi_{*}\grmot \mathrm{TC}(k(n)_{\field})/(p,\cdots,v_n)$ in Adams weight $2$ exactly consist of classes in $\mathbb{F}_p[v_{n+1}]\{\partial \lambda_{n+1}\}$. 
\end{lemm}
\begin{proof}
The classes $\partial \lambda_{n+1}v_{n+1}^k$ are cycles in the $v_{n+1}$-Bockstein spectral sequence because they are in Adams weight $2$, $\mathrm{E}_1$-page vanishes in Adams weights $\ge 3$ and the differentials increase Adams weight by $1$. 
It therefore suffices to show that they are not boundaries. 
The differentials are $v_{n+1}$-linear and are of the form $d_r(x)=v_{n+1}^ry$ where $x$ and $y$ are not $v_{n+1}$-multiples. Therefore, there cannot be $v_{n+1}$-Bockstein spectral sequence differentials hitting $\partial \lambda_{n+1}v_{n+1}^k$ for bidegree reasons. To see this, note that the highest degree class in Adams weight $1$ that is not a $v_{n+1}$-multiple in the $\mathrm{E}_1$-page of the $v_{n+1}$-Bockstein spectral sequence is $\partial \overline{\varepsilon}_n\lambda_{n+1}$ in degree $2p^{n}-2+2p^{n+1}-1$, which is strictly less than the degree of $\partial \lambda_{n+1}v_{n+1}^k$ for all $k>0$ since the degree of $\partial \lambda_{n+1}v_{n+1}^k$ is $2p^{n+1}-2+(2p^{n+1}-2)k$ for all primes $p$ and $n\ge 1$.  
\end{proof}

\begin{lemm}
Let $\field$ be finite. 
The classes $\overline{\varepsilon}_1^{a_1}\cdot \ldots \cdot \overline{\varepsilon}_{i-1}^{a_{i-1}}\overline{\varepsilon}_{i}v_i^{t}$ and $\partial\overline{\varepsilon}_1^{a_1}\cdot \ldots \cdot \overline{\varepsilon}_{i-1}^{a_{i-1}}\overline{\varepsilon}_{i}v_i^{t}$ survive the $v_j$-Bockstein spectral sequence  computing $\grmot \tc(k(n)_{\field})/(p,...,v_{j-1})$ for $i<j\le n+1$ where $a_k\in \{0,1\}$ for $1\le k\le i-1$. 
\end{lemm}
\begin{proof}
The differentials are $v_j$-linear and there is a Leibniz rule for differentials on classes of the form $\overline{\varepsilon}_1^{a_1}\cdot \ldots \overline{\varepsilon}_{i-1}^{a_{i-1}}\overline{\varepsilon}_{i}$ and  $\partial\overline{\varepsilon}_1^{a_1}\cdot \ldots \overline{\varepsilon}_{i-1}^{a_{i-1}}\overline{\varepsilon}_{i}$ by mapping to $\mathbb{F}$ and noting that the map to the $\mathbb{F}$ case is injective in the relevant bidegrees. Note that the differentials in the $v_j$ Bockstein spectral sequence only hits elements divisible by $v_j$. We then apply \autoref{syntomic field} to determine that there are no possible differentials on $v_j$-Bockstein spectral sequence differentials on $\varepsilon_s$ and for $1\le s\le i$ and $\partial$ for bidegree reasons as the lowest degree class in $\pis \grmot \tc(k(n)_{\field})/(p,\cdots ,v_j)$ is $\partial$ in bidegree $(-1,1)$.
Moreover, these classes cannot be hit by differentials since that would imply the same for the $v_j$-Bockstein spectral sequence computing $\grmot \tc(\field)/(p,...,v_{j-1})$ which we know  is not the case due to \autoref{syntomic field}.
\end{proof}

\begin{prop}\label{d1-differentials}
Let $\field$ be finite. 
Let $a_k\in \{0,1\}$ for $0\le k\le i-1$. There are differentials 
\[ 
d_1(\partial^{a_0}\overline{\varepsilon}_1^{a_1}\cdot \ldots \cdot \overline{\varepsilon}_{i-1}^{a_{i-1}}\overline{\varepsilon}_{i}v_i^{j})
=\partial^{a_0}\overline{\varepsilon}_1^{a_1}\cdot \ldots \cdot \overline{\varepsilon}_{i-1}^{a_{i-1}}v_i^{j+1}
\]
in the $v_i$-Bockstein spectral sequence
\begin{equation}\label{BSS-k(n)}
\pi_*\grmot\TC(k(n)_{\field})/(p,\cdots ,v_{i})[v_i]\implies 
\pi_*\grmot\TC(k(n)_{\field})/(p,\cdots ,v_{i-1})
\end{equation}
for each $1\le i\le n$. 
\end{prop}
\begin{proof}

We map to the Bockstein spectral sequence 
\begin{equation}\label{BSS-Fp}
\pi_*\grmot\TC(\field)/(p,\cdots ,v_{i})[v_i]\implies 
\pi_*\grmot\TC(\field)/(p,\cdots ,v_{i-1})
\end{equation}
where we have the differentials 
\[ d_1(\partial^{a_0}\overline{\varepsilon}_1^{a_1}\cdot \ldots \cdot \overline{\varepsilon}_{i-1}^{a_{i-1}}\overline{\varepsilon}_{i})
=\partial^{a_0}\overline{\varepsilon}_1^{a_1}\cdot \ldots \cdot \overline{\varepsilon}_{i-1}^{a_{i-1}}v_i
\]
for all $a_0,\cdots ,a_{i-1}\in \{0,1\}$ by \autoref{syntomic field} and note that in the relevant bidegrees the reduction map is injective by \autoref{AKHW}. 
\end{proof}

\begin{rema}
From the proposition above, we only use the fact that $\overline{\varepsilon}_i$ survives the $v_j$-Bockstein spectral sequence for $j>i$ and carries a differential $d_1(\partial \overline{\varepsilon}_iv_i^k) = \partial v_i^{k+1}$ for $k\geq 0$ on the $v_i$-Bockstein spectral sequence. 
\end{rema}

\begin{prop}\label{2-line}
Let $\field$ be finite. In Adams weight $2$ the classes of $\pi_*\grmot\TC(k(n)_{\field})/p$ are contained in 
$\field_p\{\partial \lambda_{n+1}\}\oplus \field_p[v_1,\cdots ,v_{n+1}]\{\partial\lambda_{n+1}v_{n+1}\}$\,. Here, note that $|\partial \lambda_{n+1}|=2p^{n+1}-2$ and $|\partial \lambda_{n+1}v_{n+1}|=4p^{n+1}-4$. 
\end{prop}
\begin{proof}
By \autoref{vn-bockstein}, we know that the $2$-line is contained in 
\[ \mathbb{F}_p\{\partial \lambda_{n+1}\}[v_1,\cdots ,v_{n+1}] \,.\]
Since the $d_1$-differential in the $v_i$-Bockstein spectral sequence is $v_i$-linear, it follows from \autoref{AKHW} as well as \autoref{lambda-linearity} and \autoref{d1-differentials} that there are differentials 
\[ d_1(\overline{\varepsilon}_i \partial \lambda_{n+1}  v_i^j) = \partial \lambda_{n+1} v_i^{j+1}  \,,
\]
for each $1\le i\le n$ however, we do not know if the differentials are $v_{n+1}$-linear. 
\end{proof}

\subsection{Cardinality of algebraic K-theory}
We are now prepared to prove \autoref{cardinality-k-theory} and \autoref{cardinality-k-theory-periodic}. 

We first argue that if we consider relative algebraic K-theory then the $p$-local statement captures all the necessary information in positive degrees. 
\begin{prop}\label{relative-p-local}
Let $\field_q$ be a finite field of characteristic $p$. The canonical map 
\[ \pi_s\mathrm{K}_{\textup{rel}}(k(n)_{\field_q})\to \pi_s\mathrm{K}(k(n)_{\field_q})_{(p)}
\]
is an isomorphism for all $s\ge 1$. 
\end{prop}
\begin{proof}
Consider the map of long exact sequences 
\[
\begin{tikzcd}
\cdots \arrow{r} & \pi_{s+1}\mathrm{K}(\field_q) \arrow{d} \arrow{r} & \pi_s\mathrm{K}_{\textup{rel}}(k(n)_{\field_q})\arrow{r} \arrow{d}&  \pi_s\mathrm{K}_{s}(k(n)_{\field_q}) \arrow{r} \arrow{d} & \pi_s\mathrm{K}(\field_q) \arrow{d}\arrow{r} & \cdots \\ 
\cdots \arrow{r} & 0 \arrow{r} & \pi_s\mathrm{K}_{\textup{rel}}(k(n)_{\field_q})_{(p)} \arrow{r}{\cong} & \pi_s\mathrm{K}_{s}(k(n)_{\field_q})_{(p)} \arrow{r} & 0 \arrow{r} & \cdots 
\end{tikzcd}
\] 
for $s\ge 1$ where we use the fact that  $\mathrm{K}(\field_q)_{(p)}\simeq \mathbb{Z}_{(p)}$ by \cite{Qui72}. Since 
\[\mathrm{K}_{\textup{rel}}(k(n)_{\field_q})\simeq \mathrm{TC}_{\textup{rel}}(k(n)_{\field_q})\]
by~\cite[Theorem~7.0.0.2]{DGM13} and $\mathrm{TC}_{\textup{rel}}(k(n)_{\field_q})$ is $p$-local, by \autoref{lemm thh is p complete when thh gr is}, we know that the localization map $\pi_s\mathrm{K}_{\textup{rel}}(k(n)_{\field_q})\to \pi_s\mathrm{K}_{\textup{rel}}(k(n)_{\field_q})_{(p)} $ is an isomorphism, proving the claim. 
\end{proof}

We now prove the $p$-local version of the theorem. 

\begin{thm}\label{p-local-thm}
Let $\mathbb{F}_q$ be a finite perfect field of characteristic $p$. In nonnegative degrees, there are isomorphisms 
\[ 
\mathrm{K}_{2r}(k(n)_{\field_q
})_{(p)}\cong \begin{cases} \mathbb{Z}_{(p)} & r=0 \,,\\ 
    0  &  p^{n+1}-1\nmid r\,, \ 0<r\leq 2p^{n+1}-2  \,,\\
          \mathbb{Z}/p^{m(r)} &  p^{n+1}-1\mid r\,, \ 0<r\le 2p^{n+1}-2  \,, \\
    0 & p-1 \nmid r\,, \ r> 2p^{n+1}-2 \,, \\
     A_{r}  &  p-1\mid r\,, \ r> 2p^{n+1}-2  \,,
     \end{cases}
\]
and we have identifications 
\[ 
|\mathrm{K}_{2r+1}(k(n)_{\field_q})_{(p)}|=\begin{cases}
         
            |\mathbb{W}_{\lfloor \frac{r}{p^{n}-1}\rfloor}(\field_q)| & p^{n+1}-1\nmid r \,, \ 0\le r\le 2p^{n+1}-2 \,,\\
               |\mathbb{W}_{\lfloor \frac{r}{p^{n}-1}\rfloor}(\field_q)|\cdot p^{m(r)} & p^{n+1}-1\mid r \,, \ 0\le r\le  2p^{n+1}-2 \,, \\
         |\mathbb{W}_{\lfloor \frac{r}{p^{n}-1}\rfloor}(\field_q)| & p-1\nmid r \,, \ r > 2p^{n+1}-2 \,,\\
    |\mathbb{W}_{\lfloor \frac{r}{p^{n}-1}\rfloor}(\field_q)|\cdot \lvert A_r\rvert 
    & p-1\mid r\,, \ r>2p^{n+1}-2 \,, 
\end{cases}
\]
of cardinalities in odd degrees where $\mathbb{W}_n$ denotes the truncated big Witt vectors of length $n$, $m(r)$ is a positive integer that we leave undetermined and $A_{r}$ is a finite $p$-group that we leave undetermined. 
\end{thm}
\begin{proof}
 Since $\pi_0k(n)_{\field_q}=\field_q$, we know that $\mathrm{K}_0(k(n)_{\field_q})_{(p)}\cong \mathrm{K}_0(\field_q)_{(p)}=\mathbb{Z}_{(p)}$.
 
Recall that due to \autoref{theo even trivial gives odd}, \autoref{relative-p-local} and~\cite[Theorem~7.0.0.2]{DGM13}, 
\[\pi_s\rtc(k(n)_{\field_q})\cong \pi_s\mathrm{K}_{\rel}(k(n)_{\field_q})\cong  \pi_s\mathrm{K}(k(n)_{\field_q})_{(p)}\] 
is finite for $s>0$.

We first compute the even degree homotopy groups of $\rtc(k(n)_{\field_q})$ and consequently even degree homotopy groups of $\mathrm{K}(k(n)_{\field_q})_{(p)}$. We first note that the motivic spectral sequence computing $\rtc(k(n)_\field)$ collapse due to \autoref{Adams weights} and  \autoref{defi motivic spectral sequence}.

By \autoref{Adams weights}, $\pi_*\grmot \rtc(k(n)_{\field_q})$ is concentrated in Adams weights $[1,2]$ and the map $\pis\grmot \rtc(k(n)_{\field_q})\to \pis\grmot \tc(k(n)_{\field_q})$ is an isomorphism in Adams weight $2$. Therefore, by construction of the motivic spectral sequence, the only possible contribution to $\pi_{2r}\rtc(k(n)_{\field_q})$ comes from the part of $\pi_{*}\grmot\tc(k(n)_{\field_q})$ in Adams weight two.  Consequently, by \autoref{2-line} we know that 
\[\pi_{2r}\rtc(k(n)_{\field_q})=0\]
if $p^{n+1}-1\nmid r$ and $r\leq 2p^{n+1}-2$ or $p-1\nmid r$ when $r> 2p^{n+1}-2$. 
Applying the $v_0$-Bockstein spectral sequence, \autoref{2-line} and the fact that $\field_{\mathrm{Frob}}=\mathbb{F}_p$ when $\field$ is a finite field, we know 
$\pi_{r(2p^{n+1}-2)}\rtc(k(n)_{\field_q})$ is equipped with a finite filtration with associated graded given by a finite $\fp[v_{0}]$-module on a single generator from which we deduce that it is given by $\z/p^{m(r)}$ for some positive integer $m(r)$ when $0<r\le 2$. In the case of $\pi_{2(p-1)r}\rtc(k(n)_{\field_q})$ for $(p-1)r> 2p^{n+1}-2$, we can make a similar argument except that we have not ruled out differentials
\[ d_{1}(v_{n+1}^{i_{n+1}}\partial \lambda_{n+1}\varepsilon_i)=v_{n+1}^{i_{n+1}}\partial \lambda_{n+1}v_i
\]
and we only know that abutment is a finite $p$-group in these degrees.

The cardinalities of the odd degrees homotopy groups of $\rtc(k(n)_{\field_q})$ and consequently $\mathrm{K}(k(n)_{\field_q})$ then follow from \autoref{theo even trivial gives odd}.
\end{proof}

We can now  prove the main theorem from this section. 
\mainthmcardinality*
\begin{proof}
Note that 
\[ 
\mathrm{K}_s(\mathbb{F}_q)=\begin{cases} \mathbb{Z} & s=0 \\
0 &  s=2k>0\\
\mathbb{Z}/q^{k}-1 & s=2k-1> 0 
\end{cases}
\]
by~\cite{Qui72}, so in particular, the groups vanish $p$-locally in positive degrees. Therefore, the boundary map $\partial$ in the fiber sequence 
\begin{equation}\label{K-theeory-finite-fields}
\mathrm{K}_{\textup{rel}}(k(n)_{\mathbb{F}_q})\to \mathrm{K}(k(n)_{\mathbb{F}_q}) \to \mathrm{K}(\mathbb{F}_q)\overset{\partial}{\to} \Sigma\mathrm{K}_{\textup{rel}}(k(n)_{\mathbb{F}_q})
\end{equation}
is zero in homotopy since the target is $p$-local by \autoref{relative-p-local}. The result then follows from \eqref{K-theeory-finite-fields}  and \autoref{p-local-thm}. 
\end{proof}

We thank Maxime Ramzi for suggesting the argument used in the following corollary.

\begin{coro}\label{cardinality-k-theory-periodic}
Suppose $\field_q$ is a finite field of characteristic $p$.
In nonnegative degrees, there are isomorphisms 
\[ 
\mathrm{K}_{2r}(K(n)_{\field_q
})\cong \begin{cases} \mathbb{Z} & r=0 \,, \\ 
    \mathbb{Z}/(q^r-1)  &  p^{n+1}-1\nmid r \,, \ 0<r\leq 2p^{n+1}-2 \,,\\
         \mathbb{Z}/(q^r-1) \times  \mathbb{Z}/p^{m(r)}   &  p^{n+1}-1\mid r\,,\ 0<r\le 2p^{n+1}-2 \,,\\
      \mathbb{Z}/(q^r-1) & p-1 \nmid r\,, \ r> 2p^{n+1}-2 \,, \\
     \mathbb{Z}/(q^{r}-1)\times A_{r}  &  p-1\mid r\,,  \ r> 2p^{n+1}-2 \,,
     \end{cases}
\]
there is an isomorphism $\mathrm{K}_1(K(n))\cong \mathbb{Z}\times \mathbb{Z}/(q-1)$ and we have identifications 
\[ 
|\mathrm{K}_{2r+1}(K(n)_{\field_q})|=\begin{cases}
  |\mathbb{W}_{\lfloor \frac{r}{p^{n}-1}\rfloor}(\field_q)|(q^{r+1}-1) & p^{n+1}-1\nmid r , \ 0 <r\le 2p^{n+1}-2 \,,\\
            |\mathbb{W}_{\lfloor \frac{r}{p^{n}-1}\rfloor}(\field_q)|(q^{r+1}-1)\cdot p^{m(r)}
    &  p^{n+1}-1\mid r , \ 0< r\le 2p^{n+1}-2 \,,\\
         |\mathbb{W}_{\lfloor \frac{r}{p^{n}-1}\rfloor}(\field_q)|(q^{r+1}-1) & p-1\nmid r , \ r>2p^{n+1}-2 \,,\\
    |\mathbb{W}_{\lfloor \frac{r}{p^{n}-1}\rfloor}(\field_q)|(q^{r+1}-1)\cdot \lvert A_r\rvert 
    & p-1\mid r, \ r>2p^{n+1}-2\,,
\end{cases}
\]
of cardinalities in odd degrees where $\mathbb{W}_n$ denotes the truncated big Witt vectors of length $n$, $m(r)$ is a positive integer that we leave undetermined and $A_{r}$ is a finite $p$-group that we leave undetermined. 
\end{coro}
\begin{proof}
Consider the localization sequence 
\[ \mathrm{K}(\mathbb{F}_q)\to \mathrm{K}(k(n)_{\field_q})\to \mathrm{K}(K(n)_{\field_q})
\]
from \cite{BM08,AGH19}. 
The map $\mathrm{K}(k(n)_{\field_q})\to  \mathrm{K}(\field_q)$ is a $\pi_s$ isomorphism for $s\le 2p^n-2$, so in particular it is nontrivial. 
However, the composite 
\[ \mathrm{K}(\field_q)\to \mathrm{K}(k(n)_{\field_q})\to  \mathrm{K}(\field_q)\]
of the reduction map with the transfer map is zero by \cite[Remark~4.10]{land2020purity} (cf.~\cite[p.~43]{AR09}). This implies that the left hand map above is trivial in homotopy since in homotopy, every element of the middle term above that is not $p$-power torsion maps to a nontrivial element by the second map above. 
Consequently, the long exact sequence induced by the localization sequence splits into short exact sequences 
\[ 0\to \pi_*\mathrm{K}(k(n)_{\field_q})\to \pis \mathrm{K}(K(n)_{\field_q})\to \pi_{*-1}\mathrm{K}(\mathbb{F}_q)\to  0,.
\]
and the rest follows from \autoref{cardinality-k-theory}  and~\cite{Qui72}. Note that when $*=1$, the short exact sequence 
\[ 0\to \mathbb{Z}/(q-1)\to \mathrm{K}_1(K(n)_{\field_q})\to \mathbb{Z} \to 0\]
splits. 
\end{proof}

\section{On the \texorpdfstring{$K(1)$}{K(1)}-local algebraic K-theory of \texorpdfstring{$K(1)$}{K(1)}}\label{K1local}
Here we record the fact that the $K(1)$-local algebraic K-theory groups of $K(1)$ are concentrated in even degrees. This is a direct consequence of \cite{angeltveit2008thhofainftyringspectra} and \cite{DR25}. Here we use the the forms of $K(1)$ considered by Angeltveit in~\cite[Theorem~5.17]{angeltveit2008thhofainftyringspectra}, though we expect that the result holds more generally. 
\begin{prop}\label{even K(1)-local}
There are isomorphisms
\[  \pi_{2m+1}L_{K(1)}\mathrm{K}(k(1))\cong \pi_{2m+1}L_{K(1)}\mathrm{K}(K(1))\cong 0\]
for all integers $m$. Moreover, there are isomorphisms 
\[  
\pi_{2m}L_{K(1)}\mathrm{K}(k(1)) \cong \pi_{2m}L_{K(1)}\mathrm{K}(K(1)) \cong \left (\bigoplus_{i=0}^{\infty}\mathbb{Z}_p \right )_p
\] 
for all integers $m$. Here, the nonderived $p$-completion agrees with the derived $p$-completion.  
\end{prop}
\begin{proof}
By~\cite[Theorem~5.16,~pp.~1025--1026]{angeltveit2008thhofainftyringspectra}, we know that 
\[ 
\THH_*(K(1))=\THH_*^{L_p}(K(1))=\bigoplus_{i=0}^{p-2}\Sigma^{2i}\mathbb{Q}_p/\mathbb{Z}_p[v_1,v_1^{-1}]
\]
which is concentrated in even degrees, where $L_p$ is the $p$-complete periodic Adam summand. Therefore, the homotopy orbit spectral sequence collapses so the associated graded of a filtration on $\pis (\THH(K(1))_{hS^1})$ can be identified with 
\[ \bigoplus_{i=0}^{p-2}\Sigma^{2i}\mathbb{Q}_p/\mathbb{Z}_p[v_1,v_1^{-1},z]
\]
where $|z|=2$ and since $\mathbb{Q}_p/\mathbb{Z}_p$ is injective there are no additive extensions. We conclude that
\[ 
\pi_{2k}(\Sigma(\THH(K(1))_{hS^1})=0
\]
and 
\[ 
\pi_{2k+1}(\Sigma(\THH(K(1))_{hS^1})= \bigoplus_{i=0}^{\infty} \Q_p/\mathbb{Z}_p . 
\]
To compute the homotopy groups of the $p$-completion, we start by tensoring with $\mathbb{S}/p^n$ to determine that 
\[ 
\pi_{2k}(\Sigma(\THH(K(1))_{hS^1}/p^n)=\bigoplus_{i=0}^{\infty} \mathbb{Z}/p^n . 
\]
and 
\[ 
\pi_{2k+1}((\Sigma(\THH(K(1))_{hS^1})/p^n)=0
\]
from the long exact sequence associated to the cofiber sequence $\mathbb{S}\to\mathbb{S}\to \mathbb{S}/p^n$. We then conclude that  \[ 
\pi_{2k}(\Sigma(\THH(K(1))_{hS^1})_p=\lim_n \bigoplus_{i=0}^{\infty} \mathbb{Z}/p^n = \left ( \bigoplus_{i=0}^{\infty} \mathbb{Z}_p \right )_p 
\]
and
\[ 
\pi_{2k-1}(\Sigma(\THH(K(1))_{hS^1})_p=0
\]
using a Milnor sequence since the sequence is Mittag--Leffler. 

By~\cite[Theorem~0.4.2]{DR25}, we determine the groups $\pi_*L_{K(1)}\mathrm{K}(k(1))$. By the localization sequence 
\[
\mathrm{K}(\mathbb{F}_p)\to \mathrm{K}(k(1)) \to \mathrm{K}(K(1)) \]
of~\cite{BM08,AGH19} and the $K(1)$-local vanishing of $\mathrm{K}(\mathbb{F}_p)$ of~\cite{Qui72}, we know 
\[
\pi_*L_{K(1)}\mathrm{K}(k(1))\cong \pi_*L_{K(1)}\mathrm{K}(K(1)) \,. \qedhere
\] 
\end{proof}
\begin{rem}
\autoref{even K(1)-local} does not appear to improve on or contradict our computations.  Neverless, we record this fact for completeness. 
\end{rem}

\begin{rem}
Combining this result with \cite[Corollary~D~(a)]{land2020purity} and \cite[Theorem~4.4.6]{angelini2024syntomic} this shows that $L_{T(m)}K(k(1))=0$ if and only if $m\ne 0,1,2$. This answers a question of Markus Land about whether $L_{T(n)}\mathrm{K}(k(n))$ is nontrivial, in the case $n=1$. This also provides another example of the sharpness of the purity theorem of~\cite{land2020purity}. 
\end{rem}

{\providecommand{\bysame}{\leavevmode\hbox to3em{\hrulefill}\thinspace}
\providecommand{\MR}{\relax\ifhmode\unskip\space\fi MR }
\providecommand{\MRhref}[2]{%
  \href{http://www.ams.org/mathscinet-getitem?mr=#1}{#2}
}
\providecommand{\href}[2]{#2}

}
\end{document}